\documentclass[12pt]{amsart}
\usepackage{amsmath}
\usepackage{amsfonts}
\usepackage{amssymb}
\usepackage{amscd}
\usepackage{mathtools}
\usepackage{amsxtra}
\usepackage{amsthm}
\usepackage{graphicx}
\usepackage[abbrev,alphabetic]{amsrefs}
\RequirePackage[dvipsnames,usenames]{color}
\usepackage{soul,xcolor}
\setstcolor{red}
\usepackage{stmaryrd}
\usepackage{booktabs}
\usepackage{multirow}
\newtagform{tiny}{\tiny(}{)}
\usepackage{aliascnt}

\usepackage{mathtools}
\usepackage{hyperref}
\usepackage[margin=1.25in]{geometry}
\usepackage{mathrsfs}

\usepackage{amsthm}
\usepackage{comment}
\usepackage[all,cmtip]{xy}
\usepackage{tikz-cd}
\usetikzlibrary{cd}

\tikzcdset{
  cells={font=\everymath\expandafter{\the\everymath\displaystyle}},
}

\usepackage[all]{xy}

\usepackage{cleveref}

\makeatletter
\def\@tocline#1#2#3#4#5#6#7{\relax
  \ifnum #1>\c@tocdepth 
  \else
    \par \addpenalty\@secpenalty\addvspace{#2}%
    \begingroup \hyphenpenalty\@M
    \@ifempty{#4}{%
      \@tempdima\csname r@tocindent\number#1\endcsname\relax
    }{%
      \@tempdima#4\relax
    }%
    \parindent\z@ \leftskip#3\relax \advance\leftskip\@tempdima\relax
    \rightskip\@pnumwidth plus4em \parfillskip-\@pnumwidth
    #5\leavevmode\hskip-\@tempdima
      \ifcase #1
       \or\or \hskip 1em \or \hskip 2em \else \hskip 3em \fi%
      #6\nobreak\relax
    \hfill\hbox to\@pnumwidth{\@tocpagenum{#7}}\par
    \nobreak
    \endgroup
  \fi}
\makeatother

\newcommand{\cO}{\mathcal{O}}

\newcommand{\red}{\mathrm{red}}
\newcommand{\perf}{\mathrm{perf}}
\newcommand{\Frac}{\mathrm{Frac}}
\newcommand{\Proj}{\mathrm{Proj}}

\makeatletter
\DeclareRobustCommand{\graded}{%
    \@ifnextchar\bgroup{\graded@with}{\ensuremath{\mathrm{gr}}}%
}
\newcommand{\graded@with}[1]{\ensuremath{#1\mathrm{-gr}}}
\makeatother

\DeclareMathOperator*{\grlim}{grlim}
\DeclareMathOperator*{\grcolim}{grcolim}
\newcommand{\grotimes}{\mathbin{\otimes}^{\graded}}
\newcommand{\Lgrotimes}{\mathbin{\otimes}^{L\graded}}
\newcommand{\cgrotimes}{\mathbin{\widehat{\otimes}}^{\graded}}
\newcommand{\cLgrotimes}{\mathbin{\widehat{\otimes}}^{L\graded}}

\DeclareMathOperator{\tperfd}{\widehat{\perfd}}
\DeclareMathOperator{\grpfd}{\mathrm{grpfd}}
\DeclareMathOperator{\tgrpfd}{\widehat{\grpfd}}
\DeclareMathOperator{\pfd}{pfd}

\makeatletter
\DeclareRobustCommand{\comp}[1]{%
    \@ifnextchar\bgroup{\comp@with{#1}}{\comp@without{#1}}%
}
\newcommand{\comp@with}[2]{\Lambda_{#1}(#2)}
\newcommand{\comp@without}[1]{\ensuremath{#1\mathrm{-comp}}}
\makeatother
\newcommand{\dcomp}[2]{d\Lambda_{#1}(#2)}
\newcommand{\grcomp}[2]{\Lambda^{\graded}_{#1}(#2)}
\newcommand{\dgrcomp}[2]{d\Lambda^{\graded}_{#1}(#2)}
\DeclareMathOperator{\grwedge}{{\wedge \graded}}

\DeclareMathOperator{\coMod}{coMod}

\DeclareMathOperator{\Supp}{Supp}
\DeclareMathOperator{\Spec}{Spec}

\DeclareMathOperator{\Hom}{Hom}

\DeclareMathOperator{\Pro}{Pro}

\DeclareMathOperator{\SPerfd}{SPerfd}

\newcommand{\hProj}{\widehat{\Proj}}

\theoremstyle{plain}
\newtheorem{theorem}{Theorem}[section]

\newtheorem{theoremA}{Theorem}

\crefname{theoremA}{Theorem}{Theorems}
\Crefname{theoremA}{Theorem}{Theorems}

\newaliascnt{lemma}{theorem}
\newtheorem{lemma}[lemma]{Lemma}
\aliascntresetthe{lemma}
\crefname{lemma}{Lemma}{Lemmas}
\Crefname{lemma}{Lemma}{Lemmas}

\newaliascnt{proposition}{theorem}
\newtheorem{proposition}[proposition]{Proposition}
\aliascntresetthe{proposition}
\crefname{proposition}{Proposition}{Propositions}
\Crefname{proposition}{Proposition}{Propositions}

\newaliascnt{corollary}{theorem}
\newtheorem{corollary}[corollary]{Corollary}
\aliascntresetthe{corollary}
\crefname{corollary}{Corollary}{Corollaries}
\Crefname{corollary}{Corollary}{Corollaries}

\newaliascnt{claim}{theorem}

\aliascntresetthe{claim}
\crefname{claim}{Claim}{Claims}
\Crefname{claim}{Claim}{Claims}

 \newtheorem*{claim*}{Claim}

\theoremstyle{definition}
\newaliascnt{definition}{theorem}
\newtheorem{definition}[definition]{Definition}
\aliascntresetthe{definition}
\crefname{definition}{Definition}{Definitions}
\Crefname{definition}{Definition}{Definitions}

\newaliascnt{notation}{theorem}
\newtheorem{notation}[notation]{Notation}
\aliascntresetthe{notation}
\crefname{notation}{Notation}{Notations}
\Crefname{notation}{Notation}{Notations}

\newaliascnt{example}{theorem}

\aliascntresetthe{example}
\crefname{example}{Example}{Examples}
\Crefname{example}{Example}{Examples}

\theoremstyle{remark}
\newaliascnt{remark}{theorem}
\newtheorem{remark}[remark]{Remark}
\aliascntresetthe{remark}
\crefname{remark}{Remark}{Remarks}
\Crefname{remark}{Remark}{Remarks}

\newaliascnt{setting}{theorem}

\aliascntresetthe{setting}
\crefname{setting}{Setting}{Settings}
\Crefname{setting}{Setting}{Settings}

\newaliascnt{construction}{theorem}
\newtheorem{construction}[construction]{Construction}
\aliascntresetthe{construction}
\crefname{construction}{Construction}{Constructions}
\Crefname{construction}{Construction}{Constructions}

\numberwithin{equation}{section}



\usepackage{mymacros_common_Ryo, mymacros_english_paper_Ryo}

\title{Algebraization of Absolute Perfectoidization via Section Rings}
\author{Ryo Ishizuka}
\address{Institute of Science Tokyo, Tokyo 152-8551, Japan}
\email{ishizuka.r.ac@m.titech.ac.jp}

\author{Shou Yoshikawa}
\address{Institute of Science Tokyo, Tokyo 152-8551, Japan}
\email{yoshikawa.s.9fe9@m.isct.ac.jp}



\begin{document}

\begin{abstract}
We construct and study a graded version of absolute perfectoidization for $G$-graded adic rings.
As a main geometric application, we show that the absolute perfectoidization of the structure sheaf of a projective-type formal scheme admits an algebraization.
\end{abstract}

\maketitle 

\tableofcontents

\section{Introduction}

The theory of perfectoid rings, originating in the work of Scholze \cite{scholze2012Perfectoida}, has become a fundamental tool in arithmetic geometry, algebraic geometry, and commutative algebra.
Within this development, the notion of \emph{(absolute) perfectoidization} was introduced in the prismatic theory of Bhatt--Scholze \cite{bhatt2022Prismsa}. This construction provides a universal way to attach to any \(p\)-adically complete ring \(R\) an ``initial perfectoid ring'' \(R_{\perfd}\), as a mixed-characteristic counterpart of perfection in characteristic \(p\).
Absolute perfectoidization has already been used in several contexts, especially in the study of mixed-characteristic singularities. 
For instance, it yields invariants such as perfectoid signatures, perfectoid Hilbert--Kunz multiplicities \cite{cai2025Perfectoida}, and centers of perfectoid purity \cite{fayolle2025Centersa}. These are perfectoid analogues of the $F$-signature, Hilbert--Kunz multiplicity, and center of \(F\)-purity, respectively.
Moreover, Bhatt--Ma--Patakfalvi--Schwede--Tucker--Waldron--Witaszek \cite{bhatt2024Perfectoid} introduced \emph{(lim-)perfectoid pure} singularities, which form a perfectoid analogue of \(F\)-pure singularities in mixed characteristic.


In \cite{ishizuka2025Graded}, we developed the basic theory of \emph{graded perfectoid rings}, providing a foundation for perfectoid structures in the graded setting.
Indeed, graded rings appear as section rings in algebraic geometry and as basic objects in commutative algebra and representation theory.
For this reason, it is natural to ask how perfectoid methods extend to this setting. 
In this paper, we extend the theory of absolute perfectoidization to the graded setting.



\begin{theoremA}[{\Cref{SectionTopGradedPerfd}}] \label{MainTheoremGradedPerfd}
    Let $G$ be a torsion-free abelian group with \(G = G[1/p]\) and let $p$ be a prime number.
    Let \(R\) be a \(G\)-graded ring\footnote{Even if \(R\) is just a \(\setZ\)-graded ring, we can extend the grading to a \(\setZ[1/p]\)-grading by setting \(R_g = 0\) for \(g \in \setZ[1/p] \setminus \setZ\). Therefore, the assumption \(G = G[1/p]\) does not cause any loss of generality. See also \Cref{RemarkGrading}.} whose graded components \(R_g\) are \(p\)-adically complete for all \(g \in G\). Then there exists a \(G\)-graded \(\setE_{\infty}\)-\(R\)-algebra
    \begin{equation*}
        R_{\grpfd} \in \CAlg(\mcalD_{\graded{G}}(R))
    \end{equation*}
    called the \emph{absolute graded perfectoidization} of \(R\), satisfying the following properties:
    \begin{enumerate}
        \item \(R_{\grpfd}\) is the limit of all \(G\)-graded perfectoid \(R\)-algebras in \(\CAlg(\mcalD_{\graded{G}}(R))\) (\Cref{DefGradedAbsPerfd}).
        \item The derived \(p\)-completion of \(R_{\grpfd}\) coincides with the \emph{absolute perfectoidization} \(R_{\perfd}\) of \(R\) introduced by Bhatt--Scholze \cites{bhatt2022Prismsa,bhatt2024Perfectoid} (\Cref{CompletionOfAbsoluteGradedPerfd}).
        \item If \(R\) admits a morphism from a perfectoid ring and \(R_{\perfd}\) is concentrated in degree \(0\), then \(R_{\grpfd}\) is the initial \(G\)-graded perfectoid \(R\)-algebra (\Cref{DiscretePropertyGradedPerfd}).
    \end{enumerate}
    Moreover, it satisfies further properties analogous to those of the absolute perfectoidization \(R_{\perfd}\) of $R$, such as colimit preservation (\Cref{ColimitPreserveGradedPerfd}), the base change property (\Cref{BaseChangeGradedPerfd}), and the descendability property (\Cref{GradedPerfdDescendable}).
\end{theoremA}

Although part of the argument reduces to the ungraded case, genuinely new input is required to control the grading. In particular, the proof of the second assertion uses our previous results on derived graded modules from \cite{ishizuka2026Derived}.

We also establish the absolute perfectoidization of structure sheaves on bounded \(p\)-adic formal schemes. More precisely, for any such formal scheme \(\mscrX\), we introduce a quasi-coherent \(\setE_{\infty}\)-\(\mcalO_{\mscrX}\)-algebra
\begin{equation*}
    \mcalO_{\mscrX, \perfd} \in \CAlg(\mcalD_{\qcoh}(\mscrX))
\end{equation*}
called the \emph{absolute perfectoidization} of \(\mcalO_{\mscrX}\) (\Cref{DerivedPushoutArcCohomology}). This object is constructed via \(p\)-adic arc cohomology and is identified with the limit of derived pushforwards of structure sheaves of perfectoid formal schemes over \(\mscrX\) (\Cref{GlobalPerfectoidizationSectionLimit}).

Recently, Bhatt introduced the \emph{perfectization} \(\mscrX^{\pfd}\), a stack on \(p\)-nilpotent rings \cite{bhatt2025Aspects}. This gives a geometrization of the above constructions: \(\mcalO_{\mscrX, \perfd}\) agrees with the pushforward of the structure sheaf of \(\mscrX^{\pfd}\) (\Cref{ComparisonOXperfd}) and \(R_{\grpfd}\) for a \(G\)-graded ring \(R\) agrees with the global sections of the structure sheaf of \((\Spf(R)/\mfrakD)^{\pfd}\), where \(\mfrakD = \Spf(\setZ_p[G])\) is the formal group scheme associated with \(G\) admitting an action on \(\Spf(R)\) corresponding to the grading (\Cref{StackyGradedPerfd}).

On the other hand, both the arc-topological and perfectization approaches directly produce perfectoidization only for \emph{formal} schemes, not for projective varieties themselves.
Combining the results above, we overcome this limitation: the absolute graded perfectoidization of the section ring of a projective variety gives an algebraization of  the absolute perfectoidization of the structure sheaf of its \(p\)-adic formal completion. This is our second main result.

\begin{theoremA}[{\Cref{AlgebraizableOXperfd}}] \label{MainTheoremAlgebraization}
    Let \(X\) be a quasi-compact scheme whose canonical morphism \(X \to \Spec(H^0(X, \mcalO_{X}))\) is universally closed.
    Assume that \(X\) admits an ample line bundle and that \(H^0(X, \mcalO_X)\) has bounded \(p^\infty\)-torsion.
    Then there exists a quasi-coherent \(\setE_{\infty}\)-\(\mcalO_X\)-algebra
    \begin{equation*}
        \mcalO_{X, \perfd} \in \CAlg(\mcalD_{\qcoh}(X))
    \end{equation*}
    with an isomorphism
    \begin{equation*}
        \lim_{n \geq 1} (\mcalO_{X, \perfd} \otimes^L_{\mcalO_X} \mcalO_X/p^n\mcalO_X) \cong \mcalO_{\widehat{X}, \perfd}
    \end{equation*}
    in \(\CAlg(\mcalD_{\qcoh}(\widehat{X}))\), where \(\widehat{X}\) is the \(p\)-adic formal completion of \(X\).

    More precisely, after fixing an ample line bundle \(L\), its inverse \(L^{-1}\), and an isomorphism \(L \otimes_{\mcalO_X} L^{-1} \xrightarrow{\cong} \mcalO_X\), the object \(\mcalO_{X,\perfd}\) may be chosen to be
    \[
     \mcalO_{X,\perfd} = \widetilde{R_{\grpfd}},
     \]
    where \(R\) is the section ring \(R \defeq \bigoplus_{n \geq 0} H^0(X, L^{\otimes n})\) and \(\widetilde{(-)}\) is the functor \(\widetilde{(-)} \colon \mcalD_{\graded{\setZ}}(R) \to \mcalD_{\qcoh}(X)\).
\end{theoremA}

As a corollary of \Cref{MainTheoremAlgebraization},  the absolute perfectoidization of the structure sheaf of quasi-projective formal scheme (\Cref{QuasiProjCase}) is algebraizable.
We also establish a coherence of the (algebraic) localization \(\mcalO_{X, \perfd}[1/p]\) (\Cref{CoherenceAlgebraicLocalization}) and an identification of \(\mcalO_{\widehat{X}, \perfd}\) with the limit of all quasi-coherent sheaves associated with \(G\)-graded perfectoid \(R\)-algebras (\Cref{PerfdProjLimit}).

In a subsequent paper \cite{ishizuka2026Localglobal}, we will apply absolute graded perfectoidization to global singularities in mixed characteristic and establish a local-global comparison for perfectoid pure singularities on projective varieties, as a perfectoid analogue of the corresponding local-global comparison for \(F\)-pure singularities in characteristic \(p\).

For the applications in our paper \cite{ishizuka2026Localglobal}, we only need the $p$-adic results stated above.
However, since the present paper is devoted to developing the foundational theory, we construct the absolute graded perfectoidization for more general $G$-graded adic rings and adic formal schemes, and prove analogous properties in this generality.

\subsection*{Structure of this paper}
In \Cref{SectionPreliminaries}, we introduce notation and terminology. In \Cref{SectionTopGradedPerfd}, we construct absolute graded perfectoidization for \(G\)-graded rings and prove the properties in \Cref{MainTheoremGradedPerfd}. In \Cref{SectionGlobalAbsPerfd}, we construct the global variant for structure sheaves on adic formal schemes, establish representation and comparison results, and prove the algebraization statement in \Cref{MainTheoremAlgebraization}. For completeness, we also include appendices on \(\Proj\) for \(\setQ^n\)-graded rings, on perfectoid and prismatic formal schemes, and on quasi-coherent complexes on formal schemes.

\subsection*{Acknowledgments}
The authors would like to thank Jack J. Garzella and Joe Waldron for valuable discussions on global variants of perfectoid theory, and L\'eo Navarro Chafloque and Kanau Shimada for helpful discussions.
This work was started while we attended the conference `\(p\)-adic and Characteristic \(p\) Methods in Algebraic Geometry' at EPFL and `the Summer Research Institute in Algebraic Geometry' at Colorado State University in 2025. We are very grateful for these opportunities and their hospitality.

The first-named author was supported by JSPS KAKENHI Grant number 24KJ1085.
The second-named author was supported by JSPS KAKENHI Grant number JP24K16889.

\section{Preliminaries} \label{SectionPreliminaries}

\subsection{Notation and terminology}

Throughout this paper, we use the following notation and terminology.

\subsubsection{General notation}
\begin{enumerate}
    \item We fix a prime number \(p\).
    \item We only consider commutative rings with unity, which we simply refer to as \emph{rings}. The category of (commutative) rings is denoted by \(\CRing\) (not \(\CAlg\)).
    \item Let $R$ be a ring, $M$ an $R$-module, and $I$ an ideal of $R$.
    We define the \emph{$I$-torsion part of $M$} by
    \begin{equation*}
        M[I] \defeq \{m \in M \mid Im=0\}.
    \end{equation*}
\end{enumerate}

\subsubsection{Higher categorical stuff}
\begin{enumalphp}
    \item We use the notion of an \(\infty\)-category (more precisely, an \((\infty,1)\)-category). Our main resources are \cites{lurie2017Higher, lurie2018Spectral}.
    \item The \(\infty\)-category of \(\setE_{\infty}\)-ring spectra is denoted by \(\CAlg\), which admits all small limits and colimits. We simply call an object of \(\CAlg\) an \emph{\(\setE_{\infty}\)-ring}. If we take the slice category over an \(\setE_{\infty}\)-ring \(R\), we denote the \(\infty\)-category of \(\setE_{\infty}\)-\(R\)-algebras as \(\CAlg_R\). In this paper, we do not use the notion of derived rings, introduced in \cite{raksit2026Hochschild} following Akhil Mathew, in place of \(\setE_{\infty}\)-rings, although similar arguments should also work in that framework.
    \item If we say \emph{discrete} rings and modules, we mean commutative rings and modules without higher homotopy groups. Since we rarely use topological rings equipped with the discrete topology, and make this explicit when we do, we hope that no confusion will arise.
\end{enumalphp}

\subsubsection{Graded rings} \label{GradedNotation}
\begin{enumalphp}
    \item Let \(G\) be a torsion-free abelian group with identity element \(0\).
    \item A \emph{\(G\)-graded ring} is a pair of a commutative ring \(R\) and additive subgroups \(\{R_g\}_{g \in G}\) such that \(R\) admits a decomposition \(R = \bigoplus_{g \in G} R_g\) and satisfies \(R_g R_{g'} \subseteq R_{g + g'}\) for any \(g, g' \in G\). A \emph{homogeneous element} of \(R\) is an element in \(R_g\) for some \(g \in G\) and a \emph{homogeneous ideal} is an ideal of \(R\) which is generated by homogeneous elements.
    The subset $\{g \in G \mid R_g \neq 0\}$ of $G$ is called the \emph{support of $R$} and is denoted by $\Supp(R)$.
    Furthermore, for a submonoid $H$ of $G$, if the support of $R$ is contained in $H$, then we say that $R$ is an \emph{$H$-graded ring}.
    \item A \emph{\(G\)-graded ring morphism} between \(G\)-graded rings \(R = \bigoplus_{g \in G} R_g\) and \(S = \bigoplus_{g \in G} S_g\) is a ring homomorphism \(\varphi \colon R \to S\) such that \(\varphi(R_g) \subseteq S_g\) for any \(g \in G\).
    \item For a (discrete) \(G\)-graded ring \(R\), we will use the \(\infty\)-category \(\mcalD_{\graded{G}}(R)\) of graded \(R\)-modules defined in \cite{ishizuka2026Derived}*{Section 3}.
    \item In \(\mcalD_{\graded{G}}(R)\) of a \(G\)-graded ring \(R\), the (co)limits and symmetric monoidal products are denoted by \(\grlim\), \(\grcolim\), and \(- \Lgrotimes -\) (if they exist), but in \(\mcalD(R)\) of a discrete ring \(R\), i.e., derived limits and derived tensor products, they are often denoted by \(\lim\) (or \(R\lim\)), \(\colim\), and \(- \otimes^L -\).
    \item A \emph{\(G\)-graded adic ring} is a \(G\)-graded ring \(R\) with an adic topology for a finitely generated homogeneous ideal \(I\). We say that such ideal \(I\) is a \emph{homogeneous ideal of definition} of \(R\).
    \item A \(G\)-graded adic ring \(R\) is said to be \emph{gradedwise complete} if the canonical morphism \(R \to \bigoplus_{g \in G} \lim_{n \geq 1}(R/I^n)_g\) is an isomorphism of \(G\)-graded rings. See more comprehensively \cites{ishizuka2025Graded, ishizuka2026Derived}.
\end{enumalphp}

\subsubsection{Topological rings}
\begin{enumalphp}
    \item A topological ring is said to have a linear topology if there exists a fundamental system of neighborhoods of \(0\) consisting of ideals. A linear topological ring \(R\) is called \emph{complete} if any Cauchy net in \(R\) uniquely converges. If there exists the universal complete topological \(R\)-algebra, then it is called the \emph{completion} of \(R\).
    \item If a topological ring \(R\) has an ideal \(I\) such that \(I\) is finitely generated and the set \(\{I^n\}_{n \geq 0}\) consists a fundamental system of neighborhoods of \(0\), then we call \(R\) an \emph{adic ring} and \(I\) an \emph{ideal of definition} of \(R\).
    \item Morphisms between adic rings are assumed to be continuous ring homomorphisms.
\end{enumalphp}

\subsubsection{\texorpdfstring{\(p\)-adic stuff}{p-adic stuff}}
\begin{enumalphp}
    \item We freely use the notion of perfectoid rings, prisms, and related concepts, e.g., perfectoidization, following \cites{bhatt2018Integral, bhatt2022Prismsa}.
    \item An adic ring \(R\) is called an \emph{adic perfectoid ring} if \(p\) is a topologically nilpotent element of \(R\) and the underlying ring \(R\) is a perfectoid ring (\cite{takaya2026Relative}*{Definition 4.4}).
    \item For a base adic ring \(A\) with an ideal of definition \(I\), we denote by \(\Perfd_A^{\wedge I}\) the category of \(I\)-adically complete perfectoid \(A\)-algebras.
    \item The \emph{\(p\)-root closure} \(C(R)\) of a \(p\)-torsion-free ring \(R\) is the subset \(\set{x \in R[1/p]}{\exists n \in \setZ_{\geq 1}, \ x^{p^n} \in R}\) which becomes a subring of \(R[1/p]\) following \cite{roberts2008Root}. See \cite{ishizuka2024Calculationa} for some history of root closure and its relationship to perfectoid theory.
\end{enumalphp}

\subsubsection{Formal schemes}
\begin{enumalphp}
    \item A \emph{formal scheme} \(\mscrX\) is a topologically ringed space \(\mscrX = (\abs{\mscrX}, \mcalO_{\mscrX})\) such that it is Zariski locally isomorphic to the affine formal spectrum \(\Spf(R)\) of a linear topological ring \(R\). Our definition is based on \cite{grothendieck1971Elements}*{Ch.I \S 10}.
    \item A formal scheme \(\mscrX\) is said to be \emph{adic}, \emph{quasi-compact}, or \emph{separated} if it is Zariski locally isomorphic to the affine formal spectrum \(\Spf(R)\) of an adic ring \(R\), if \(\mscrX\) is quasi-compact as a topological space, or if any intersection of two affine open subsets of \(\mscrX\) is affine, respectively.
    \item For an adic quasi-compact formal scheme \(\mscrX\), a finitely generated sheaf of ideals \(\mcalI \subseteq \mcalO_{\mscrX}\) is called an \emph{ideal of definition} of \(\mscrX\) if \(\{\mcalI(\mscrU)^n\}_{n \geq 0}\) forms a fundamental system of neighborhoods of \(0\) in \(\mcalO_{\mscrX}(\mscrU)\) for any affine open subset \(\mscrU\).
    \item Given a scheme \(X\) and a closed subscheme \(Z\) of \(X\), we can construct the \emph{formal completion \(\widehat{X}\) of \(X\) along \(Z\)}, which is a formal scheme such that its underlying topological space is \(Z\) and the structure sheaf is the pullback of the sheaf \(\lim_{n \geq 1} \mcalO_X/\mcalI_Z^n\) on \(X\) to \(Z\), where \(\mcalI_Z\) is the quasi-coherent sheaf of ideals of \(\mcalO_X\) defining \(Z\).
    \item A \emph{(formal) \(\delta\)-scheme} is a pair \((X, \delta)\) of a (formal) scheme \(X\) and a map \(\delta \colon \mcalO_X \to \mcalO_X\) of sheaves of sets such that it gives Zariski locally a \(\delta\)-structure on \(\mcalO_X\). See also \cite{bhatt2022Prismatization}*{Remark 4.1}.
\end{enumalphp}
    
\subsubsection{Objects in spectral algebraic geometry}
\begin{enumalphp}
    \item Let \(X\) be a topological space, and let \(\mcalU(X)\) denote the collection of open subsets of \(X\). Fix an \(\infty\)-category \(\mcalC\) which admits small limits. A \emph{\(\mcalC\)-valued presheaf} is a functor \(\mcalU(X)^{\opposite} \to \mcalC\). The \(\infty\)-category of \(\mcalC\)-valued presheaves is denoted by \(\PSh(X, \mcalC)\).
    \item A \(\mcalC\)-valued presheaf \(\mcalF\) is a \emph{\(\mcalC\)-valued sheaf} if the following condition holds: for any open covering \(U = \cup_{\lambda \in \Lambda} U_{\lambda}\) of an open subset \(U\) of \(X\), the canonical morphism \(\mcalF(U) \xrightarrow{\simeq} \lim_{V \in \mcalU'}\mcalF(V)\) is an equivalence in \(\mcalC\), where \(\mcalU'\) is the collection of open subsets \(V \subseteq X\) that are contained in some \(U_{\lambda}\). See \cite{lurie2018Spectral}*{\S 1.1}. The \(\infty\)-category of \(\mcalC\)-valued sheaves is denoted by \(\Shv(X, \mcalC)\).
    \item Take a morphism of topological spaces \(f \colon X \to Y\) and a \(\mcalC\)-valued (pre)sheaf \(\mcalF\) on \(X\). Then we define the pushforward \(f_*\mcalF\), which is a \(\mcalC\)-valued (pre)sheaf on \(Y\) defined by \(U \mapsto \mcalF(f^{-1}(U))\) for any open subset \(U \subseteq Y\).
\end{enumalphp}

\subsubsection{Completion functors}

Let $R$ be a ring and let $I$ be a finitely generated ideal of $R$.
Let $S$ be a $G$-graded ring and let $J$ be a finitely generated homogeneous ideal of $S$.
\begin{enumalphp}
    \item The symbol \(\comp{I}{-}\) is used only for the (classical) \(I\)-adic completion of discrete rings and modules.
    \item The symbol \(\dcomp{I}{-}\) is used to denote the derived (\(I\)-)completion of objects in a derived category such as \(\mcalD(R)\).
    \item The symbol \(\grcomp{J}{-}\) is the \emph{gradedwise (\(J\)-)completion} of discrete graded rings and modules which is defined in \cite{ishizuka2025Graded}*{Construction 3.3}. In loc. cit., it is denoted by \((-)^{\grwedge}\) but in this paper, we use the symbol \(\grcomp{J}{-}\) to avoid confusion with other derived completion functors.
    \item The symbol \(\dgrcomp{J}{-}\) is the \emph{derived gradedwise \(J\)-completion} of objects in \(\mcalD_{\graded{G}}(S)\), which is defined in \cite{ishizuka2026Derived}. See \Cref{DefDerivedGradedwiseComp}.
    \item The full subcategory of \emph{derived \(I\)-complete} objects in \(\mcalD(R)\) is denoted by \(\mcalD^{\comp{I}}(R)\) and the full subcategory of \emph{derived gradedwise \(J\)-complete} objects in \(\mcalD_{\graded{G}}(S)\) is denoted by \(\mcalD_{\graded{G}}^{\comp{J}}(S)\) (\cite{ishizuka2026Derived}*{Definition 4.6}).
    \item We will denote the derived (gradedwise) completed tensor products by \(- \widehat{\otimes}^L -\) and \(- \cLgrotimes -\) in \(\mcalD^{\comp{I}}(R)\) and \(\mcalD_{\graded{G}}^{\comp{J}}(S)\) respectively.
    \item We will also use the symbols \(\widehat{(-)}\) and \((-)^{\wedge}\) when convenient; their meanings will be made clear from the context.
\end{enumalphp}

\subsection{Perfectoid rings and formal schemes}

First, we recall the definition and properties of adic perfectoid rings.

\begin{definition}[{\cite{takaya2026Relative}*{Definition 4.4}}] 
    An adic ring \(R\) is called an \emph{adic perfectoid ring} if \(p\) is a topologically nilpotent element of \(R\) and the underlying ring \(R\) is a perfectoid ring.
    If an adic perfectoid ring \(R\) is complete with respect to the adic topology on \(R\), then we say that \(R\) is a \emph{complete adic perfectoid ring}.
\end{definition}

\begin{lemma}[{\cite{ishizuka2025Graded}*{Lemma 2.3}}] \label{BoundedTorsionPerfd}
    Let \(R\) be a perfectoid ring and let \(I\) be a finitely generated ideal of \(R\) containing \(p\).
    Then the following assertions hold.
    \begin{enumerate}
        \item There exists a generator \(f_1, \dots, f_r\) of \(I\) such that \(f_i\) has a compatible system of \(p\)-power roots in \(R\).
        \item  The ring \(R\) has bounded \(f_i^{\infty}\)-torsion for such choice of \(f_i\) in (1). In particular, \(R\) has bounded \(I^{\infty}\)-torsion.
        \item The \(I\)-adic completion of \(R\) is a perfectoid ring and agrees with the derived \(I\)-completion \(\dcomp{I}{R}\) of \(R\).
    \end{enumerate}
\end{lemma}

We also recall the notion of a perfectoid formal scheme.

\begin{definition}[{\cite{takaya2026Relative}*{Definition 4.21}}] \label{DefPerfectoidFormalScheme}
    An adic formal scheme \(\mscrX\) such that \(p\) is a topologically nilpotent element of \(\mcalO_{\mscrX}\) is called a \emph{perfectoid formal scheme} if there exists an affine open covering \(\{\mscrU_i\}_{i \in I}\) of \(\mscrX\) such that for each \(i \in I\), the adic ring \(\mcalO_{\mscrX}(\mscrU_i)\) is a complete adic perfectoid ring.
\end{definition}

Nontrivially, any affine open section of a perfectoid formal scheme is an adic perfectoid ring:

\begin{proposition}[{\cite{takaya2026Relative}*{Proposition 4.22}}] \label{PropPerfectoidFormalScheme}
    An adic formal scheme \(\mscrX\) is a perfectoid formal scheme if and only if, for every affine open subset \(\mscrU \subseteq \mscrX\), the ring \(\mcalO_{\mscrX}(\mscrU)\) is a complete adic perfectoid ring.
\end{proposition}

\subsection{Derived graded modules}

We recall several results on derived graded modules that will be used throughout the paper.
See \cite{ishizuka2026Derived} for more details.

\begin{definition}[{cf. \cite{ishizuka2026Derived}*{Definition 3.5}}] \label{DefDerivedGradedwiseComp}
    Let \(R\) be a \(G\)-graded ring and let \(I\) be a finitely generated homogeneous ideal of \(R\).
    Then we can define the \(\infty\)-category \(\mcalD_{\graded{G}}(R)\) of \(G\)-graded derived \(R\)-modules and its full subcategory \(\mcalD_{\graded{G}}^{\comp{I}}(R)\) of derived gradedwise \(I\)-complete objects.
    The inclusion functor \(\mcalD_{\graded{G}}^{\comp{I}}(R) \hookrightarrow \mcalD_{\graded{G}}(R)\) admits a left adjoint functor \(\dgrcomp{I}{-}\), which is called the \emph{derived gradedwise \(I\)-completion}.
\end{definition}

\begin{definition}[{cf. \cite{ishizuka2026Derived}*{Definition 6.3}}]
    Let \(R\) be a \(G\)-graded ring and let \(M\) be an object of \(\mcalD(R)\).
    Let \(\rho_R \colon R \to R[G]\) be the coaction corresponding to the \(G\)-grading on \(R\), where \(R[G] \defeq \bigoplus_{g \in G} R \cdot t^g\) is the group ring of \(G\) over \(R\).
    Then we define a functor
    \begin{equation*}
        \mcalD(R) \to \mcalD(R[G]) \to \mcalD(R); \quad M \to M \otimes^L_R R[G] \to \rho_{R, *}(M \otimes^L_R R[G]) \eqdef M[G].
    \end{equation*}
    Take a finitely generated homogeneous ideal \(I\) of \(R\). The derived \(I\)-completion induces a functor
    \begin{equation*}
        \mcalD(R) \to \mcalD^{\comp{I}}(R); \quad M \to \dcomp{I}{M[G]} \eqdef M\abracket{G}.
    \end{equation*}
\end{definition}

\begin{proposition}[{cf. \cite{ishizuka2026Derived}*{Corollary 6.6}}] \label{AdjointDerivedGradedMod}
    Let \(R\) be a \(G\)-graded ring and let \(I\) be a finitely generated homogeneous ideal of \(R\).
    Then the functor
    \begin{equation*}
        \mcalF^I \colon \mcalD_{\graded{G}}^{\comp{I}}(R) \to \mcalD^{\comp{I}}(R); \quad M \to \dcomp{I}{M}
    \end{equation*}
    admits a right adjoint \(\mcalG^I\) with a commutative diagram
    \begin{equation*}
        \begin{tikzcd}
        \mcalD^{\comp{I}}(R) \arrow[rd, "{M \mapsto M[G]}"'] \arrow[r, "\mcalG^I"] & \mcalD_{\graded{G}}^{\comp{I}}(R) \arrow[d, "\mathrm{forget}"] \\
        & \mcalD(R)  
        \end{tikzcd}
    \end{equation*}
    of \(\infty\)-categories.
\end{proposition}

We record some results on derived graded modules. The first one is a Nakayama-type lemma.

\begin{lemma}[{cf. \cite{ishizuka2026Derived}*{Corollary 4.10 and Lemma 5.17}}] \label{GradedNakayamaLemma}
    The functors
    \begin{align*}
        \mcalF^I & \colon \mcalD_{\graded{G}}^{\comp{I}}(R) \to \mcalD^{\comp{I}}(R); \quad M \to \dcomp{I}{M} \quad \text{and} \\
        \Pro(\mcalF^I) & \colon \Pro(\mcalD_{\graded{G}}^{\comp{I}}(R)) \to \Pro(\mcalD^{\comp{I}}(R))
    \end{align*}
    are conservative.
\end{lemma}

In general, the forgetful functor \(\mcalD^{\comp{I}}_{\graded{G}}(R) \to \mcalD^{\comp{I}}(R)\) does not preserve limits, but we give a sufficient condition under which it does.

\begin{theorem}[{cf. \cite{ishizuka2026Derived}*{Theorem 6.21}}] \label{GradedLimitCommutative}
    Let $R$ be a $G$-graded ring and let $I$ be a finitely generated homogeneous ideal of \(R\).
    Take a diagram \(\{M_j\}_{j \in J} \colon J \to \mcalD_{\graded{G}}^{\comp{I}}(R)\) indexed by a (not necessarily small) simplicial set \(J\).
    Assume the existence of the limits of the diagrams
    \begin{align*}
        \{\mcalF^I(M_j)\}_{j \in J} = \{\dcomp{I}{M_j}\}_{j \in J} & \colon J \to \mcalD^{\comp{I}}(R); \quad j \mapsto \dcomp{I}{M_j} \quad \text{and} \\
        \{\mcalT^I(\mcalF^I(M_j))\} = \{\dcomp{I}{M_j}\abracket{G}\}_{j \in J} & \colon J \to \mcalD^{\comp{I}}(R); \quad j \mapsto \dcomp{I}{M_j} \abracket{G}.
    \end{align*}
    If the natural morphism
    \begin{equation} \label{AssumptionIsomMorphism}
        (\lim_{J} \dcomp{I}{M_j})\abracket{G} \to \lim_{J}(\dcomp{I}{M_j}\abracket{G})
    \end{equation}
    is an isomorphism in \(\mcalD^{\comp{I}}(R)\), then the limit $\grlim_{J} M_j$ of $\{M_j\}_{j \in J}$ exists in $\mcalD_{\graded{G}}^{\comp{I}}(R)$, and the natural morphism
    \begin{equation*}
        \mcalF^I(\grlim_{J} M_j) = \dcomp{I}{\grlim_{J} M_j} \to \lim_{J} \dcomp{I}{M_j} = \lim_{J} \mcalF^I(M_j)
    \end{equation*}
    is an isomorphism in \(\mcalD^{\comp{I}}(R)\), i.e., the limit is preserved by the functor \(\mcalF^I\).
\end{theorem}

\subsection{Graded perfectoid rings}

Following \cite{ishizuka2025Graded}, we introduce the notion of graded perfectoid rings.
We recall the definitions and properties that will be used later. See \cite{ishizuka2025Graded} for more details.

\begin{definition}[{cf. \cite{ishizuka2025Graded}*{Definition 4.1}}] \label{DefGradedPerfd}
    We define a graded variant of perfectoid rings as follows.
    \begin{enumerate}
        \item Let $R$ be a $G$-graded ring.
        We say that $R$ is \emph{a $G$-graded perfectoid ring} if $R$ is gradedwise $p$-adic complete and $R^{\wedge p}$ is a perfectoid ring.
        \item Let \(R = (R, R_{\graded})\) be a pro-\(G\)-graded ring. We say that \(R\) is a \emph{pro-\(G\)-graded perfectoid ring} if \(p\) is topologically nilpotent in \(R_{\graded}\) and \(R_{\graded}\) is a \(G\)-graded perfectoid ring.
    \end{enumerate}
    On a base \(G\)-graded adic ring \(A\) with a homogeneous ideal of definition \(I\), the category of \(I\)-adic \(G\)-graded perfectoid \(A\)-algebras is denoted by \(\Perfd_{\graded{G}}^{\wedge I}(A)\).
\end{definition}

\begin{remark}[{cf. \cite{ishizuka2025Graded}*{Remark 4.2}}] \label{CatEquivGrPerfd}
    The assignment \((R, R_{\graded}) \mapsto R_{\graded}\) induces an equivalence of categories between the category of pro-\(G\)-graded \(p\)-adic perfectoid rings and that of \(G\)-graded perfectoid rings.
    So the following proposition also holds for \(G\)-graded perfectoid rings.
\end{remark}

\begin{proposition}[{cf. \cite{ishizuka2025Graded}*{Proposition 4.4}}] \label{graded-perfectoid-equiv}
    Let \((R, R_{\graded})\) be a pro-\(G\)-graded adic ring such that \(p\) is topologically nilpotent in \(R_{\graded}\).
    Then the following conditions are equivalent:
    \begin{enumerate}
        \item \((R, R_{\graded})\) is pro-\(G\)-graded perfectoid;
        \item \(R\) and \(R_0\) are perfectoid rings;
        \item there exists a topologically nilpotent homogeneous element \(\varpi \in R_0\) with \(p \in \varpi^p R_0\) such that the Frobenius map
        \begin{equation} \label{IsomPPower}
            R_{\graded}/\varpi R_{\graded}
            \xrightarrow{\,a \mapsto a^p\,}
            R_{\graded}/\varpi^p R_{\graded}
        \end{equation}
        is an isomorphism, and the multiplicative map
        \begin{equation} \label{MultiplicativePPower}
            R_{\graded}[\varpi^\infty]
            \xrightarrow{\,a \mapsto a^p\,}
            R_{\graded}[\varpi^\infty]
        \end{equation}
        is bijective. 
    \end{enumerate}
    Furthermore, in this case, we necessarily have \(G = G[1/p]\).
\end{proposition}

\begin{corollary}[{cf. \cite{ishizuka2025Graded}*{Corollary 4.7}}] \label{PureSubPerfd}
    Let \((R, R_{\graded})\) be a pro-\(G\)-graded perfectoid ring with a homogeneous ideal of definition \(I\).
    Let \(H \subseteq G\) be a submonoid satisfying \(H = H[1/p]\), and set 
    \[
    R'_{\graded} \coloneqq \bigoplus_{h \in H} R_h \subseteq R_{\graded}.
    \]
    Then the pair \((R', R'_{\graded})\) is an \(H\)-graded perfectoid ring, where 
    \(R'\) denotes the \((I \cap R'_{\graded})\)-adic completion of \(R'_{\graded}\).
\end{corollary}


First we recall some properties on graded perfectoid rings. See \cite{ishizuka2025Graded} for the proof of the following two lemmas.

\begin{lemma}[{cf. \cite{ishizuka2025Graded}*{Lemma 4.8}}] \label{CompatibleSystemModp}
    Let \(R\) be a \(G\)-graded ring such that \(R\) is gradedwise \(p\)-complete and \(R/pR\) is semiperfect.
    Take a homogeneous element \(f\) of \(R\).
    Then there exists a homogeneous element \(g \in R\) such that \(f \equiv g \bmod pR\) and \(g\) admits a compatible system \(\{g^{1/p^n}\}_{n \geq 0}\) of homogeneous \(p\)-power roots in \(R\).
\end{lemma}


\begin{lemma}[{cf. \cite{ishizuka2025Graded}*{Lemma 4.9}}] \label{BoundedTorsionGradedPerfd}
    Let \((R, R_{\graded})\) be a \(p\)-adic \(G\)-graded perfectoid ring and let \(I\) be a finitely generated homogeneous ideal of \(R_{\graded}\) containing \(p\).
    Then the following assertions hold.
    \begin{enumerate}
        \item There exists a homogeneous generator \(f_1, \dots, f_r\) of \(I\) such that \(f_i\) has a compatible system of \(p\)-power roots in \(R_{\graded}\).
        \item The \(G\)-graded ring \(R_{\graded}\) has bounded \(f_i^{\infty}\)-torsion for such choice of \(f_i\) in (1). In particular, \(R_{\graded}\) has bounded \(I^{\infty}\)-torsion.
        \item The pair consisting of the \(I\)-adic completion \(\comp{I}{R}\) of \(R\) and the gradedwise \(I\)-adic completion \(\grcomp{I}{R_{\graded}}\) of \(R_{\graded}\) is an \(I\)-adic pro-\(G\)-graded perfectoid ring.
        \item The canonical morphism
        \begin{equation*}
            \dcomp{I}{R_{\graded}} \to \comp{I}{R_{\graded}}
        \end{equation*}
        is an isomorphism in \(\mcalD(R_{\graded})\).
    \end{enumerate}
\end{lemma}

In our previous paper \cite{ishizuka2026Derived}, we have introduced the derived gradedwise completion, so we need to compare it with the usual gradedwise completion for graded perfectoid rings as follows.

\begin{lemma} \label{DerivedGrcompGradedPerfd}
    Let \((R, R_{\graded})\) be a \(p\)-adic \(G\)-graded perfectoid ring.
    Let $I$ be a finitely generated homogeneous ideal of $R_{\graded}$ containing $p$.
    Then the canonical morphism
    \begin{equation*}
        \dgrcomp{I}{R_{\graded}} \to \grcomp{I}{R_{\graded}}
    \end{equation*}
    is an isomorphism in \(\mcalD_{\graded{G}}(R_{\graded})\).
\end{lemma}

\begin{proof}
    By \Cref{BoundedTorsionGradedPerfd}(1), we may choose homogeneous generators \(f_1, \dots, f_r\) of \(I\) such that \(f_i\) has a compatible system of \(p\)-power roots in \(R_{\graded}\).
    Set \(I_i \defeq (f_1, \dots, f_i)\) in \(R_{\graded}\).
    By \Cref{BoundedTorsionGradedPerfd}(2) and (3), the \(I_i\)-adic completion \((\comp{I_i}{R}, \grcomp{I_i}{R_{\graded}})\) is an \(I_i\)-adic \(G\)-graded perfectoid ring and \(\grcomp{I_i}{R_{\graded}}\) has bounded \(f_{i+1}^{\infty}\)-torsion
    Therefore, we have an isomorphism
    \[
    \dgrcomp{f_{i+1}}{\grcomp{I_i}{R_{\graded}}} \xrightarrow{\simeq} \grcomp{f_{i+1}}{\grcomp{I_i}{R_{\graded}}}
    \]
    by \cite{ishizuka2026Derived}*{Lemma 4.8(8)}.
    Iterating this argument yields the desired isomorphism.
\end{proof}

\begin{lemma} \label{LocalizationGradedPerfd}
Let \((R, R_{\graded})\) be a \(G\)-graded perfectoid ring with a homogeneous ideal of definition \(I\) containing \(p\), and let \(f\) be a homogeneous element of \(R_{\graded}\).
\begin{enumerate}
    \item The pair \((\comp{I}{R[1/f]}, \grcomp{I}{R_{\graded}[1/f]})\) is a \(G\)-graded perfectoid ring.
    \item The canonical morphism
    \[
      \dgrcomp{I}{R_{\graded}[1/f]} \;\longrightarrow\; \grcomp{I}{R_{\graded}[1/f]}
    \]
    in \(\mcalD_{\graded{G}}(R_{\graded})\) is an isomorphism.
    \item The canonical morphism
    \begin{equation*}
        \dcomp{I}{R_{\graded}[1/f]} \to \comp{I}{R_{\graded}[1/f]}
    \end{equation*}
    in \(\mcalD(R_{\graded})\) is an isomorphism.
\end{enumerate}
\end{lemma}

\begin{proof}
(1) By \cite{ishizuka2025Graded}*{Lemma~3.12}, the pair \((\comp{p}{R_{\graded}[1/f]}, \grcomp{p}{R_{\graded}[1/f]})\) is a pro-\(G\)-graded ring.
By \cite{bhatt2019Regular}*{Example~3.8(7)}, both \(\comp{p}{R[1/f]}\) and \((\grcomp{p}{R_{\graded}[1/f]})_0\) are perfectoid; in particular, \((\comp{p}{R_{\graded}[1/f]}, \grcomp{p}{R_{\graded}[1/f]})\) is \(p\)-adic \(G\)-graded perfectoid.
Since \(I\) contains \(p\), passing to the \(I\)-adic completion yields the claim by \Cref{BoundedTorsionGradedPerfd}(3).

(2) Because \(R_{\graded} \hookrightarrow R\) is injective, the ring \(R_{\graded}\) has bounded \(p^\infty\)-torsion, and hence so does \(R_{\graded}[1/f]\).
By (1), the pair \((\comp{p}{R_{\graded}[1/f]}, \grcomp{p}{R_{\graded}[1/f]})\) is \(p\)-adic \(G\)-graded perfectoid. Therefore
\[
  \dgrcomp{I}{\dgrcomp{p}{R_{\graded}[1/f]}}
    \xrightarrow{\simeq}
  \dgrcomp{I}{\grcomp{p}{R_{\graded}[1/f]}}
    \xrightarrow{\simeq}
  \grcomp{I}{\grcomp{p}{R_{\graded}[1/f]}},
\]
where the first isomorphism follows from \cite{ishizuka2026Derived}*{Lemma 4.8(8)} and the second isomorphism follows from \Cref{DerivedGrcompGradedPerfd}.

(3): The given morphism can be identified with the composition of the following isomorphisms:
\begin{equation*}
    \dcomp{I}{R_{\graded}[1/f]} \xrightarrow{\cong} \dcomp{I}{\dgrcomp{I}{R_{\graded}[1/f]}} \xrightarrow{\cong} \dcomp{I}{\grcomp{I}{R_{\graded}[1/f]}} \xrightarrow{\cong} \comp{I}{\grcomp{I}{R_{\graded}[1/f]}} \xleftarrow{\cong} \comp{I}{R_{\graded}[1/f]},
\end{equation*}
where the first and last isomorphisms come from \cite{ishizuka2026Derived}*{Definition 4.3 and Lemma 4.8(9)}, the second one follows from (2), and the third one follows from \Cref{DerivedGrcompGradedPerfd} since the pair \((\comp{p}{\grcomp{I}{R_{\graded}[1/f]}}, \grcomp{I}{R_{\graded}[1/f]})\) is a \(p\)-adic \(G\)-graded perfectoid ring by (1).
\end{proof}

Finally, we show that a graded ring whose graded pieces are invertible modules over a perfectoid ring produces a graded perfectoid ring.

\begin{lemma} \label{EnlargeGradedPerfdNew}
    Let \(R = \bigoplus_{g \in G} R_g\) be a \(p\)-adically gradedwise complete \(G\)-graded ring such that \(R_0\) is a perfectoid ring and \(G = G[1/p]\).
    Assume that, for every \(g \in \Supp(R)\), the \(R_0\)-module \(R_g\) is invertible, and that the multiplication map \(R_g \otimes_{R_0} R_h \to R_{g+h}\) is an isomorphism for any \(g, h \in \Supp(R)\).
    Then there exists a split \(G\)-graded ring homomorphism
    \begin{equation*}
        R \hookrightarrow R' = \bigoplus_{g \in G} R'_g
    \end{equation*}
    such that \(R'\) is a \(G\)-graded perfectoid ring and the inclusion \(R_g \hookrightarrow R'_g\) is the identity morphism for any \(g \in \Supp(R)\).
\end{lemma}

\begin{proof}
    Since the Picard group of a perfectoid ring is uniquely \(p\)-divisible (\cite{bhatt2022Prismsa}*{Corollary 9.7} and \cite{dorfsman-hopkins2023Untilting}*{Corollary 4.2}), we can show that there exists a unique compatible system \(\{R_g^{1/p^n}\}_{n \geq 0}\) of \(p\)-power roots of \(R_g\) for each \(g \in \Supp(R)\), i.e., invertible \(R_0\)-modules \(R_g^{1/p^{n+1}}\) such that \((R_g^{1/p^{n+1}})^{\otimes p} \cong R_g^{1/p^n}\) and \(R_g^{1/p^0} = R_g\).
    If \(g \notin \Supp(R)\), we set \(R_g^{1/p^n} = 0\) for all \(n \geq 0\).
    Because of the uniqueness, the multiplication map \(R_g \otimes_{R_0} R_h \xrightarrow{\cong} R_{g+h}\) extends to a multiplication map
    \begin{equation*}
        R_g^{1/p^k} \otimes_{R_0} R_h^{1/p^k} \xrightarrow{\cong} R_{g+h}^{1/p^k}
    \end{equation*}
    for any \(g, h \in \Supp(R)\) and \(k \geq 0\).
    Also, the same reason gives the isomorphism
    \begin{equation*}
        R_g \xrightarrow{\cong} R_{p^kg}^{1/p^k}
    \end{equation*}
    for any \(g \in \Supp(R)\) and \(k \geq 0\).
    This defines a \(G\)-graded ring \(R^{(k)}\) by
    \begin{equation*}
        R^{(k)} \defeq \bigoplus_{g \in G} R_g^{1/p^k}
    \end{equation*}
    for each \(k \geq 0\) and the split graded ring homomorphism
    \begin{equation*}
        R^{(k)} \hookrightarrow R^{(k+1)}; \quad R_g^{1/p^k} \xrightarrow{\cong} R_{pg}^{1/p^{k+1}}.
    \end{equation*}
    Taking the graded colimit over \(k \geq 0\), we obtain a \(G\)-graded ring
    \begin{equation*}
        R' \defeq \grcolim_{k \geq 0} R^{(k)} = \bigoplus_{g \in G} R'_g,
    \end{equation*}
    where \(R'_{g/p^k} = R_g^{1/p^k}\) for every \(g \in \Supp(R)\) and \(k \geq 0\), so \(\Supp(R') = \Supp(R)[1/p]\) holds.
    Since \(R^{(0)}\) is the same as \(R\) itself, we have a split graded ring homomorphism \(R \hookrightarrow R'\) such that the inclusion \(R_g \hookrightarrow R'_g\) is the identity morphism for any \(g \in \Supp(R)\).
    Any \(R'_g\) is an invertible \(R_0\)-module and thus this \(R'\) is \(p\)-adically gradedwise complete.

    To show that \(R'\) is a \(G\)-graded perfectoid ring, it suffices to check the bijections \eqref{IsomPPower} and \eqref{MultiplicativePPower} in \Cref{graded-perfectoid-equiv}.
    Fix an element \(\varpi \in R_0\) such that \(p \in \varpi^p R_0\). We need to show that the \(p\)-power maps
    \begin{equation*}
        R'_{g/p}/\varpi R'_{g/p} \xrightarrow{a \mapsto a^p} R'_{g}/\varpi^p R'_{g} \quad \text{and} \quad R'_{g/p}[\varpi^{\infty}] \xrightarrow{a \mapsto a^p} R'_{g}[\varpi^{\infty}]
    \end{equation*}
    are bijective for any \(g \in G\).
    For each \(g \in \Supp(R)[1/p]\), the module \(R'_g\) is an invertible \(R_0\)-module, so Zariski locally, the above maps are the same as the \(p\)-th power maps along \(R_0\), which are bijective since \(R_0\) and its localization are perfectoid. This proves that \(R'\) is a \(G\)-graded perfectoid ring.
\end{proof}





\section{Absolute graded perfectoidizations} \label{SectionTopGradedPerfd}

\subsection{Absolute perfectoidization of adic rings}

\begin{definition} \label{DefTopologicalAbsolutePerfectoidization}
    Let \(R\) be an adic ring with an ideal of definition \(I\) containing \(p\).
    We define the \emph{topological absolute perfectoidization \(R_{\tperfd}\) of \(R\)} to be the limit in \(\CAlg_R\)
    \begin{equation*}
        R_{\tperfd} \defeq \lim_{P \in \Perfd_R^{\wedge}} P
    \end{equation*}
    where the limit runs over the category \(\Perfd_R^{\wedge}\) of complete adic perfectoid \(R\)-algebras \(P\), whose existence follows from the existence of the usual absolute perfectoidization \(R_{\perfd}\) and \Cref{CompletionOfAbsolutePerfectoidization} below.
    This construction is functorial on the category of adic topological rings equipped with an ideal of definition containing \(p\).
\end{definition}

\begin{remark} \label{RemarkCompletePerfd}
    Let \(R\) be an adic ring with an ideal of definition \(I\) containing \(p\) and let \(R \to P\) be a continuous morphism to a complete adic perfectoid ring \(P\) (not necessarily \(p\)-adic).
    Then it factors uniquely as
    \[
        R \to \comp{I}{R} \to P' \to P,
    \]
    where \(P'\) denotes the ring \(P\) equipped with the \(I\)-adic topology; in particular, \(P'\) is a complete \(I\)-adic perfectoid ring.
    This shows that \(R_{\tperfd}\) can equally be computed as the limit over the category \(\Perfd^{\wedge I}(R)\) of complete \(I\)-adic perfectoid \(R\)-algebras \(P\), and in particular that
    \[
        R_{\tperfd} \cong (\comp{I}{R})_{\tperfd}.
    \]
    
    Thus, it suffices to consider \(I\)-adic perfectoid rings rather than general adic perfectoid rings.
    In particular, if \(R\) is a \(p\)-adic topological ring, then the topological absolute perfectoidization \(R_{\tperfd}\) is the same as the usual absolute perfectoidization \(R_{\perfd}\).
\end{remark}

For later use, we redefine the topological absolute perfectoidization for any derived \(I\)-complete animated rings.

    


\begin{proposition} \label{CompletionOfAbsolutePerfectoidization}
    Let \(R\) be an adic topological ring with topologically nilpotent \(p\) and let \(I\) be an ideal of definition of \(R\) which contains \(p\). Then the topological absolute perfectoidization \(R_{\tperfd}\) is canonically identified with the derived \(I\)-completion of the absolute perfectoidization \(R_{\perfd}\) of \(R\) in the sense of \cite{bhatt2024Perfectoid}*{Definition 3.10}.
\end{proposition}

\begin{proof}
    Let \(\mcalC_1\) (resp., \(\mcalC_2\)) be the category of perfectoid \(R\)-algebras (resp., complete adic perfectoid \(R\)-algebras).
    Then we have isomorphisms in \(\CAlg_R\)
    \begin{equation*}
        \dcomp{I}{R_{\perfd}} = \dcomp{I}{\lim_{P \in \mcalC_1} P} \overset{(\star_1)}{\simeq} \lim_{P \in \mcalC_1} \dcomp{I}{P} \overset{(\star_2)}{=} \lim_{P' \in \mcalC_2} P' = R_{\tperfd}
    \end{equation*}
    as follows:
    On \((\star_1)\), we know that the derived \(I\)-completion functor on \(\CAlg_R\) commutes with limits.
    
    Using \Cref{BoundedTorsionPerfd}, we know that the derived \(I\)-completion of any perfectoid \(R\)-algebra \(P\) coincides with the \(I\)-adic completion \(\comp{I}{P}\) and this is a complete \(I\)-adic perfectoid \(R\)-algebra.
    The isomorphism \((\star_2)\) follows from the fact that the category \(\mcalC_2\) identifies with the full subcategory of \(I\)-adically completed objects of \(\mcalC_1\).
    This finishes the proof.
\end{proof}

\begin{corollary} \label{DiscretePropertyTopAbsPerfd}
    Let \(R\) be an adic topological ring with topologically nilpotent \(p\) and let \(I\) be an ideal of definition of \(R\) containing \(p\).
    Assume that there exists a morphism from a perfectoid ring to \(R\).
    If \(R_{\perfd}\) is concentrated in degree \(0\), then \(R_{\tperfd}\) is an initial object in the category of \(I\)-adic complete perfectoid \(R\)-algebras.
\end{corollary}

\begin{proof}
    In this case, \(R_{\perfd}\) becomes the initial perfectoid \(R\)-algebra (\cite{bhatt2022Prismsa}*{Corollary 8.14}) and then the \(I\)-adic completion \(\comp{I}{R_{\perfd}}\) is the initial \(I\)-adic complete perfectoid \(R\)-algebra, which must be isomorphic to \(R_{\tperfd}\) by the definition in \Cref{DefTopologicalAbsolutePerfectoidization}.
    Then \(R_{\tperfd}\) is concentrated in degree \(0\) and \(I\)-adically complete.
\end{proof}

\begin{corollary} \label{CommutativeTensorTopPerfd}
    Let \(R\) be an adic perfectoid ring with topologically nilpotent \(p\) and let \(I\) be an ideal of definition of \(R\) containing \(p\).
    Then the functor
    \begin{equation*}
        \CRing^{\wedge I}_R \to \CAlg_R^{\comp{I}}; \quad S \mapsto S_{\tperfd}
    \end{equation*}
    preserves small colimits, where \(\CRing^{\wedge I}_R\) is the category of \(I\)-adically complete \(R\)-algebras and \(\CAlg_R^{\comp{I}}\) is the \(\infty\)-category of derived \(I\)-complete \(\setE_{\infty}\)-\(R\)-algebras.
\end{corollary}

\begin{proof}
    In this proof, to distinguish colimits, we write the colimits in \(\CRing_R\), \(\CAlg_R^{an}\), and \(\CAlg_R\) as \(\colim\), \(\colim^{an}\), and \(\colim^{\setE_{\infty}}\) respectively.

    The inclusion functor \(\CRing_R \to \CAlg^{an}_R\) admits the truncation functor
    \begin{equation*}
        \tau_{\leq 0} = \pi_0 \colon \CAlg^{an}_R \to \CRing_R
    \end{equation*}
    as a left adjoint; in particular, \(\pi_0\) commutes with small colimits.
    In particular, the colimit of any small diagram \(\{S_k\}_{k \in K}\) in the category \(\CRing_R^{\comp{p}}\) of derived \(p\)-complete \(R\)-algebras is given by
    \begin{equation} \label{ColimDpCompRings}
        \pi_0\bigl(\dcomp{p}{\colim_{k \in K} S_k}\bigr) \cong \pi_0\bigl(\dcomp{p}{\pi_0(\colim_{k \in K}^{an} S_k)}\bigr).
    \end{equation}
    
    Considering the composition of functors
    \begin{equation} \label{TopPerfdFuncP}
        \CRing_R^{\wedge p} \hookrightarrow \CRing^{\comp{p}} \to \CAlg^{an, \comp{p}}_R \xrightarrow{S \mapsto S_{\perfd}} \CAlg_R^{\comp{p}},
    \end{equation}
    where \(\CAlg^{an,\comp{p}}_R\) is the \(\infty\)-category of derived \(p\)-complete animated \(R\)-algebras.
    By \cite{bhatt2022Prismsa}*{Construction 7.6}, the absolute perfectoidization functor \(\CAlg^{an, \comp{p}}_R \to \CAlg_R^{\comp{p}}, S \mapsto S_{\perfd}\) commutes with small colimits.
    Take any small diagram \(\{S_k\}_{k \in K}\) of \(\CRing_R^{\wedge p}\).
    Applying the functor above to this colimit, we obtain
    \begin{align*}
        (\comp{p}{\colim_{k \in K} S_k})_{\perfd} & \xleftarrow{\cong} (\dcomp{p}{\colim_{k \in K} S_k})_{\perfd} \xleftarrow{\cong} (\pi_0(\dcomp{p}{\pi_0(\colim^{an}_{k \in K} S_k})))_{\perfd} \\
        & \xleftarrow{\cong} (\colim^{an}_{k \in K} S_k)_{\perfd} \xrightarrow{\cong} (\dcomp{p}{\colim^{an}_{k \in K} S_k})_{\perfd} \xleftarrow{\cong} \dcomp{p}{\colim^{\setE_{\infty}}_{k \in K} (S_k)_{\perfd}}
    \end{align*}
    in \(\CAlg_R^{\comp{p}}\) by the isomorphism \eqref{ColimDpCompRings} and the invariance (resp., commutativity) of \((-)_{\perfd}\) under (derived) \(p\)-completion and taking \(\pi_0\) (resp., with small colimits).
    Therefore, the functor \eqref{TopPerfdFuncP} commutes with small colimits.

    Take any small diagram \(\{S_k\}_{k \in K}\) of \(\CRing_R^{\wedge I}\).
    Note that the topological absolute perfectoidization functor \((-)_{\tperfd}\) is independent on the (derived) \(I\)-completion (\Cref{RemarkCompletePerfd}) and is the derived \(I\)-completion of \((-)_{\perfd}\) (\Cref{CompletionOfAbsolutePerfectoidization}).
    Using these facts together with the preceding statement for \((-)_{\perfd}\), we obtain isomorphisms
    \begin{equation*}
        (\comp{I}{\colim_{k \in K} S_k})_{\tperfd} \xleftarrow{\cong} (\comp{p}{\colim_{k \in K} S_k})_{\tperfd} \cong \dcomp{I}{\dcomp{p}{\colim^{\setE_{\infty}}_{k \in K} (S_k)_{\perfd}}}  \xrightarrow{\cong} \dcomp{I}{\colim^{\setE_{\infty}}_{k \in K} (S_k)_{\tperfd}}
    \end{equation*}
    in \(\CAlg_R^{\comp{I}}\).
    Since colimits in \(\CRing_R^{\wedge I}\) and \(\CAlg_R^{\comp{I}}\) are computed by taking the (derived) \(I\)-completion of the corresponding colimits in \(\CRing_R\) and \(\CAlg_R\), respectively, we know that the topological absolute perfectoidization functor also preserves small colimits.
\end{proof}

\begin{corollary} \label{CommutativeLocalizationTopPerfd}
    Let \(R\) be an adic topological ring with topologically nilpotent \(p\) and let \(I\) be an ideal of definition of \(R\) containing \(p\).
    Then there is a canonical isomorphism
    \begin{equation*}
        \dcomp{I}{R_{\tperfd}[1/f]} \xrightarrow{\cong} (R[1/f])_{\tperfd}
    \end{equation*}
    in \(\CAlg_R\).
    In particular, we have the following commutativity of limits:
    \begin{equation} \label{CommutesLImitIsomTop}
        \dcomp{I}{R_{\tperfd}[1/f]} = \dcomp{I}{(\lim_{R \to P}P)[1/f]} \xrightarrow{\cong} \lim_{R \to P} \comp{I}{P[1/f]}
    \end{equation}
    in \(\CAlg_R\), where the limits run through all \(I\)-adic complete perfectoid \(R\)-algebras \(P\).
    The same statement holds if we localize at any multiplicative subset \(W\) of \(R\).
\end{corollary}

\begin{proof}
    The general case follows by the same argument as in the case of \(W = \{1, f, f^2, \dots\}\), and we only prove this case.
    By \cite{bhatt2024Perfectoid}*{Proposition 3.18 and Lemma 3.21}, the isomorphism \(\dcomp{p}{R_{\perfd}[1/f]} \xrightarrow{\cong} (\dcomp{p}{R[1/f]})_{\perfd} \cong (R[1/f])_{\perfd}\) holds in \(\CAlg_R\).
    Combining this with \Cref{CompletionOfAbsolutePerfectoidization}, we have the following isomorphism
    \begin{align*}
        \dcomp{I}{R_{\tperfd}[1/f]} & \cong \dcomp{I}{(\dcomp{I}{R_{\perfd}})[1/f]} \cong \dcomp{I}{R_{\perfd}[1/f]} \cong \dcomp{I}{\dcomp{p}{R_{\perfd}[1/f]}} \\
        & \cong \dcomp{I}{(R[1/f])_{\perfd}} \cong (R[1/f])_{\tperfd}
    \end{align*}
    in \(\CAlg_R\).

    Finally, we will prove \eqref{CommutesLImitIsomTop}.
    By the definition of \(R_{\tperfd}\) in \Cref{DefTopologicalAbsolutePerfectoidization}, we have an isomorphism
    \begin{equation} \label{PartialCommuteslimitIsomTop}
        \dcomp{I}{(\lim_{R \to P}P)[1/f]} = \dcomp{I}{R_{\tperfd}[1/f]} \xrightarrow{\cong} (R[1/f])_{\tperfd} = \lim_{R[1/f] \to P'} P'
    \end{equation}
    where the first limits run through all \(I\)-adic complete perfectoid \(R\)-algebras \(P\) and the second one is over all \(I\)-adic complete perfectoid \(R[1/f]\)-algebras \(P'\).
    Any \(I\)-adic complete perfectoid \(R[1/f]\)-algebra \(P'\) is an \(I\)-adic complete perfectoid \(R\)-algebra and satisfies \(\comp{I}{P'[1/f]} = P'\).
    Conversely, for any \(I\)-adic complete perfectoid \(R\)-algebra \(P\), the localization \(\comp{I}{P[1/f]}\) is an \(I\)-adic complete perfectoid \(R[1/f]\)-algebra.
    So the isomorphism
    \begin{equation} \label{PerfdAlgebraLocalizationLimitTop}
        \lim_{R[1/f] \to P'} P' \cong \lim_{R \to P} \comp{I}{P[1/f]}
    \end{equation}
    holds.
    Combining this with the isomorphism \eqref{PartialCommuteslimitIsomTop} above shows the desired isomorphism \eqref{CommutesLImitIsomTop}.
\end{proof}

\subsection{Limits of graded perfectoid rings}

\begin{proposition} \label{CofinalGradedPerfdForGradedRing}
    Let \(R\) be a \(G\)-graded adic ring with \(G = G[1/p]\).
    Let $H \subseteq G$ be a submonoid such that $Supp(R) \subseteq H$.
    Take any complete adic perfectoid \(R\)-algebra \(P\) with an ideal of definition \(J\) containing \(p\).
    Then there exists a gradedwise complete \(G\)-graded adic perfectoid \(R\)-algebra \(Q\) and an \(R\)-algebra morphism \(Q \to P\) such that $\Supp(Q) \subseteq H[1/p]$.
\end{proposition}

\begin{proof}
    Take a complete \(J\)-adic perfectoid \(R\)-algebra \(\psi \colon R \to P\). Note that this is assumed to be an adic morphism.
    Since \(\Supp(R) \subseteq H \subseteq H[1/p]\), we may regard \(R\) as an \(H[1/p]\)-graded ring, and we have a morphism of \(H[1/p]\)-graded \(R\)-algebras
    \begin{equation*}
        \phi \colon R \to Q \defeq P[H[1/p]]; \quad \sum_{g \in H[1/p]} r_g \mapsto \sum_{g \in H[1/p]} \psi(r_g) t^g
    \end{equation*}
    whose composition with the evaluation morphism \(Q \to P\) sending each \(t^g\) to \(1\) gives the original morphism \(\psi \colon R \to P\).\footnote{This is a variant of the right adjoint given in \Cref{AdjointDerivedGradedMod}.}

    Since \(P\) is \(J\)-adically complete, the \(H[1/p]\)-graded \(J\)-adic ring \(Q\) is \(J\)-adically gradedwise complete and \(\phi\) is an adic morphism.
    Therefore, it remains to show that \(Q\) is an \(H[1/p]\)-graded perfectoid ring.
    
    Choose an element \(\varpi \in P\) such that \(\varpi^p = pu\) for some unit \(u \in P\).
    Using \Cref{graded-perfectoid-equiv}, it suffices to show that the Frobenius map
    \begin{equation*}
        \varphi \colon Q/\varpi Q \to Q/p Q; \quad a \mapsto a^p
    \end{equation*}
    is an isomorphism and the multiplicative map
    \begin{equation*}
        Q[\varpi^\infty] \xrightarrow{a \mapsto a^p} Q[\varpi^\infty]
    \end{equation*}
    is bijective.
    On each degree \(h/p^n \in H[1/p]\), these maps are given by
    \begin{equation*}
        Pt^{h/p^n}/\varpi Pt^{h/p^n} \xrightarrow{a \mapsto a^p} Pt^{h/p^{n-1}}/p Pt^{h/p^{n-1}} \quad \text{and} \quad Pt^{h/p^n}[ \varpi^\infty] \xrightarrow{a \mapsto a^p} Pt^{h/p^{n-1}}[\varpi^\infty].
    \end{equation*}
    Since each graded piece \(Pt^{h/p^n}\) is canonically isomorphic to \(P\), and \(P\) is perfectoid, these maps are isomorphisms.
    This proves that \(Q\) is an \(H[1/p]\)-graded perfectoid ring.
\end{proof}

\begin{corollary} \label{TopPerfdGradedRing}
    Let \(R\) be a \(G\)-graded adic ring with \(G = G[1/p]\) and let \(I\) be a homogeneous ideal of definition of \(R\) containing \(p\).
    Then the canonical morphisms
    \begin{equation*}
        R_{\tperfd} \to \lim_{P \in \Perfd_{\graded{G}}^{\wedge I}(R)} \dcomp{I}{P} \to \lim_{P \in \Perfd_{\graded{G}}^{\wedge I}(R)} \lim_{n > 0} P/I^nP
    \end{equation*}
    are isomorphisms in \(\CAlg_R\), where the middle term is the limit taken over the category \(\Perfd_{\graded{G}}^{\wedge I}(R)\) of gradedwise complete \(G\)-graded \(I\)-adic perfectoid \(R\)-algebras.
\end{corollary}

\begin{proof}
    By \Cref{BoundedTorsionGradedPerfd}(4), any gradedwise complete \(G\)-graded \(I\)-adic perfectoid \(R\)-algebra \(P\) has an isomorphism \(\dcomp{I}{P} \cong \comp{I}{P}\) in \(\CAlg_R\).
    Because of the Mittag--Leffler condition, the second morphism is an isomorphism.

    We prove that the first morphism is an isomorphism.
    Because of the definition of \(R_{\tperfd}\) (\Cref{DefTopologicalAbsolutePerfectoidization} and \Cref{RemarkCompletePerfd}), it suffices to show that the functor
    \begin{equation*}
        \Perfd_{\graded{G}}^{\wedge I}(R) \to \Perfd^{\wedge I}(R); \quad P \mapsto \comp{I}{P}
    \end{equation*}
    is cofinal, where \(\Perfd^{\wedge I}_R\) is the category of \(I\)-adic complete perfectoid \(R\)-algebras.
    This follows from \Cref{CofinalGradedPerfdForGradedRing}.
\end{proof}

\subsection{Absolute perfectoidization of graded rings}

\begin{lemma} \label{GradedPerfdCofinalFunctors}
    Let \(R\) be a \(G\)-graded adic ring with \(G = G[1/p]\) and let \(I\) be a homogeneous ideal of definition of \(R\) containing \(p\).
    Let \(H\) be the submonoid of \(G\) generated by the support \(\Supp(R)\) of \(R\).
    Then the canonical functors of categories
    \begin{equation*}
        \Perfd_{\graded{H[1/p]}}^{\wedge I}(R) \hookrightarrow \Perfd_{\graded{G}}^{\wedge I}(R) \hookrightarrow \Perfd_{\graded{G}}^{\wedge}(R)
    \end{equation*}
    are cofinal, where \(\Perfd_{\graded{G}}^{\wedge}(R)\) (resp., \(\Perfd_{\graded{G}}^{\wedge I}(R)\)) is the category of gradedwise complete \(G\)-graded adic (resp., \(I\)-adic) perfectoid \(R\)-algebras and \(\Perfd_{\graded{H[1/p]}}^{\wedge I}(R)\) is the full subcategory of \(\Perfd_{\graded{G}}^{\wedge I}(R)\) consisting of all gradedwise complete \(G\)-graded \(I\)-adic perfectoid \(R\)-algebras \(P\) such that \(\Supp(P) = H[1/p]\).
\end{lemma}

\begin{proof}
    Any morphism \(R \to P\) of \(G\)-graded rings to a \(G\)-graded \(I\)-adic perfectoid ring \(P\) factors through \(R \to \bigoplus_{g \in H[1/p]} P_g \hookrightarrow P\).
    Using \Cref{PureSubPerfd}, this \(P' \defeq \bigoplus_{g \in H[1/p]} P_g\) is an \(H[1/p]\)-graded \(I\)-adic perfectoid \(R\)-algebra.
    Therefore, the first functor is cofinal.
    
    Similarly as in \Cref{RemarkCompletePerfd}, any gradedwise complete \(G\)-graded adic perfectoid \(R\)-algebra \(P\) has a morphism from a gradedwise complete \(G\)-graded \(I\)-adic perfectoid \(R\)-algebra \(P'\) whose underlying ring is \(P\).
    So the second functor is also cofinal.
\end{proof}

\begin{lemma} \label{GradedPerfdLimitIsom}
    In the setting of \Cref{GradedPerfdCofinalFunctors}, the canonical morphism
    \begin{equation*}
        \grlim_{P \in \Perfd_{\graded{G}}^{\wedge}(R)} P \to \grlim_{P \in \Perfd_{\graded{G}}^{\wedge I}(R)} \grlim_{n > 0} P/I^nP
    \end{equation*}
    is an isomorphism in \(\CAlg(\mcalD_{\graded{G}}(R))\).
\end{lemma}

\begin{proof}
    This is a direct consequence of \Cref{DerivedGrcompGradedPerfd} as in the proof of \Cref{TopPerfdGradedRing}.
\end{proof}

\begin{definition} \label{DefGradedAbsPerfd}
    Let \(R\) be a \(G\)-graded adic ring with \(G = G[1/p]\) and topologically nilpotent \(p\).
    The \emph{absolute graded perfectoidization} of \(R\) is defined as the limit
    \begin{equation*}
        R_{\tgrpfd} \defeq \grlim_{P \in \Perfd_{\graded{G}}^{\wedge}(R)} P \in \CAlg(\mcalD_{\graded{G}}(R)).
    \end{equation*}
    See \Cref{RemarkGradedPerfd} below for another description of this limit.
    If \(R\) is \(p\)-adic, then we will write this as \(R_{\grpfd}\).
    The existence of this limit will be proved in \Cref{ExistenceOfGradedPerfd} below, and this construction defines a functor from the category of \(G\)-graded adic rings to \(\CAlg(\mcalD_{\graded{G}}(\setZ))\).
\end{definition}

The analogy between \(R_{\tgrpfd}\) and \(R_{\tperfd}\) is not limited to the similarity of their defining limits in \Cref{TopPerfdGradedRing}, but also \Cref{CompletionOfAbsoluteGradedPerfd} below, which informally says \((R_{\tperfd}, R_{\tgrpfd})\) is a ``derived pro-graded ring''.

\begin{remark} \label{RemarkGrading}
    Even if \(G\) does not contain \(1/p\), i.e., \(G \subsetneq G[1/p]\), we can trivially extend the grading of \(R\) to a \(G[1/p]\)-grading by setting \(R_g = 0\) for any \(g \in G[1/p] \setminus G\).
    For such \(R\), the graded absolute perfectoidization \(R_{\tgrpfd}\) in \Cref{DefGradedAbsPerfd} is defined as above using this extended grading.
    Under this definition, we can always assume \(G = G[1/p]\) on considering the graded absolute perfectoidization without loss of generality.
\end{remark}



\begin{remark} \label{RemarkGradedPerfd}
    Because of \Cref{GradedPerfdCofinalFunctors} and \Cref{GradedPerfdLimitIsom}, if \(I\) is a homogeneous ideal of definition of \(R\) containing \(p\), then the absolute graded perfectoidization defined in \Cref{DefGradedAbsPerfd} can be calculated as the limit
    \begin{equation*}
        R_{\tgrpfd} \cong \grlim_{P \in \Perfd_{\graded{H[1/p]}}^{\wedge I}(R)} P \cong \grlim_{P \in \Perfd_{\graded{G}}^{\wedge I}(R)} P \cong \grlim_{P \in \Perfd_{\graded{G}}^{\wedge I}(R)} \grlim_{n > 0} P/I^nP
    \end{equation*}
    in \(\CAlg(\mcalD_{\graded{G}}(R))\).
    In particular, \(R_{\tgrpfd}\) is derived gradedwise \(I\)-complete and \(\Supp(R_{\tgrpfd})\) is equal to \(H[1/p]\) for the monoid \(H\) generated by \(\Supp(R)\).
\end{remark}

\begin{proposition} \label{ComparisonAdicGradedPerfd}
    Let \(R\) be a \(G\)-graded adic ring with \(G = G[1/p]\) and let \(I\) be a homogeneous ideal of definition containing \(p\).
    Then the absolute graded perfectoidization \(R_{\tgrpfd}\) is the derived gradedwise \(I\)-completion of the absolute graded perfectoidization \(R_{\grpfd}\) of the underlying \(G\)-graded \(p\)-adic ring of \(R\).
\end{proposition}

\begin{proof}
    This is proved in essentially the same way as in the non-graded case treated in \Cref{CompletionOfAbsolutePerfectoidization}, by using \Cref{BoundedTorsionGradedPerfd} instead of \Cref{BoundedTorsionPerfd}.
\end{proof}


\begin{proposition} \label{GradedPerfdZeroPart}
    Let \(R\) be a \(G\)-graded adic ring with \(G = G[1/p]\) and topologically nilpotent \(p\).
    Assume that \(R_g = 0\) or \(R_{-g} = 0\) holds for each \(g \in G\).
    Then the degree-zero part
    \[
        R_0 \to (R_{\tgrpfd})_0
    \]
    of the structure morphism \(R \to R_{\tgrpfd}\) can be identified with the topological absolute perfectoidization \(R_0 \to (R_0)_{\tperfd}\) in \(\CAlg_{R_0}\).
\end{proposition}

\begin{proof}
    Note that the functor \(\mcalD_{\graded{G}}(R) \to \mcalD(R_0)\) taking the graded \(0\)-part preserves limits and is symmetric monoidal.
    So, by \Cref{DefGradedAbsPerfd}, the degree-zero part \((R_{\tgrpfd})_0\) is given by
    \begin{equation*}
        (R_{\tgrpfd})_0 = \lim_{R \to P} P_0 \in \CAlg(\mcalD(R_0)) = \CAlg_{R_0},
    \end{equation*}
    where the limit runs through all morphisms \(R \to P\) of adic \(G\)-graded rings such that \(P\) is a gradedwise complete adic \(G\)-graded perfectoid ring.
    The assumption on \(R\) ensures that \(\mfrakm \defeq \bigoplus_{g \neq 0} R_g\) is a homogeneous ideal of \(R\) and we can take the quotient morphism \(R \twoheadrightarrow R/\mfrakm \cong R_0\) which is a morphism of \(G\)-graded adic rings, where \(R_0\) has the trivial grading.
    For any complete adic perfectoid \(R_0\)-algebra \(Q\), the composition \(R \twoheadrightarrow R_0 \to Q\) gives a morphism of \(G\)-graded adic rings
    \begin{equation*}
        R \twoheadrightarrow R_0 \to Q,
    \end{equation*}
    where \(Q\) is the gradedwise complete \(G\)-graded adic perfectoid ring with the trivial grading.
    This shows the isomorphism
    \begin{equation*}
        (R_{\tgrpfd})_0 = \lim_{R \to P} P_0 \xrightarrow{\cong} \lim_{R_0 \to Q} Q = (R_0)_{\tperfd}
    \end{equation*}
    holds in \(\CAlg_{R_0}\).
\end{proof}

\begin{proposition} \label{TopologicalGradedPerfd}
    Let \(R\) be a \(G\)-graded adic ring with \(G = G[1/p]\) and topologically nilpotent \(p\).
    If \(R_{\tperfd}\) is concentrated in degree \(0\), then so is \(R_{\tgrpfd}\).
\end{proposition}

\begin{proof}
    We assume that \(R_{\tperfd}\) is concentrated in degree \(0\).
    For simplicity, we denote the limit \(\lim_{P \in \Perfd_{\graded{G}}^{\wedge I}(R)} \lim_{n > 0} (-)\) by \(\lim_{(P, n)} (-)\).
    Since the limits in \(\mcalD_{\graded{G}}(\setZ)\) are computed gradedwise, the description of \(R_{\tgrpfd}\) in \Cref{RemarkGradedPerfd} gives a sequence of morphisms
    \begin{align*}
        \lim_{(P, n)} (P/I^nP)_g & \to \bigoplus_{g \in G} \lim_{(P, n)} (P/I^nP)_g \to \lim_{(P, n)} \bigoplus_{g \in G} (P/I^nP)_g \\
        & \to \lim_{(P, n)} \prod_{g \in G} (P/I^nP)_g \to \lim_{(P, n)} (P/I^nP)_g
    \end{align*}
    in \(\mcalD(\setZ)\) for any \(g \in G\) such that its composition is the identity morphism. 
    The second term is the same as \(R_{\tgrpfd}\).
    The assumption yields that the third term is concentrated in degree \(0\), and then the first term is also so.
    Therefore, their coproduct \(R_{\tgrpfd}\) is concentrated in degree \(0\).
\end{proof}

\subsection{Absolute perfectoidization of pro-graded rings}

For later convenience, we define the absolute graded perfectoidization of pro-graded rings.

\begin{definition} \label{DefGradedAbsPerfdProGraded}
    Take a pro-\(G\)-graded adic ring \((R, R_{\graded})\) and assume that \(G = G[1/p]\) and \(p\) is topologically nilpotent in \(R_{\graded}\).
    The \emph{graded absolute perfectoidization} is the pair
    \begin{equation*}
        (R_{\tperfd}, R_{\tgrpfd}) \in \CAlg_R \times \CAlg(\mcalD_{\graded{G}}(R_{\graded}))
    \end{equation*}
    for which \(R_{\tperfd}\) is the topological absolute perfectoidization of the adic ring \(R\) defined in \Cref{DefTopologicalAbsolutePerfectoidization} and \(R_{\tgrpfd}\) is the absolute graded perfectoidization of the \(G\)-graded adic ring \(R_{\graded}\) defined in \Cref{DefGradedAbsPerfd}.
    Both objects exist by the preceding constructions.

\end{definition}

\begin{remark} \label{RemarkGradedPerfdProGraded}
    Take a pro-\(G\)-graded adic ring \((R, R_{\graded})\) with \(G = G[1/p]\) and let \(I\) be a homogeneous ideal of definition of \(R_{\graded}\) containing \(p\).
    By \Cref{RemarkCompletePerfd} and \Cref{RemarkGradedPerfd}, the absolute graded perfectoidization \((R_{\tperfd}, R_{\tgrpfd})\) sits in \(\CAlg_R^{\comp{I}} \times \CAlg(\mcalD_{\graded{G}}^{\comp{I}}(R))\).
\end{remark}

\subsection{Representation as a cosimplicial limit}

We start to prove the existence of absolute graded perfectoidization.

\begin{remark}
    Using a method based on our previous paper \cite{ishizuka2026Derived}, we can give another proof of the existence (\Cref{CompletionOfAbsoluteGradedPerfd}). However, the proof below (\Cref{ExistenceOfGradedPerfd}) does not depend on \cite{ishizuka2026Derived}, and gives not only the existence but also a representation as a cosimplicial limit.
\end{remark}

For the semiperfectoid case, we have already proved the existence of graded absolute perfectoidization as follows.

\begin{lemma}[{cf. \cite{ishizuka2025Graded}}] \label{GradedPerfdSemiPerfd}
    Let \(R\) be a \(G\)-graded \(I\)-adic semiperfectoid ring\footnote{A \emph{\(G\)-graded semiperfectoid ring} \(R\) is a \(G\)-graded ring \(R\) which admits a surjective graded ring homomorphism \(P \twoheadrightarrow R\) from a \(G\)-graded perfectoid ring \(P\).} for a finitely generated homogeneous ideal \(I\) of \(R\).
    Then the absolute graded perfectoidization \(R_{\tgrpfd}\) exists and is a \(G\)-graded perfectoid ring.
    In particular, there are isomorphisms
    \[
        \dcomp{I}{R_{\tgrpfd}} \cong \comp{I}{R_{\tgrpfd}} \xrightarrow{\cong} R_{\tperfd}.
    \]
\end{lemma}

\begin{proof}
    Taking the pro-\(G\)-graded \(I\)-adic ring \((\comp{I}{R}, R)\), we can show that this is a pro-\(G\)-graded \(I\)-adic semiperfectoid ring in the sense of \cite{ishizuka2025Graded}*{Definition 4.14}.
    By \cite{ishizuka2025Graded}*{Theorem 4.16}, there exists the initial pro-\(G\)-graded \(I\)-adic perfectoid ring over \((\comp{I}{R}, R)\).
    Because of the categorical equivalence between pro-\(G\)-graded perfectoid rings and \(G\)-graded perfectoid rings (\Cref{CatEquivGrPerfd}), there exists the initial \(G\)-graded \(I\)-adic perfectoid ring \(R'\) over \(R\) and its \(I\)-adic completion \(\comp{I}{R'}\) is the initial \(I\)-adic perfectoid \(R\)-algebra.
    Using the description of \(R_{\tgrpfd}\) and \(R_{\tperfd}\) by \Cref{TopPerfdGradedRing} and \Cref{RemarkGradedPerfd}, we have isomorphisms
    \begin{equation*}
        R_{\tgrpfd} \cong R' \quad \text{and} \quad R_{\tperfd} \cong \comp{I}{R'}.
    \end{equation*}
    In particular, \(R_{\tperfd}\) and \(R_{\tgrpfd}\) exist and are \(I\)-adic perfectoid ring and \(G\)-graded \(I\)-adic perfectoid ring respectively.
    On the last statement, the first isomorphism follows from \Cref{BoundedTorsionGradedPerfd}(4).
\end{proof}

To prove the existence of the absolute graded perfectoidization, we first construct a weakly initial object in the category of \(G\)-graded perfectoid \(R\)-algebras.

\begin{lemma} \label{WeaklyInitialPerfd}
    Let \(R\) be a \(G\)-graded adic ring with \(G = G[1/p]\) and a homogeneous ideal of definition \(I\) containing \(p\).
    Let \(R_{\prism, \graded}^{\perf}\) be the category of all \(G\)-graded \(I\)-adic perfectoid \(R\)-algebras.
    Then this category admits a weakly initial object \(R_{w.i.}\).
\end{lemma}

\begin{proof}
    First we assume that the graded \(0\)-part \(R_0\) of \(R\) has a morphism from a perfectoid ring \(\mcalO\).
    Let \(P\) be a \(G\)-graded \(I\)-adic perfectoid \(R\)-algebra, and let \(x\) be a homogeneous element of \(R\).
    Then there exist homogeneous elements \(y_x\) and \(z_x\) of \(P\) such that \(x = y_x^p + pz_x\) in \(P\) by the graded perfectoidness of \(P\) (\cite{ishizuka2025Graded}*{Proposition 4.4}).
    Define a \(G\)-graded \(R\)-algebra and a morphism of \(G\)-graded \(R\)-algebras
    \begin{equation*}
        R_1 \defeq \grcomp{I}{R[Y_x, Z_x]_{x \in R_{\graded}}/(x-Y_x^p - pZ_x)_{x \in R_{\graded}}} \to P; \quad Y_x \mapsto y_x, \ Z_x \mapsto z_x,
    \end{equation*}
    where \(\deg(Y_x) \defeq \deg(y_x)\) and \(\deg(Z_x) \defeq \deg(z_x)\).
    Especially, any element \(x\) of \(R\) has a \(p\)-th root \(Y_x\) in \(R_1\) modulo \(p\).
    Repeating this construction, we can get a \(G\)-graded \(I\)-adic \(R\)-algebra \(\widetilde{R}\) such that any homogeneous element of \(\widetilde{R}\) has a \(p\)-th root modulo \(p\) and its graded \(0\)-part \(\widetilde{R}_0\) has a morphism from the perfectoid ring \(\mcalO\).
    Then we can take a surjective \(G\)-graded ring homomorphism
    \begin{equation*}
        \grcomp{I}{\mcalO \grotimes_{\setZ_p} A_{\inf}^{\graded}(\widetilde{R})} \twoheadrightarrow \widetilde{R},
    \end{equation*}
    where \(A_{\inf}^{\graded}(\widetilde{R})\) is the graded version of the ring of Witt vectors associated to \(\widetilde{R}^{\flat}\) defined in \cite{ishizuka2025Graded}*{Construction 3.22}.
    As in \cite{bhatt2019Topologicala}*{Remark 4.22}, the source is a \(G\)-graded \(I\)-adic perfectoid ring; hence \(\widetilde{R}\) is a \(G\)-graded \(I\)-adic semiperfectoid \(R\)-algebra admitting a morphism to any \(G\)-graded \(I\)-adic perfectoid \(R\)-algebra \(P\).
    Taking the absolute graded perfectoidization \(\widetilde{R}_{\grpfd}^{\graded}\) of \(\widetilde{R}\) by \Cref{GradedPerfdSemiPerfd}, we get a weakly initial object of \(R_{\prism, \graded}^{\perf}\).

    In the general case, we will prove that there exists a \(G\)-graded \(I\)-adic \(R\)-algebra \(R'\) such that \(R'_0\) admits a morphism from a perfectoid ring and a morphism to any \(G\)-graded \(I\)-adic perfectoid \(R\)-algebra \(P\).
    Once we get such \(R'\), the above construction of \(\widetilde{R'}\) works for \(R'\) and then we are done.

    Take the \(X\)-adic completed perfection
    \begin{equation*}
        C \defeq \comp{X}{\setF_p[X, Y, Y^{-1}, Z_1, Z_2, \dots]_{\perf}}
    \end{equation*}
    and consider an element
    \begin{equation*}
        d_C \defeq [X] - p (Y, Z_1, Z_2, \dots) \in W(C)
    \end{equation*}
    which is a distinguished element (\cite{bhatt2022Prismsa}*{Lemma 2.33}).
    For any perfect prism \((A, (d))\), there exists a morphism of perfect rings
    \begin{equation*}
        C \to A/pA; \quad X \mapsto d_0, \ Y \mapsto u_0, \ Z_i \mapsto u_i,
    \end{equation*}
    where \(d = [d_0] + p(u_0, u_1, u_2, \dots)\) is the Teichm\"uller expansion of \(d\).
    Taking the Witt ring, we have a morphism of perfect prisms \((W(C), (d_C)) \to (A, (d))\).
    Then \(\widetilde{C} \defeq W(C)/d_CW(C)\) is a weakly initial object of the category \((\setZ_p)_{\prism}^{\perf}\) of perfectoid rings.
    Then any \(G\)-graded \(I\)-adic perfectoid \(R\)-algebra \(P\) admits a morphism \(\widetilde{C} \to P_0\) since \(P_0\) is a perfectoid ring by \Cref{graded-perfectoid-equiv}. This gives a morphism
    \begin{equation*}
        R' \defeq \grcomp{I}{\widetilde{C} \grotimes_{\setZ_p} R} \to P
    \end{equation*}
    of \(G\)-graded \(R\)-algebras and therefore \(R'\) is a weakly initial object of \(R_{\prism, \graded}^{\perf}\).
\end{proof}

\begin{proposition} \label{ExistenceOfGradedPerfd}
    Let \(R\) be a \(G\)-graded adic ring with \(G = G[1/p]\) and a homogeneous ideal of definition \(I\) containing \(p\).
    Then the graded absolute perfectoidization \(R_{\tgrpfd}\) exists, i.e., the limit \eqref{DefGradedAbsPerfd} exists and, moreover, \(R_{\tgrpfd}\) can be constructed as a cosimplicial limit
    \begin{equation} \label{SimplicialLimitGradedPerfd}
        R_{\tgrpfd} \xrightarrow{\cong} \grlim_{[n] \in \Delta} (R_{w.i.}^{\cgrotimes_R (n+1)})_{\tgrpfd}
    \end{equation}
    of the absolute graded perfectoidization of graded semiperfectoid rings \(R_{w.i.}^{\cgrotimes_R(n+1)}\) in \(\CAlg(\mcalD_{\graded{G}}(R))\), where \(R_{w.i.}\) is a weakly initial object of \(R^{\perf}_{\prism,\graded}\), whose existence is guaranteed by \Cref{WeaklyInitialPerfd}.
\end{proposition}

\begin{proof}

    Let \(R_{\prism, \graded}^{\perf}\) be the category of all \(G\)-graded \(I\)-adic perfectoid \(R_{\graded}\)-algebras.
    As in the proof of \cite{bhatt2022Prismsa}*{Proposition 8.5}, we pick a weakly initial object \(R_{w.i.}\) of \(R_{\prism, \graded}^{\perf}\) by \Cref{WeaklyInitialPerfd}.
    The gradedwise \(I\)-adically completed tensor product
    \begin{equation*}
        R_{w.i.}^{\cgrotimes_R (n+1)} \defeq \comp{I}{\underbrace{R_{w.i.} \grotimes_R R_{w.i.} \grotimes_R \cdots \grotimes_R R_{w.i.}}_{(n+1)\text{-times}}}
    \end{equation*}
    has a surjective graded morphism
    \begin{equation*}
        \comp{I}{\prod_{i = 0}^n R_{w.i.}} \twoheadrightarrow \comp{I}{R_{w.i.} \grotimes_{\setZ} \cdots \grotimes_{\setZ} R_{w.i.}} \twoheadrightarrow R_{w.i.}^{\cgrotimes_R (n+1)}.
    \end{equation*}
    Since the source is a \(G\)-graded \(I\)-adic perfectoid ring, the target is a \(G\)-graded \(I\)-adic semiperfectoid \(R\)-algebra for any \(n \in \setZ_{\geq 0}\).
    Then its absolute graded perfectoidization \((R_{w.i.}^{\cgrotimes_R (n+1)})_{\tgrpfd}\) exists in \(\CAlg(\mcalD_{\graded{G}}(R))\) as a discrete \(G\)-graded \(R\)-algebra by \Cref{GradedPerfdSemiPerfd}.
    So it suffices to show that the isomorphism \eqref{SimplicialLimitGradedPerfd} holds.

    For any \(P \in R_{\prism, \graded}^{\perf}\), the augmented simplicial anima
    \begin{equation*}
        P_{\bullet} \colon \Delta_+^{\opposite} \to \Ani; \quad [n] \mapsto \Map_{R_{\prism, \graded}^{\perf}}((R_{w.i.}^{\cgrotimes_R (n+1)})_{\tgrpfd}, P)
    \end{equation*}
    admits a splitting because there exists a morphism \(R_{w.i.} \to P\) of \(G\)-graded adic \(R\)-algebras, and then its geometric realization \(\colim_{\Delta^{\opposite}} P_{\bullet} \simeq P_{-1}\) is contractible.
    
    Take the functor
    \begin{equation} \label{DeltaFunctorPerfCat}
        \Delta \to R_{\prism, \graded}^{\perf}; \quad [n] \mapsto (R_{w.i.}^{\cgrotimes_R (n+1)})_{\tgrpfd}
    \end{equation}
    and the simplicial set
    \begin{equation*}
        K_P \defeq \Delta \times_{R_{\prism, \graded}^{\perf}} (R_{\prism, \graded}^{\perf})_{/P}
    \end{equation*}
    for an object \(P \in R_{\prism, \graded}^{\perf}\).
    We will show that \(K_P\) is weakly contractible, namely, its geometric realization is equivalent to a point.
    The canonical morphism \(K_P \to \Delta\) is a right fibration and it corresponds to the functor \(P_{\bullet}\) above by the straightening-unstraightening equivalence (e.g., \cite{lurie2009Higher}*{Theorem 3.2.0.1}).
    Applying \cite{lurie2009Higher}*{Corollary 3.3.4.6} for the left fibration \(K_P^{\opposite} \to \Delta^{\opposite}\), we can show that
    \begin{equation*}
        \abs{K_P} \simeq \abs{K_P^{\opposite}} \simeq \colim_{\Delta^{\opposite}} P_{\bullet} \simeq \ast \in \Ani
    \end{equation*}
    by the computation of the colimit above.
    This tells us that the simplicial set \(K_P\) is weakly contractible for any \(P \in R_{\prism, \graded}^{\perf}\).

    By \citeKero{02NY}, the functor \eqref{DeltaFunctorPerfCat} is left cofinal.
    So we can show that
    \begin{equation*}
        R_{\prism, \graded}^{\perf} \to \CAlg(\mcalD_{\graded{G}}(R)); \quad P \mapsto P
    \end{equation*}
    has a limit if and only if the composition
    \begin{equation*}
        \Delta \to R_{\prism, \graded}^{\perf} \to \CAlg(\mcalD_{\graded{G}}(R)); \quad [n] \mapsto (R_{w.i.}^{\cgrotimes_R (n+1)})_{\tgrpfd}
    \end{equation*}
    does and in this case, their limits are naturally isomorphic (see \citeKero{02NS} and \citeKero{02XU}).
    Therefore, \(R_{\tgrpfd}\) can be constructed as a cosimplicial limit
    \begin{equation*}
        R_{\tgrpfd} = \grlim_{P_{\graded} \in R_{\prism, \graded}^{\perf}} P_{\graded} \cong \grlim_{[n] \in \Delta} (R_{w.i.}^{\cgrotimes_R (n+1)})_{\tgrpfd}
    \end{equation*}
    in \(\CAlg(\mcalD_{\graded{G}}(R_{\graded}))\).
    \qedhere
\end{proof}

\begin{corollary} \label{RepresentationPerfd}
    Keep the notation in \Cref{ExistenceOfGradedPerfd}.
    Then the topological absolute perfectoidization \(R_{\tperfd}\) can be written as a cosimplicial limit
    \begin{equation*}
        R_{\tperfd} \xrightarrow{\cong} \lim_{[n] \in \Delta} \dcomp{I}{(R_{w.i.}^{\cgrotimes_R(n+1)})_{\tgrpfd}}
    \end{equation*}
    in \(\CAlg_R\).
\end{corollary}

\begin{proof}
    Note that any weakly initial object \(R_{w.i.}\) of \(R_{\prism,\graded}^{\perf}\) gives rise to a weakly initial object \(\comp{I}{R_{w.i.}}\) of the category \(R_{\prism}^{\perf}\) of all \(I\)-adic perfectoid \(R\)-algebras by \Cref{CofinalGradedPerfdForGradedRing}.
    The same proof of \Cref{ExistenceOfGradedPerfd} shows that \(R_{\tperfd}\) can be written as a cosimplicial limit
    \begin{equation*}
        R_{\tperfd} \xrightarrow{\cong} \lim_{[n] \in \Delta} (\comp{I}{R_{w.i.}}^{\widehat{\otimes}_R(n+1)})_{\tperfd}
    \end{equation*}
    in \(\CAlg_R\).
    Each \(\comp{I}{R_{w.i.}}^{\widehat{\otimes}_R(n+1)}\) is the \(I\)-adic completion of the graded semiperfectoid ring \(R_{w.i.}^{\cgrotimes_R(n+1)}\) and then we have an isomorphism
    \begin{equation*}
        R_{\tperfd} \xrightarrow{\cong} \lim_{[n] \in \Delta} (R_{w.i.}^{\cgrotimes_R(n+1)})_{\tperfd}.
    \end{equation*}
    Because of \Cref{GradedPerfdSemiPerfd}, each component \((R_{w.i.}^{\cgrotimes_R(n+1)})_{\tperfd}\) of the cosimplicial diagram is the derived \(I\)-completion of \((R_{w.i.}^{\cgrotimes_R(n+1)})_{\tgrpfd}\). This completes the proof.
\end{proof}



\subsection{Completion of the absolute graded perfectoidization}


In this subsection, we prove the isomorphism \(\dcomp{I}{R_{\tgrpfd}} \xrightarrow{\cong} R_{\tperfd}\) (\Cref{CompletionOfAbsoluteGradedPerfd}).
Our proof is based on a reduction to the non-graded case and on the formal action of the group scheme associated with the grading.
This gives another proof of the existence of absolute graded perfectoidization.




\begin{lemma} \label{RelativelyPerfectGroupRing}
    Let \(R\) be a \(p\)-adic ring with topologically nilpotent \(p\) and let \(G\) be a torsion-free abelian group with \(G = G[1/p]\).
    Take the group ring \(R[G] \defeq \bigoplus_{g \in G} Rt^g\) and its \(p\)-adic completion \(R\abracket{G}\).
    Then the canonical morphism
    \begin{equation*}
        R \to R\abracket{G}; \quad r \mapsto r t^0
    \end{equation*}
    becomes relatively perfect after reduction modulo \(p\) in the sense of \cite{bhatt2024Perfectoid}*{Proposition 3.18}.
    Especially, the canonical morphism
    \begin{equation} \label{RelativelyPerfectGroupRingIso}
        R_{\perfd} \widehat{\otimes}^L_R R\abracket{G} \to (R\abracket{G})_{\perfd}
    \end{equation}
    is an isomorphism in \(\CAlg_{R\abracket{G}}^{\comp{p}}\).
\end{lemma}

\begin{proof}
    The relative Frobenius morphism associated with the reduction of the above morphism modulo \(p\) is given by
    \begin{equation*}
        R\abracket{G}/^L p \otimes^L_{R/^L p} F_* (R/^L p) \to F_*(R\abracket{G}/^L p); \quad at^g \otimes F_*r \mapsto F_* r a^p t^{pg}.
    \end{equation*}
    Since the left-hand side is isomorphic to \(\bigoplus_{g \in G} F_*(R/^L p) t^g\) and the right-hand side is isomorphic to \(\bigoplus_{g \in G} F_*(R/^L p) t^{g/p}\) as \(F_*(R/^L p)\)-modules, this morphism is the same as the canonical inclusion
    \begin{equation*}
        \bigoplus_{g \in G} F_*(R/^L p) t^g \to \bigoplus_{g \in G} F_*(R/^L p) t^{g/p}; \quad at^g \mapsto at^g.
    \end{equation*}
    Because of \(G = G[1/p]\), this is an isomorphism.
    The last assertion follows from \cite{bhatt2024Perfectoid}*{Proposition 3.18}.
\end{proof}

\begin{lemma} \label{LimitAddG}
    Let \(R\) be a \(G\)-graded adic ring with \(G = G[1/p]\) and let \(I\) be a homogeneous ideal of definition containing \(p\).
    Then the canonical morphism
    \begin{equation*}
        (R[G])_{\tperfd} \to \lim_{R \to P} (\dcomp{I}{P[G]})
    \end{equation*}
    in \(\CAlg_R\) is an isomorphism where the limit is taken over all \(I\)-adic \(G\)-graded perfectoid \(R\)-algebras \(P\).
\end{lemma}

\begin{proof}
    Take any \(I\)-adic perfectoid \(R[G]\)-algebra \(Q\). Then there exists an \(I\)-adic \(G\)-graded perfectoid \(R\)-algebra \(P\) and a morphism \(P \to Q\) of \(R\)-algebras with a commutative diagram of \(R\)-algebras, as follows, by \Cref{CofinalGradedPerfdForGradedRing}:
    \begin{center}
        \begin{tikzcd}
            R \arrow[d] \arrow[r] & \exists P \arrow[d, "\exists"'] \arrow[r] & {P[G]} \arrow[ld, "\exists"] \\
            {R[G]} \arrow[r]      & Q                                         &                             
        \end{tikzcd}
    \end{center}
    where the morphism \(P[G] \to Q\) of \(R[G]\)-algebras sends \(t^g\) in \(P[G]\) to the image of \(t^g\) along the structure morphism \(R[G] \to Q\).

    Taking the \(I\)-adic completion \(\comp{I}{P[G]}\) of \(P[G]\), which becomes an \(I\)-adic perfectoid \(R[G]\)-algebra, we can show that such \(\comp{I}{P[G]}\) form a cofinal system in all \(I\)-adic perfectoid \(R[G]\)-algebras.
    Therefore, the canonical morphism is an isomorphism.
\end{proof}

Using this, we can show the following result as a counterpart of \Cref{CompletionOfAbsolutePerfectoidization} in the graded setting.

\begin{corollary} \label{CompletionOfAbsoluteGradedPerfd}
    Let \(R\) be a \(G\)-graded adic ring such that \(G = G[1/p]\) and let \(I\) be a homogeneous ideal of definition containing \(p\).
    Then \(R_{\tgrpfd}\) exists and the canonical morphism \(R_{\tgrpfd} \to R_{\tperfd}\) induces an isomorphism
    \begin{equation}
        \dcomp{I}{R_{\tgrpfd}} \xrightarrow{\cong} R_{\tperfd}
    \end{equation}
    in \(\CAlg_R\).
\end{corollary}

\begin{proof}
    Take the diagram
    \begin{equation*}
        \Perfd_{\graded{G}}^{\wedge I}(R) \to \mcalD_{\graded{G}}^{\comp{I}}(R); \quad P \mapsto P
    \end{equation*}
    of all gradedwise complete \(G\)-graded \(I\)-adic perfectoid \(R\)-algebras \(P\).
    Consider the canonical morphism
    \begin{equation*}
        (\lim_{R \to P} \dcomp{I}{P})\abracket{G} \to \lim_{R \to P} (\dcomp{I}{P}\abracket{G})
    \end{equation*}
    in \(\mcalD^{\comp{I}}(R)\), where the left-hand side exhibits \(R_{\tperfd}\abracket{G}\) by \Cref{TopPerfdGradedRing} and the right-hand side exhibits \((R[G])_{\tperfd}\) by \Cref{LimitAddG}.
    This morphism is an isomorphism by \Cref{RelativelyPerfectGroupRing}.

    Therefore, we may apply \Cref{GradedLimitCommutative} to the diagram above to obtain the existence of the limit
    \begin{equation*}
        R_{\tgrpfd} = \grlim_{R \to P} P
    \end{equation*}
    in \(\mcalD_{\graded{G}}^{\comp{I}}(R)\) and the natural isomorphism
    \begin{equation*}
        \dcomp{I}{R_{\tgrpfd}} \cong \dcomp{I}{\grlim_{R \to P} P} \xrightarrow{\cong} \lim_{R \to P} \dcomp{I}{P} \cong R_{\tperfd}
    \end{equation*}
    in \(\mcalD^{\comp{I}}(R)\).
    Since all these morphisms are in \(\CAlg_R\), we conclude the desired isomorphism.
\end{proof}

\begin{corollary} \label{DiscretePropertyGradedPerfd}
    Let \(R\) be a \(G\)-graded adic ring such that \(G = G[1/p]\) and let \(I\) be a homogeneous ideal of definition containing \(p\).
    Assume that \(R\) admits a (not necessarily graded) morphism from a perfectoid ring.
    If \(R_{\perfd}\) is concentrated in degree \(0\), then \(R_{\tgrpfd}\) is also concentrated in degree \(0\), and the morphism \(R \to R_{\tgrpfd}\) is initial among morphisms from \(R\) to \(G\)-graded perfectoid rings.
\end{corollary}

\begin{proof}
    It suffices to show that \(R_{\tgrpfd}\) is a \(G\)-graded perfectoid ring.
    By \Cref{DiscretePropertyTopAbsPerfd} and \Cref{TopologicalGradedPerfd}, we can show that \(R_{\tgrpfd}\) is concentrated in degree \(0\).
    Moreover, \Cref{CompletionOfAbsoluteGradedPerfd} gives the isomorphism \(\dcomp{I}{R_{\tgrpfd}} \cong R_{\tperfd}\).
    Taking the \(I\)-adic completion, we have \(\comp{I}{R_{\tgrpfd}} \cong R_{\tperfd}\).

    In the case of \(I = (p)\), this isomorphism implies that \(R_{\grpfd}\) is a \(G\)-graded perfectoid ring by \Cref{DefGradedPerfd}.
    The general case follows from the case of \(I = (p)\) by \Cref{ComparisonAdicGradedPerfd} and \Cref{BoundedTorsionGradedPerfd}.
\end{proof}

\subsection{Colimit preservation of absolute graded perfectoidization}

Using \Cref{CompletionOfAbsoluteGradedPerfd}, we first show that the absolute graded perfectoidization functor preserves small colimits if it has a base perfectoid ring.

    

\begin{corollary} \label{ColimitPreserveGradedPerfd}
    Let \(R\) be a gradedwise complete \(G\)-graded adic ring with \(G = G[1/p]\) and let \(I\) be a homogeneous ideal of definition containing \(p\).
    Let \(\CRing_{\graded{G}}^{\wedge I}(R)\) be the category of gradedwise complete \(G\)-graded \(I\)-adic \(R\)-algebras.
    
    Suppose that there exists a morphism from a \(G\)-graded perfectoid ring to \(R\).
    Then the absolute graded perfectoidization functor
    \begin{equation*}
        (-)_{\tgrpfd} \colon \CRing_{\graded{G}}^{\wedge I}(R) \to \CAlg(\mcalD_{\graded{G}}^{\comp{I}}(R)); \quad S \mapsto S_{\tgrpfd}
    \end{equation*}
    preserves small colimits.
    Especially, \((-)_{\tgrpfd}\) is symmetric monoidal with respect to the gradedwise completed tensor products \(- \cgrotimes -\) and \(- \cLgrotimes -\).
\end{corollary}

\begin{proof}
    Take a small diagram \(\{S_k\}_{k \in K} \colon K \to \CRing_{\graded{G}}^{\wedge I}(R)\) defined on a small simplicial set \(K\).
    It gives a universal morphism
    \begin{equation*}
        \dgrcomp{I}{\grcolim_{k \in K} ((S_k)_{\tgrpfd})} \to (\grcolim_{k \in K} S_k)_{\tgrpfd}
    \end{equation*}
    in \(\CAlg(\mcalD_{\graded{G}}^{\comp{I}}(R))\).
    Forgetting the graded structures and taking the derived \(I\)-completion, this morphism induces a morphism
    \begin{equation*}
        \dcomp{I}{\colim_{k \in K} ((S_k)_{\tgrpfd})} \to \dcomp{I}{(\grcolim_{k \in K} S_k)_{\tgrpfd}}
    \end{equation*}
    in \(\CAlg_R^{\comp{I}}\).
    Using \Cref{CompletionOfAbsoluteGradedPerfd}, this morphism can be identified with the morphism
    \begin{equation*}
        \dcomp{I}{\colim_{k \in K} ((S_k)_{\tperfd})} \to (\grcolim_{k \in K} S_k)_{\tperfd} \cong (\colim_{k \in K} S_k)_{\tperfd}
    \end{equation*}
    in \(\CAlg_R^{\comp{I}}\), which is an isomorphism by \Cref{CommutativeTensorTopPerfd}, since we assumed the existence of a morphism from a perfectoid ring to \(R\).
    Therefore, the original morphism is also an isomorphism by the Nakayama-type lemma (\Cref{GradedNakayamaLemma}).
\end{proof}

Although the existence of a morphism from a graded perfectoid ring is assumed in \Cref{ColimitPreserveGradedPerfd}, we can remove this assumption for localization (\Cref{CommutativityLimitsPerfd}).
This follows from more general results regarding the base change of absolute graded perfectoidization along relatively perfect morphisms. We first show this.

\begin{proposition}[{cf. \cite{bhatt2024Perfectoid}*{Proposition 3.17}}] \label{PullbackSquareGraded}
    Let \((V, \mfrakm, k) \hookrightarrow (W, \mfrakn, k_{\perf})\) be a faithfully flat extension of discrete valuation rings such that \(\mfrakm\) contains \(p\). 
    Let \(\mcalO_C\) be the \(p\)-completed ring of integers of a perfectoid field \(C\) which is the \(p\)-completion of a totally ramified Galois extension of \(\Frac(W)\) with Galois group \(\Gamma = \setZ_p\).
    Take any \(G\)-graded adic \(V\)-algebra \(R\) with \(G = G[1/p]\) and a homogeneous ideal of definition \(I\) containing \(p\).
    Then we have a pullback diagram
    \begin{equation} \label{GradedPerfdFiberPullbackDiagram}
        \begin{tikzcd}
            {R_{\tgrpfd}} \arrow[d] \arrow[rr] &  & {\fib((R \cLgrotimes_{V} \mcalO_{C})_{\tgrpfd} \xrightarrow{\gamma - 1} (R \cLgrotimes_{V} \mcalO_{C})_{\tgrpfd})} \arrow[d] \\
            \grcomp{I}{(R/pR)_{\perf}} \arrow[rr]         &  & {\grcomp{I}{(R/pR)_{\perf}} \bigoplus \grcomp{I}{(R/pR)_{\perf}}[-1]}                                                                                         
        \end{tikzcd}
    \end{equation}
    in \(\mcalD_{\graded{G}}^{\comp{I}}(R)\), where \(\gamma \in \setZ_p\) is a generator and the morphisms are the canonical morphisms induced from
    \begin{align*}
        R_{\tgrpfd} & \to (R \cLgrotimes_{V} \mcalO_C)_{\tgrpfd}, \\
        R_{\tgrpfd} & \to (R/pR)_{\tgrpfd} \cong \grcomp{I}{(R/pR)_{\perf}} \\
        (R \cLgrotimes_{V} \mcalO_C)_{\tgrpfd} & \to ((R/pR) \cLgrotimes_{V} k_{\perf})_{\tgrpfd} \cong \grcomp{I}{(R/pR)_{\perf}}.
    \end{align*}
    Note that the perfection \((R/pR)_{\perf}\) has the natural \(G\)-grading and the isomorphism \(\dcomp{I}{(R/pR)_{\perf}} \cong \comp{I}{(R/pR)_{\perf}}\) holds by \Cref{BoundedTorsionGradedPerfd}.
\end{proposition}

\begin{proof}
    By \cite{bhatt2025Aspects}*{Proposition 4.5.1 and Example 4.5.2} (or \cite{bhatt2024Perfectoid}*{Proposition 3.17} if \(V = \setZ_p\)), we have a pullback diagram
    \begin{center}
        \begin{tikzcd}
            {R_{\tperfd}} \arrow[d] \arrow[rr] &  & {\fib((R \widehat{\otimes}^L_{V} \mcalO_{C})_{\tperfd} \xrightarrow{\gamma - 1} (R \widehat{\otimes}^L_{V} \mcalO_{C})_{\tperfd})} \arrow[d] \\
            \comp{I}{(R/pR)_{\perf}} \arrow[rr]         &  & {\comp{I}{(R/pR)_{\perf}} \bigoplus \comp{I}{(R/pR)_{\perf}}[-1]}                                                                                         
        \end{tikzcd}
    \end{center}
    in \(\mcalD^{\comp{I}}(R)\).
    Since the derived \(I\)-completion of the given diagram \eqref{GradedPerfdFiberPullbackDiagram} is the same as above diagram, it becomes a pullback diagram in \(\mcalD_{\graded{G}}^{\comp{I}}(R)\) by \Cref{GradedNakayamaLemma}.
\end{proof}

\begin{proposition}[{cf. \cite{bhatt2024Perfectoid}*{Proposition 3.18}}] \label{BaseChangeGradedPerfd}
    Let \(R \to S\) be a morphism of \(G\)-graded adic rings with \(G = G[1/p]\) such that the induced morphism \(R/^L p \to S/^L p\) of animated rings is relatively perfect, i.e., its relative Frobenius morphism is an isomorphism.
    Then the canonical morphism
    \begin{equation*}
        S \cLgrotimes_{R} R_{\tgrpfd} \to S_{\tgrpfd}
    \end{equation*}
    is an isomorphism in \(\CAlg(\mcalD_{\graded{G}}(S))\).
\end{proposition}

\begin{proof}
    Taking the derived \(I\)-completion of the underlying morphism in \(\CAlg_S\) and using \Cref{CompletionOfAbsoluteGradedPerfd}, we obtain an isomorphism
    \begin{equation*}
        S \widehat{\otimes}^L_R R_{\tperfd} \to S_{\tperfd}
    \end{equation*}
    in \(\CAlg_S\) by \cite{bhatt2024Perfectoid}*{Proposition 3.18} and \Cref{CompletionOfAbsolutePerfectoidization}.
    So \Cref{GradedNakayamaLemma} implies that the original morphism is also an isomorphism.
\end{proof}

\begin{corollary}[{cf. \cite{bhatt2024Perfectoid}*{Lemma 3.21}}] \label{CommutativityLimitsPerfd}
    Let \(R\) be a \(G\)-graded adic ring with \(G = G[1/p]\) and let \(I\) be a homogeneous ideal of definition containing \(p\).
    Take a homogeneous element \(f\) of \(R\).
    Then we get a canonical isomorphism
    \begin{equation*}
        \dgrcomp{I}{R_{\tgrpfd}[1/f]} \xrightarrow{\cong} (R[1/f])_{\tgrpfd}
    \end{equation*}
    in \(\CAlg(\mcalD_{\graded{G}}(R_{\graded}))\).
    Moreover, we have the following commutativity of limits:
    \begin{equation} \label{CommutesLimitIsom}
        \dgrcomp{I}{R_{\tgrpfd}[1/f]} = \dgrcomp{I}{(\grlim_{P \in \Perfd_{\graded{G}}^{\wedge I}(R)} P)[1/f]} \xrightarrow{\cong} \grlim_{P \in \Perfd_{\graded{G}}^{\wedge I}(R)}(\dgrcomp{I}{P[1/f]})
    \end{equation}
    in \(\CAlg(\mcalD_{\graded{G}}(R))\), where the limits run over all \(I\)-adic \(G\)-graded perfectoid \(R\)-algebras \(P\).
    In general, the same isomorphisms hold if we take the localization \(W^{-1}R\) by any multiplicative subset \(W\) of homogeneous elements of \(R\).
\end{corollary}

\begin{proof}
    Since the localization \(R \to R[1/f]\) is relatively perfect modulo \(p\), the assertion follows from \Cref{BaseChangeGradedPerfd} above.
    The second isomorphism \eqref{CommutesLimitIsom} follows from the first isomorphism and the same argument in the proof of \Cref{CommutativeLocalizationTopPerfd}.
\end{proof}

\subsection{Descendability of absolute graded perfectoidization}

In non-graded cases, the absolute perfectoidization produces a descendable morphism:

\begin{lemma}[{cf. \cite{bhatt2022Prismsa}*{Lemma 8.6}}] \label{PerfdDescendableNonGraded}
    Let \(R\) be an \(I\)-adic perfectoid ring.
    Then the canonical morphism of commutative algebra objects
    \begin{equation*}
        (\comp{I}{R[X_1, \dots, X_n]})_{\tperfd} \to \comp{I}{R[X_1^{1/p^\infty}, \dots, X_n^{1/p^\infty}]}
    \end{equation*}
    is descendable in \(\mcalD^{\comp{I}}(R)\).
\end{lemma}

\begin{proof}
    Since the morphism \((\comp{p}{R[\underline{T}]})_{\perfd} \to \comp{p}{R[\underline{T}^{1/p^\infty}]}\) is descendable in \(\mcalD^{\comp{p}}(R)\) by \cite{bhatt2022Prismsa}*{Lemma 8.6}, taking the derived \(I\)-completion gives the desired result by \cite{mathew2016Galois}*{Corollary 3.21}.
\end{proof}

Using this descendability result in non-graded cases and the comparison between graded and non-graded absolute perfectoidizations (\Cref{CompletionOfAbsoluteGradedPerfd}), we can show the following graded version.

\begin{corollary} \label{GradedPerfdDescendable}
    Let \(R\) be a \(G\)-graded adic perfectoid ring with \(G = G[1/p]\) and let \(I\) be a homogeneous ideal of definition containing \(p\).
    Equip a polynomial ring \(R[X_1, \dots, X_n]\) with a \(G\)-grading such that the canonical inclusion \(R \to R[\underline{X}]\) is a morphism of \(G\)-graded rings, and equip \(R[X_1^{1/p^\infty}, \dots, X_n^{1/p^\infty}]\) with the induced \(G\)-grading.\footnote{For example, we can consider two typical cases: (1) each \(X_i\) is graded in degree \(1\), which gives the total degree on \(R[\underline{X}]\); (2) replacing \(G\) with the product \(G^n\) and considering the multidegree on \(R[\underline{X}]\), i.e., \(X_i\) is graded in the degree having \(1\) only in the \(i\)-th component and \(0\) elsewhere.}
    Then the canonical morphism of commutative algebra objects
    \begin{equation*}
        (\grcomp{I}{R[X_1, \dots, X_n]})_{\tgrpfd} \to \grcomp{I}{R[X_1^{1/p^\infty}, \dots, X_n^{1/p^\infty}]}
    \end{equation*}
    is descendable in \(\mcalD_{\graded{G}}^{\comp{I}}(R)\).
\end{corollary}

\begin{proof}
    For simplicity, set \(T \defeq R[X_1, \dots, X_n]\) and \(T_{\infty} \defeq R[X_1^{1/p^\infty}, \dots, X_n^{1/p^\infty}]\).
    It suffices to show that the canonical morphism
    \begin{equation*}
        \{(\grcomp{I}{T})_{\tgrpfd}\}_{n \geq 0} \to \{\Tot^{\graded}_{\leq n} \cech^{\bullet}(\grcomp{I}{T_{\infty}}/(\grcomp{I}{T})_{\tgrpfd})\}_{n \geq 0}
    \end{equation*}
    is an isomorphism in the pro-category \(\Pro(\mcalD_{\graded{G}}^{\comp{I}}(R))\), where \(\cech^{\bullet}(\grcomp{I}{T_{\infty}}/(\grcomp{I}{T})_{\tgrpfd})\) is the \v{C}ech nerve of the morphism \((\grcomp{I}{T})_{\tgrpfd} \to \grcomp{I}{T_{\infty}}\) in \(\mcalD_{\graded{G}}^{\comp{I}}(R)\) and \(\Tot^{\graded}_{\leq n}(-)\) is the finite limit
    \begin{equation*}
        \Tot^{\graded}_{\leq n}(-) \defeq \grlim_{\Delta_{\leq n}} (-)
    \end{equation*}
    in \(\mcalD_{\graded{G}}^{\comp{I}}(R)\) over the truncated simplex category \(\Delta_{\leq n}\).

    Sending the above morphism to \(\Pro(\mcalD^{\comp{I}}(R))\) via the conservative functor \(\Pro(\mcalF^I)\) in \Cref{GradedNakayamaLemma}, it suffices to show that the induced morphism
    \begin{equation} \label{CompletionGradedTot}
        \{(\comp{I}{T})_{\tperfd}\}_{n \geq 0} \to \{\dcomp{I}{\Tot^{\graded}_{\leq n} \cech^{\bullet}(\grcomp{I}{T_{\infty}}/(\grcomp{I}{T})_{\tgrpfd})}\}_{n \geq 0}
    \end{equation}
    is an isomorphism in \(\Pro(\mcalD^{\comp{I}}(R))\) by \Cref{CompletionOfAbsoluteGradedPerfd}.

    Since the forgetful functor \(\mcalD_{\graded{G}}^{\comp{I}}(R) \to \mcalD(R)\) is exact, it commutes with finite limits.
    So each term of the target of \eqref{CompletionGradedTot} can be identified with
    \begin{equation*}
        \Tot_{\leq n} \dcomp{I}{\cech^{\bullet}(\grcomp{I}{T_{\infty}}/(\grcomp{I}{T})_{\tgrpfd})} = \Tot_{\leq n} \cech^{\bullet}(\comp{I}{T_{\infty}}/(\comp{I}{T})_{\tperfd})
    \end{equation*}
    in \(\mcalD^{\comp{I}}(R)\) because of \Cref{CompletionOfAbsoluteGradedPerfd} again, where the \v{C}ech nerve on the right-hand side is taken in \(\mcalD^{\comp{I}}(R)\).
    Therefore, the morphism \eqref{CompletionGradedTot} is the morphism induced from the descendable morphism \((\comp{I}{T})_{\tperfd} \to \comp{I}{T_{\infty}}\) in \(\mcalD^{\comp{I}}(R)\) by \Cref{PerfdDescendableNonGraded}.
    This completes the proof.
\end{proof}

\section{Absolute perfectoidization of structure sheaves} \label{SectionGlobalAbsPerfd}

\subsection{Derived categories and derived sheaves}


To compare objects in the derived \(\infty\)-category of \(\mcalO_{\mscrX}\)-modules, we often consider their derived global sections.
In the higher-categorical setting, this is justified by the following result.

\begin{proposition} \label{SheafifiedDerivedObjects}
    Let \(\mscrX\) be an adic quasi-compact separated formal scheme such that \(p\) is topologically nilpotent in \(\mcalO_{\mscrX}\).
    Take an object \(\mcalF\) of \(\mcalD(\mscrX_{\Zar}, \mcalO_{\mscrX})\).
    Then the assignment
    \begin{align*}
        R\Gamma(\mcalF) \colon \mscrX_{\Zar} & \to \mcalD(\setZ) \\
        \mscrU & \mapsto R\Gamma(\mscrU, \mcalF)
    \end{align*}
    is a \(\mcalD(\setZ)\)-valued \(R\Gamma(\mcalO_{\mscrX})\)-module sheaf 
    in the sense of \cite{lurie2018Spectral}*{Definition 1.3.1.1}, where \(R\Gamma(U,\mcalF)\) denotes the derived global sections of \(\mcalF\) on \(U\) (\citeSta{07A5}).
\end{proposition}

\begin{proof}
    By \cite{lurie2018Spectral}*{Proposition A.3.3.1}, it suffices to show that the \(\mcalD(\setZ)\)-valued presheaf \(R\Gamma(\mcalF)\) preserves finite products and satisfies \v{C}ech descent.
    First note that the sheaf cohomology \(R\Gamma(\mscrU, \mcalF)\) for any \(\mscrU \in \mscrX_{\Zar}\) can be identified with \(R\Hom_{\Shv(\mscrX_{\Zar}, \Ab)}(\setZ_{h_{\mscrU}}^{\sharp}, \mcalF)\), where \(h_{\mscrU} \colon \mscrX_{\Zar} \to \Set\) is the \(\Set\)-valued presheaf represented by \(\mscrU\) and \(\setZ_{h_{\mscrU}}^{\sharp}\) is the sheafification of the free abelian sheaf generated by \(h_{\mscrU}\) (\citeSta{01GA}).

    The preservation of finite products follows from this representation of \(R\Gamma(\mscrU, \mcalF)\).

    For \v{C}ech descent, it suffices to show that, for any finite open covering \(\{\mscrU_i\}_{i\in I}\) of \(\mscrX\), the \v{C}ech nerve \(\mscrU^{\bullet}\to\mscrX\) of \(\mscrU \defeq \bigcup_{i \in I} \mscrU_i \to \mscrX\) induces an isomorphism
    \begin{equation*}
        R\Gamma(\mscrX, \mcalF) \xrightarrow{\cong} \lim_{\Delta} R\Gamma(\mscrU^{\bullet}, \mcalF)
    \end{equation*}
    in \(\mcalD(\setZ)\).
    Using the \v{C}ech nerve \(\mscrU^{\bullet}\to\mscrX\), we obtain the associated simplicial presheaf \(F(\mscrU^{\bullet})\) as in \citeSta{01G1}.
    As in \citeSta{01GA}, take the free abelian sheaf \(\setZ_{F(\mscrU^{\bullet})}^{\sharp}\) of abelian sheaves on \(\mscrX_{\Zar}\), whose value at \([n] \in \Delta^{\opposite}\) is given by the sheafification of the free abelian presheaf \(\setZ_{h_{\mscrU^{\times n}}}\) of the presheaf \(h_{\mscrU^{\times n}}\) represented by \(\mscrU^{\times n}\).
    Then \citeSta{01GF} shows that the canonical morphism
    \begin{equation*}
        M_*(\setZ_{h_{\mscrU^{\bullet}}}^{\sharp}) \to \setZ_{h_{\mscrX}}^{\sharp}
    \end{equation*}
    is an isomorphism in \(\mcalD(\setZ)\), where \(M_*(\setZ_{h_{\mscrU^{\bullet}}}^{\sharp})\) is the associated complex of abelian sheaves of the simplicial one.
    The left-hand side can be identified with the simplicial colimit \(\colim_{\Delta^{\opposite}} \setZ_{h_{\mscrU^{\bullet}}}^{\sharp}\) (see, for example, \cite{bunke2013Differential}*{Problem 4.24} and \cite{arakawa2025Homotopy}*{Proposition 2.4 (1)}).
    Using this equivalence, we have isomorphisms
    \begin{align*}
        R\Gamma(\mscrX, \mcalF) & \cong R\Hom_{\Shv(\mscrX_{\Zar}, \Ab)}(\setZ_{h_{\mscrX}}^{\sharp}, \mcalF) \cong R\Hom(\colim_{\Delta^{\opposite}} \setZ_{h_{\mscrU^{\bullet}}}^{\sharp}, \mcalF) \\
        & \cong \lim_{\Delta} R\Hom(\setZ_{h_{\mscrU^{\bullet}}}^{\sharp}, \mcalF) \cong \lim_{\Delta} R\Gamma(\mscrU^{\bullet}, \mcalF)
    \end{align*}
    in \(\mcalD(\setZ)\).
    This shows the \v{C}ech descent.
\end{proof}

\begin{lemma} \label{RGammaCommutesLimits}
    In the setting of \Cref{SheafifiedDerivedObjects}, the functor
    \begin{equation*}
        R\Gamma \colon \mcalD(\mscrX_{\Zar}, \mcalO_{\mscrX}) \to \Mod_{R\Gamma(\mcalO_{\mscrX})}(\Shv(\mscrX_{\Zar}, \mcalD(\setZ))); \quad \mcalF \mapsto R\Gamma(\mcalF)
    \end{equation*}
    commutes with limits and is fully faithful.
\end{lemma}

\begin{proof}
    Since limits in the \(\infty\)-category \(\Shv(\mscrX_{\Zar}, \mcalD(\setZ))\) are computed sectionwise, it suffices to show that \(R\Gamma(\mscrU, -)\) commutes with limits for any \(\mscrU \in \mscrX_{\Zar}\). This follows from the fact that \(R\Gamma(\mscrU, -)\) can be realized as \(R\Hom_{\Shv(\mscrX_{\Zar}, \Ab)}(\setZ^{\sharp}_{h_{\mscrU}}, -)\) as in the proof of \Cref{SheafifiedDerivedObjects}.

    By forgetting the \(\mcalO_{\mscrX}\)-module structures, it suffices to show that the functor
    \begin{equation*}
        R\Gamma \colon \mcalD(\mscrX_{\Zar}, \Ab) \to \PSh(\mscrX_{\Zar}, \mcalD(\setZ)); \quad \mcalF \mapsto R\Gamma(\mcalF)
    \end{equation*}
    is fully faithful.
    Since the inclusion functor \(\iota \colon \Shv(\mscrX_{\Zar}, \Ab) \hookrightarrow \PSh(\mscrX_{\Zar}, \Ab)\) is fully faithful and admits the sheafification functor \(a\) as a left adjoint, and they are exact, the induced functor
    \begin{equation} \label{SheafToPresheaf}
        \iota \colon \mcalD(\mscrX_{\Zar}, \Ab) = \mcalD(\Shv(\mscrX_{\Zar}, \Ab)) \hookrightarrow \mcalD(\PSh(\mscrX_{\Zar}, \Ab))
    \end{equation}
    is also fully faithful, which sends a \(K\)-injective resolution \(\mcalI^{\bullet}\) of \(\mcalF\) in \(\mcalD(\mscrX_{\Zar}, \Ab)\) to a \(K\)-injective complex \(\iota(\mcalI^{\bullet})\) in \(\mcalD(\PSh(\mscrX_{\Zar}, \Ab))\).

    On the other hand, the combinatorial model structure on \(\Ch(\Ab)\) (\cite{lurie2017Higher}*{Proposition 1.3.5.3}) induces a combinatorial model structure on \(\PSh(\mscrX_{\Zar}, \Ch(\Ab))\) by sectionwise weak equivalences and cofibrations, and the evaluation functor
    \begin{equation*}
        \mscrX_{\Zar} \times N_{dg}(\PSh(\mscrX_{\Zar}, \Ch(\Ab))^\circ) \to N_{dg}(\Ch(\Ab)^\circ); \quad (\mscrU, \mcalI^{\bullet}) \mapsto \Gamma(\mscrU, \mcalI^{\bullet})
    \end{equation*}
    induces an equivalence of \(\infty\)-categories
    \begin{align} \label{PresheavesCatEquivFunc}
        & \mcalD(\PSh(\mscrX_{\Zar}, \Ab)) \simeq N_{dg}(\PSh(\mscrX_{\Zar}, \Ch(\Ab))^\circ)[W^{'-1}] \\
        & \xrightarrow{\simeq} \PSh(\mscrX_{\Zar}, N_{dg}(\Ch(\Ab)^\circ)[W^{-1}]) \simeq \PSh(\mscrX_{\Zar}, \mcalD(\setZ)) \nonumber
    \end{align}
    by \cite{lurie2017Higher}*{Proposition 1.3.4.25}, where \(W\) and \(W'\) are the weak equivalences in \(\Ch(\Ab)\) and \(\PSh(\mscrX_{\Zar}, \Ch(\Ab))\), respectively.
    By construction, the composition of the above two functors \eqref{SheafToPresheaf} and \eqref{PresheavesCatEquivFunc} sends \(\mcalF\) to \(R\Gamma(\mcalF)\) and is fully faithful.
\end{proof}

\begin{lemma} \label{AffineVanishing}
    Let \(\mscrX\) be an adic quasi-compact separated formal scheme such that \(p\) is topologically nilpotent in \(\mcalO_{\mscrX}\) and let \(\mcalF\) be a quasi-coherent \(\mcalO_{\mscrX}\)-module in the sense of \cite{gabber2018Foundations}*{Definition 15.1.31}.\footnote{Namely, there exists an affine open covering \(\mscrX = \cup_{i \in I} \mscrU_i\) and an \(\mcalO_{\mscrX}(\mscrU_i)\)-module \(M_i\) with an isomorphism \(\restr{\mcalF}{\mscrU_i} \cong M_i^{\Delta}\), where \(M_i^{\Delta}\) is the quasi-coherent sheaf associated to \(M_i\) on \(\mscrU_i\) (\Cref{DefCompletionQcoh} and \Cref{ConstDerivedCompletionDiagram}).}
    For any affine open subset \(\mscrU\) of \(\mscrX\), the section \(R\Gamma(\mcalF)(\mscrU)\) is naturally isomorphic to \(\mcalF(\mscrU)\).
\end{lemma}

\begin{proof}
    This follows from the affine vanishing theorem for quasi-coherent sheaves on affine formal schemes; see, for example, \cite{gabber2018Foundations}*{Theorem 15.1.37}.
\end{proof}

\begin{definition} \label{DefDerivedStructureSheaf}
    Let \(\mscrX\) be an adic quasi-compact separated formal scheme such that \(p\) is topologically nilpotent in \(\mcalO_{\mscrX}\).
    For notational convenience, we define the \(\mcalD(\setZ)\)-valued sheaf
    \begin{equation*}
        \mcalO^d_{\mscrX} \defeq R\Gamma(\mcalO_{\mscrX}) \colon \mscrX_{\Zar} \to \mcalD(\setZ); \quad \mscrU \mapsto R\Gamma(\mscrU, \mcalO_{\mscrX}).
    \end{equation*}
    By \Cref{AffineVanishing}, the section \(\mcalO^d_{\mscrX}(\mscrU)\) is naturally isomorphic to \(\mcalO_{\mscrX}(\mscrU)\) for any affine open subset \(\mscrU\) of \(\mscrX\).
\end{definition}

\begin{lemma} \label{PushforwardDerivedStructureSheaf}
    Let \(f \colon \mscrY \to \mscrX\) be a morphism of adic quasi-compact separated formal schemes such that \(p\) is topologically nilpotent in their structure sheaves.
    Then the presheaf
    \begin{equation*}
        f_*\mcalO^d_{\mscrY} \colon \mscrX_{\Zar} \to \mcalD(\setZ); \quad \mscrU \mapsto \mcalO^d_{\mscrY}(f^{-1}(\mscrU)) = R\Gamma(f^{-1}(\mscrU), \mcalO_{\mscrY})
    \end{equation*}
    is naturally isomorphic to the \(\mcalD(\setZ)\)-valued sheaf \(R\Gamma(Rf_*\mcalO_{\mscrY})\) on \(\mscrX_{\Zar}\) and admits a morphism \(\mcalO^d_{\mscrX} \to f_*\mcalO^d_{\mscrY}\) which is compatible with the morphism \(\mcalO_{\mscrX} \to f_*\mcalO_{\mscrY}\).
\end{lemma}

\begin{proof}
    Note that the natural isomorphism
    \begin{equation*}
        R\Gamma(\mscrU, Rf_*\mcalO_{\mscrY}) \cong R\Gamma(f^{-1}(\mscrU), \mcalO_{\mscrY})
    \end{equation*}
    holds in \(\mcalD(\setZ)\), as in \eqref{CommutativeGlobalSections}.
    So the pushforward \(f_*\mcalO^d_{\mscrY}\) is naturally isomorphic to \(R\Gamma(Rf_*\mcalO_{\mscrY})\), where \(Rf_*\mcalO_{\mscrY}\) denotes the derived pushforward of \(\mcalO_{\mscrY}\) in \(\mcalD(\mscrX_{\Zar}, \mcalO_{\mscrX})\).
    The morphism \(\mcalO^d_{\mscrX} \to f_*\mcalO^d_{\mscrY}\) is induced from the morphism \(\mcalO_{\mscrX} \to f_*\mcalO_{\mscrY}\) by applying \(R\Gamma(\mscrU, -)\) for each \(\mscrU \in \mscrX_{\Zar}\).
\end{proof}

\subsection{Construction of absolute perfectoidizations} \label{SubSectionAbsPerfdArc}

In this subsection, we will introduce the absolute perfectoidization of the structure sheaf using the arc-topology, and compare it with another construction via the perfectization of formal schemes in \cite{bhatt2025Aspects}.
First we recall the definition of the continuous arc-topology and its properties.

\begin{definition}[{\cite{takaya2026Relative}*{\S 4.14}, \cite{bhatt2022Prismsa}*{\S 8.2}, and \cite{bhatt2021Arctopology}}] \label{ContinuousArcTopology}
    Let \(\fSch\) be the category of adic quasi-compact separated formal schemes such that \(p\) is topologically nilpotent on their structure sheaves.
    A morphism \(\mscrY \to \mscrX\) in \(\fSch\) is a \emph{continuous arc-cover} if for every morphism \(\Spf(V) \to \mscrX\) in \(\fSch\) from a complete valuation ring of rank \(1\), there exists a faithfully flat extension \(V \hookrightarrow W\) of complete valuation rings of rank \(1\) and a morphism \(\Spf(W) \to \mscrY\) in \(\fSch\) such that the following diagram
    \begin{center}
        \begin{tikzcd}
            \mscrX           & \Spf(V) \arrow[l]           \\
            \mscrY \arrow[u] & \Spf(W) \arrow[u] \arrow[l]
        \end{tikzcd}
    \end{center}
    becomes commutative in \(\fSch\).
    The \emph{continuous arc-topology} on \(\fSch\) is the Grothendieck topology generated by continuous arc-covers and open covers.
    We will denote the site as \(\fSch^{\arc}\), which is called the \emph{continuous arc-site}.
    Taking the slice category, we also have the \emph{continuous arc-site \(\fSch^{\arc}_{/\mscrX}\) over \(\mscrX\)}.

    The continuous arc-site carries a presheaf of rings
    \[
    \mcalO_{\arc}\colon (\mscrY \to \mscrX) \mapsto \mcalO_{\mscrY}(\mscrY),
    \]
    which we call the structure presheaf on \(\fSch^{\arc}\).
\end{definition}

\begin{proposition}[{\cite{takaya2026Relative}*{Example 4.15 and Corollary 4.18} and \cite{bhatt2022Prismsa}*{Lemma 8.8 and Proposition 8.10}}] \label{ArcTopologyPerfd}
    Let \(\Perfd^{\arc}\) be the subcategory of \(\fSch^{\arc}\) consisting of perfectoid formal schemes, whose Grothendieck topology is the induced one.
    Any adic quasi-compact separated formal scheme \(\mscrX \in \fSch\) admits a continuous arc-cover by some \(\mscrY \in \Perfd^{\arc}\).
    In particular, restriction induces a canonical equivalence of categories
    \begin{equation} \label{EquivPerfdArc}
        \Shv(\fSch^{\arc}, \Ab) \xrightarrow{\simeq} \Shv(\Perfd^{\arc}, \Ab)
    \end{equation}
    holds.
    Moreover, the restriction \(\restr{\mcalO_{\arc}}{\Perfd^{\arc}}\) becomes a sheaf of rings on \(\Perfd^{\arc}\).
\end{proposition}

\begin{proof}
    The first assertion follows from \cite{takaya2026Relative}*{Example 4.15}: Indeed, since any open coverings are continuous arc-covers, we may assume that \(\mscrX = \Spf(R)\) is affine. In loc. cit., it is proved that any complete \(I\)-adic perfectoid ring \(R\) has a continuous arc-cover \(R \to S\) such that \(S\) is the \(I\)-adically complete product of \(p\)-adically complete absolute integrally closed valuation rings of rank \(1\) but this also works for general complete adic ring \(R\) as in \cite{bhatt2022Prismsa}*{Remark 8.9}.
    The categorical equivalence \eqref{EquivPerfdArc} follows from this assertion since \(\Perfd^{\arc}\) is now a basis of \(\fSch^{\arc}\).

    The sheaf property of \(\mcalO_{\arc}\) on \(\Perfd^{\arc}\) also follows from \cite{takaya2026Relative}*{Corollary 4.18}.
\end{proof}

\begin{lemma} \label{SheafificationArcStructureSheaf}
    Let \(\mscrX\) be an adic quasi-compact separated formal scheme such that \(p\) is topologically nilpotent in \(\mcalO_{\mscrX}\).
    Via the equivalence of categories \eqref{EquivPerfdArc}, let \(\widetilde{\mcalO}_{\arc}\) be the sheaf of rings on \(\fSch^{\arc}_{/\mscrX}\) whose restriction on \(\Perfd^{\arc}\) is \(\restr{\mcalO_{\arc}}{\Perfd^{\arc}}\).
    Then there exists a canonical morphism \(\mcalO_{\arc} \to \widetilde{\mcalO}_{\arc}\) of presheaves on \(\fSch^{\arc}_{/\mscrX}\) which exhibits \(\widetilde{\mcalO}_{\arc}\) as the sheafification of \(\mcalO_{\arc}\).
\end{lemma}

\begin{proof}
    Let \(\mcalO'_{\arc}\) be the sheafification of \(\mcalO_{\arc}\) on \(\fSch^{\arc}_{/\mscrX}\).
    By the equivalence \eqref{EquivPerfdArc}, the restriction \(\restr{\mcalO'_{\arc}}{\Perfd^{\arc}}\) is also the sheafification of \(\restr{\mcalO_{\arc}}{\Perfd^{\arc}}\) on \(\Perfd^{\arc}\).
    Since \(\restr{\mcalO_{\arc}}{\Perfd^{\arc}_{/\mscrX}}\) is already a sheaf, namely \(\restr{\widetilde{\mcalO}_{\arc}}{\Perfd^{\arc}_{/\mscrX}}\), the restriction \(\restr{\mcalO'_{\arc}}{\Perfd^{\arc}_{/\mscrX}}\) is isomorphic to this as sheaves of rings on \(\Perfd^{\arc}_{/\mscrX}\).
    Again using the equivalence \eqref{EquivPerfdArc}, we have an isomorphism \(\widetilde{\mcalO}_{\arc} \xrightarrow{\cong} \mcalO'_{\arc}\) of sheaves of rings on \(\fSch^{\arc}_{/\mscrX}\) and thus the morphism \(\mcalO_{\arc} \to \mcalO'_{\arc}\) exhibits \(\widetilde{\mcalO}_{\arc}\) as the sheafification of \(\mcalO_{\arc}\).
\end{proof}

\begin{construction} \label{DerivedPushforward}
    Let \(\mscrX\) be an adic quasi-compact separated formal scheme such that \(p\) is topologically nilpotent in \(\mcalO_{\mscrX}\).
    Take the continuous arc-site (resp., Zariski site) \(\fSch^{\arc}_{/\mscrX}\) (resp., \(\fSch^{\Zar}_{/\mscrX}\)) of formal schemes over \(\mscrX\), and the Zariski site \(\mscrX_{\Zar}\) of \(\mscrX\).
    The continuous functor of sites
    \begin{equation*}
        \nu \colon \mscrX_{\Zar} \to \fSch^{\Zar}_{/\mscrX} \to \fSch^{\arc}_{/\mscrX}; \quad (\mscrU \hookrightarrow \mscrX) \mapsto (\mscrU \hookrightarrow \mscrX)
    \end{equation*}
    of sites (\citeSta{00WV}) induces a functor of the abelian categories of sheaves
    \begin{equation*}
        \nu_* \colon \Shv(\fSch^{\arc}_{/\mscrX}, \Ab) \to \Shv(\mscrX_{\Zar}, \Ab); \quad \mcalF \mapsto ((\mscrU \hookrightarrow \mscrX) \mapsto \mcalF(\mscrU \hookrightarrow \mscrX))
    \end{equation*}
    by restriction along the topology.
    By \citeSta{03Q2} and \citeSta{04JC}, this functor \(\nu_*\) admits an exact left adjoint \(\nu^{-1}\).
    Moreover, since the arc structure sheaf \(\mcalO_{\arc}\) coincides with the structure sheaf \(\mcalO_{\mscrX}\) on \(\mscrX_{\Zar}\), this \(\nu_*\) gives rise to a morphism of ringed topoi (\citeSta{01D3}):
    \begin{equation*}
        (\nu_*, \mathrm{can}) \colon (\Shv(\fSch^{\arc}_{/\mscrX}, \Ab), \widetilde{\mcalO}_{\arc}) \to (\Shv(\mscrX_{\Zar}, \Ab), \mcalO_{\mscrX}),
    \end{equation*}
    where \(\mathrm{can}\colon \mcalO_{\mscrX}=\restr{\mcalO_{\arc}}{\mscrX_{\Zar}}\to \nu_*\widetilde{\mcalO}_{\arc}\) is the canonical morphism supplied by \Cref{SheafificationArcStructureSheaf}.
    This defines the adjunction
    \begin{equation*}
        \nu^* \colon \Mod(\mscrX_{\Zar}, \mcalO_{\mscrX}) \rightleftarrows \Mod(\fSch^{\arc}_{/\mscrX}, \widetilde{\mcalO}_{\arc}) \colon \nu_*
    \end{equation*}
    of Grothendieck abelian categories of modules (\citeSta{03D5}) and a commutative diagram
    \begin{equation*}
        \begin{tikzcd}
            {\Mod(\fSch^{\arc}_{/\mscrX}, \widetilde{\mcalO}_{\arc})} \arrow[rr, "\nu_*"] \arrow[rd, "{\Gamma(\mscrU, -)}"'] &                               & {\Mod(\mscrX_{\Zar}, \mcalO_{\mscrX})} \arrow[ld, "{\Gamma(\mscrU, -)}"] \\
            & \Mod(\mcalO_{\mscrX}(\mscrU)) &   
        \end{tikzcd}
    \end{equation*}
    for any affine open formal subscheme \(\mscrU\) of \(\mscrX\), where \(\Mod(\mcalO_{\mscrX}(\mscrU))\) is the abelian category of \(\mcalO_{\mscrX}(\mscrU)\)-modules.
    Note that the functors are lax symmetric monoidal.

    Especially, taking the right derived functors, we have a commutative diagram
    \begin{equation} \label{CommutativeGlobalSections}
        \begin{tikzcd}
            {\mcalD(\fSch^{\arc}_{/\mscrX}, \widetilde{\mcalO}_{\arc})} \arrow[rd, "{R\Gamma_{\arc}(\mscrU, -)}"'] \arrow[rr, "R\nu_*"] &               & {\mcalD(\mscrX_{\Zar}, \mcalO_{\mscrX})} \arrow[ld, "{R\Gamma_{\Zar}(\mscrU, -)}"] \\
            & \mcalD(\mcalO_{\mscrX}(\mscrU)) &  
        \end{tikzcd}
    \end{equation}
    of derived \(\infty\)-categories for each affine open formal subscheme \(\mscrU\) of \(\mscrX\) and the functors are lax symmetric monoidal as well.
    Note that this construction is functorial on \(\mscrU \in \mscrX_{\Zar}\) and \(\mscrX \in \fSch\).
\end{construction}

\begin{definition} \label{DerivedPushoutArcCohomology}
    Let \(\mscrX\) be an adic quasi-compact separated formal scheme such that \(p\) is topologically nilpotent in \(\mcalO_{\mscrX}\).
    We define an \(\setE_{\infty}\)-ring object \(\mcalO_{\mscrX,\perfd}\) of \(\mcalD(\mscrX_{\Zar}, \mcalO_{\mscrX})\) as the derived pushforward
    \begin{equation*}
        \mcalO_{\mscrX,\perfd} \defeq R\nu_{*}\widetilde{\mcalO}_{\arc} \in \CAlg(\mcalD(\mscrX_{\Zar}, \mcalO_{\mscrX}))
    \end{equation*}
    along the lax symmetric monoidal functor \(R\nu_*\) defined in \Cref{DerivedPushforward}. We call this the \emph{absolute perfectoidization} of the structure sheaf \(\mcalO_{\mscrX}\) on \(\mscrX\).
    By construction, \(\mcalO_{\mscrX,\perfd}\) is functorial in \(\mscrX\) in the sense that for any morphism \(f \colon \mscrY \to \mscrX\) of formal schemes in \(\fSch\), there exists a morphism
    \begin{equation*}
        \mcalO_{\mscrX, \perfd} \to Rf_*\mcalO_{\mscrY, \perfd}
    \end{equation*}
    in \(\CAlg(\mcalD(\mscrX_{\Zar}, \mcalO_{\mscrX}))\) which is compatible with the morphism \(\mcalO_{\mscrX} \to Rf_*\mcalO_{\mscrY}\) of structure sheaves.
\end{definition}

\begin{corollary} \label{SectionDerivedPushout}
    Let \(\mscrX\) be an adic quasi-compact separated formal scheme such that \(p\) is topologically nilpotent in \(\mcalO_{\mscrX}\) and let \(\mscrU = \Spf(R)\) be an affine open subset of \(\mscrX\).
    Then there exists an isomorphism
    \begin{equation*}
        R\Gamma(\mscrU, \mcalO_{\mscrX,\perfd}) \xrightarrow{\cong} R_{\tperfd}
    \end{equation*}
    in \(\CAlg_R\), which is functorial on affine open subsets \(\mscrU\) of \(\mscrX\), where the left-hand side denotes the derived global sections of \(\mcalO_{\mscrX,\perfd} \in \mcalD(\mscrX_{\Zar}, \mcalO_{\mscrX})\) on \(\mscrU\).
\end{corollary}

\begin{proof}
    Since \(R\Gamma_{\arc}(\mscrU, \widetilde{\mcalO}_{\arc})\) is the cohomology \(R\Gamma(\fSch^{\arc}_{/\mscrU}, \widetilde{\mcalO}_{\arc})\) of \(\widetilde{\mcalO}_{\arc}\) on the continuous arc-site \(\fSch^{\arc}_{/\mscrU}\), this is the same as the cohomology \(R\Gamma(\Perfd^{\arc}_{/\mscrU}, \restr{\mcalO_{\arc}}{\Perfd^{\arc}_{/\mscrU}})\) by \Cref{ArcTopologyPerfd}.
    When \(\mscrU = \Spf(R)\) is an affine open subset of \(\mscrX\), the derived global section \(R\Gamma_{\arc}(\Spf(R), \widetilde{\mcalO}_{\arc})\) admits canonical isomorphisms in \(\mcalD(R)\);
    \begin{align*}
        R\Gamma_{\arc}(\Spf(R), \widetilde{\mcalO}_{\arc}) & \overset{(a)}{\cong} R\Gamma(\fSch^{\arc}_{/\Spf(R)}, \widetilde{\mcalO}_{\arc}) \overset{(b)}{\cong} R\Gamma(\AffPerfd^{\arc}_{/\Spf(R)}, \mcalO_{\arc}) \\
        & \overset{(c)}{\cong} \lim_{R \to P} R\Gamma_{\arc}(\Spf(P), \mcalO_{\arc}) \overset{(d)}{\cong} \lim_{R \to P} P \overset{(e)}{\cong} R_{\tperfd},
    \end{align*}
    where \(\AffPerfd^{\arc}_{/\Spf(R)}\) is the full subcategory of \(\Perfd^{\arc}_{/\Spf(R)}\) consisting of affine objects, by the following argument as in the proof of \cite{bhatt2022Prismsa}*{Corollary 8.11}: the isomorphism (a) follows from the above argument. The second one (b) follows from the existence of continuous arc-cover of \(\Spf(R)\) by the proof of \Cref{ArcTopologyPerfd}. The third one (c) follows from the definition of cohomology on a site. The fourth isomorphism (d) follows from the acylicity of \(\mcalO_{\arc}\) on \(\AffPerfd^{\arc}\) in \Cref{ArcTopologyPerfd}. The last one (e) follows from \Cref{DefTopologicalAbsolutePerfectoidization}.
    The functoriality follows from the functoriality of the above isomorphisms.
\end{proof}

\begin{remark}
    By the sheaf property, there is a uniquely determined object \(\mcalP\) of \(\mcalD(\mscrX_{\Zar}, \mcalO_{\mscrX})\) such that the derived global section \(R\Gamma(\mscrU, \mcalP)\) is isomorphic to \(R_{\tperfd}\) for any affine open subset \(\mscrU = \Spf(R)\) of \(\mscrX\).
    There is another construction of such an object \(\mcalP\) by using more geometric methods, which is the perfectization of formal schemes in \cite{bhatt2025Aspects}. See \Cref{SectionPerfectization} (especially \Cref{ComparisonOXperfd}) for the details.
\end{remark}

Using \Cref{CommutativeLocalizationTopPerfd}, we can show the quasi-coherence of the absolute perfectoidization \(\mcalO_{\mscrX, \perfd}\) on \(\mscrX\) as follows.
In \Cref{AppendixQCoh}, we summarize the notion of derived category \(\mcalD_{\qcoh}(\mscrX)\) of a formal scheme \(\mscrX\).
See \Cref{SectionPerfectization} for another proof of this quasi-coherence by using the perfectization of formal schemes in \cite{bhatt2025Aspects}.

\begin{corollary} \label{QuasiCoherentnessOfPerfdGlobal}
    Let \(\mscrX\) be an adic quasi-compact separated formal scheme such that \(p\) is topologically nilpotent in \(\mcalO_{\mscrX}\).
    Assume that the ideal of definition \(\mcalI\) is generated by a weakly proregular sequence on any affine open subset of \(\mscrX\) (e.g. \(\mscrX\) is locally Noetherian or \(\mcalI\) is principal and \(\mscrX\) has bounded \(\mcalI^{\infty}\)-torsion).
    Then the absolute perfectoidization \(\mcalO_{\mscrX, \perfd} \in \mcalD(\mscrX_{\Zar}, \mcalO_{\mscrX})\) of \(\mcalO_{\mscrX}\) on \(\mscrX\) is quasi-coherent in the sense of \Cref{DefDerivedQcohFormalScheme}.
\end{corollary}

\begin{proof}
    We may assume that \(\mscrX = \Spf(R)\) is affine and fix an ideal of definition \(I\) of \(R\) generated by a weakly proregular sequence.
    For any \(f\in R\), we have canonical isomorphisms
    \begin{equation*}
        \mcalO_{\Spf(R), \perfd}(D(f)) \cong (R[1/f])_{\tperfd} \xleftarrow{\cong} \dcomp{I}{R_{\tperfd}[1/f]}
    \end{equation*}
    in \(\CAlg_{R[1/f]}\) by \Cref{SectionDerivedPushout} and \Cref{CommutativeLocalizationTopPerfd}.
    Since \(I\) is generated by a weakly proregular sequence, this is isomorphic to
    \begin{equation*}
        \lim_{n \geq 1} ((R_{\tperfd} \otimes^L_R R/I^n)[1/f]) \cong \lim_{n \geq 1} R\Gamma(D(f), \widetilde{(R_{\tperfd} \otimes^L_R R/I^n)}),
    \end{equation*}
    where \(\widetilde{(R_{\tperfd} \otimes^L_R R/I^n)}\) is the quasi-coherent sheaf on \(\Spec(R/I^n)\) associated to \(R_{\tperfd} \otimes^L_R R/I^n\) in \(\mcalD(R/I^n)\).
    These isomorphisms are functorial in \(f\) and thus we have an isomorphism
    \begin{equation*}
        \mcalO_{\Spf(R), \perfd} \cong \lim_{n \geq 1} \widetilde{(R_{\tperfd} \otimes^L_R R/I^n)} \cong (R_{\tperfd})^{L\Delta}
    \end{equation*}
    in \(\mcalD(\Spf(R)_{\Zar}, \mcalO_{\Spf(R)})\) by the categorical equivalence in \Cref{DerivedQCohEquivAffine}, which proves the quasi-coherence of \(\mcalO_{\mscrX,\perfd}\).
\end{proof}

\subsection{Absolute perfectoidization and derived sheafification}

In this subsection, we will compute the derived sections over arbitrary open subsets of the absolute perfectoidization \(\mcalO_{\mscrX,\perfd}\) defined in \Cref{DerivedPushoutArcCohomology}. To prove this, the sheaf \(\mcalO_{\mscrX, \perfd}\) will be considered as a \(\mcalD(\setZ)\)-valued sheaf on \(\mscrX\) by using derived sheafification developed in \Cref{SheafifiedDerivedObjects}.

\begin{construction} \label{DefGlobalPerfectoidization}
    Let \(\mscrX\) be an adic quasi-compact separated formal scheme such that \(p\) is topologically nilpotent in \(\mcalO_{\mscrX}\).
    Let \(\mcalO^d_{\mscrX,\perfd}\) be the object defined in the following construction.
    \begin{equation*}
        \mcalO^d_{\mscrX, \perfd} \defeq R\Gamma(\mcalO_{\mscrX, \perfd}) \in \Shv(\mscrX_{\Zar}, \mcalD(\setZ))
    \end{equation*}
    which is naturally equipped with the structure of an \(\mcalO^d_{\mscrX}\)-algebra by the morphism \(\mcalO_{\mscrX} \to \mcalO_{\mscrX,\perfd}\) in \(\CAlg(\mcalD(\mscrX_{\Zar}, \mcalO_{\mscrX}))\).
    Namely, \(\mcalO^d_{\mscrX, \perfd}\) is an object of the \(\infty\)-category \(\CAlg_{\mcalO^d_{\mscrX}}(\Shv(\mscrX_{\Zar}, \mcalD(\setZ)))\) of \(\mcalO^d_{\mscrX}\)-algebras on \(\mscrX\).
\end{construction}

\begin{remark}
    Even if \(\mscrX\) is an adic quasi-compact separated perfectoid formal scheme, the absolute perfectoidization \(\mcalO^d_{\mscrX,\perfd}\) is not necessarily equivalent, as a \(\CAlg\)-valued presheaf, to the structure (pre)sheaf \(\mcalO_{\mscrX}\), but it is isomorphic to \(\mcalO^d_{\mscrX}\) as a \(\CAlg\)-valued sheaf.
    This follows from the following argument: The canonical morphism \(\mcalO^d_{\mscrX} \to \mcalO^d_{\mscrX,\perfd}\) is an isomorphism on every affine open subset of \(\mscrX\) by \Cref{SectionDerivedPushout}.
    Since they are sheaves on \(\mscrX_{\Zar}\), the canonical morphism is an isomorphism.
\end{remark}

\begin{remark} \label{RemarkAffineSectionGlobalPerfd}
    If \(\mscrX\) is an adic quasi-compact separated formal scheme, then \(\mcalO^d_{\mscrX,\perfd}(\mscrU)\) is canonically identified with the topological absolute perfectoidization \(\mcalO_{\mscrX}(\mscrU)_{\tperfd}\) of the ring \(\mcalO_{\mscrX}(\mscrU)\) for any affine formal open subset \(\mscrU \subseteq \mscrX\) defined in \Cref{DefTopologicalAbsolutePerfectoidization}.
    This follows from \Cref{SectionDerivedPushout}.
\end{remark}



\begin{lemma} \label{UniversalityOfGlobalPerfd}
    Let \(\mscrX\) be an adic quasi-compact separated formal scheme such that \(p\) is topologically nilpotent in \(\mcalO_{\mscrX}\).
    For any morphism \(\pi\colon \mscrX_{\infty}\to\mscrX\) from a perfectoid formal scheme \(\mscrX_{\infty}\), there exists a unique morphism of \(\mcalO^d_{\mscrX}\)-algebras
    \begin{equation*}
        \mcalO^d_{\mscrX, \perfd} \to \pi_*\mcalO^d_{\mscrX_{\infty}}
    \end{equation*}
    such that for any affine open subset \(\mscrU \subseteq \mscrX\) and \(\mscrV \subseteq \pi^{-1}(\mscrU)\) the induced morphism
    \begin{equation*}
        \mcalO_{\mscrX}(\mscrU)_{\tperfd} \cong \mcalO^d_{\mscrX, \perfd}(\mscrU) \to \pi_*\mcalO^d_{\mscrX_{\infty}}(\mscrU) \to \mcalO^d_{\mscrX_{\infty}}(\mscrV) \cong \mcalO_{\mscrX_{\infty}}(\mscrV)
    \end{equation*}
    is the canonical projection from the absolute perfectoidization \(\mcalO_{\mscrX}(\mscrU)_{\perfd}\), where the isomorphisms are consequences of \Cref{SectionDerivedPushout} and \Cref{AffineVanishing}.
\end{lemma}

\begin{proof}
    Let \(\mscrU=\Spf(R)\) be an affine open formal subscheme of \(\mscrX\), and consider the induced morphism \(\mcalO_{\mscrX}(\mscrU) \to \mcalO_{\mscrX_{\infty}}(\pi^{-1}(\mscrU))\).
    For any affine open subset \(\mscrV\) of \(\pi^{-1}(\mscrU)\), the morphism \(\mcalO_{\mscrX}(\mscrU) \to \mcalO_{\mscrX_{\infty}}(\pi^{-1}(\mscrU)) \to \mcalO_{\mscrX_{\infty}}(\mscrV)\) uniquely factors through the universal map \(\mcalO_{\mscrX}(\mscrU) \to \mcalO_{\mscrX}(\mscrU)_{\tperfd}\) since \(\mcalO_{\mscrX_{\infty}}(\mscrV)\) is a complete adic perfectoid ring over \(\mcalO_{\mscrX}(\mscrU)\) by \Cref{PropPerfectoidFormalScheme}.
    Moreover, the factorization \(\mcalO_{\mscrX}(\mscrU)_{\tperfd} \to \mcalO_{\mscrX_{\infty}}(\mscrV)\) is compatible with any restriction map induced from containments \(\mscrV \supset \mscrV'\) of affine open subsets by uniqueness.
    The morphism \(\mcalO_{\mscrX_{\infty}}(\mscrV) \xrightarrow{\cong} \mcalO^d_{\mscrX_{\infty}}(\mscrV)\) is an isomorphism by \Cref{AffineVanishing}.
    We have the unique compatible system of morphisms \(\{\mcalO_{\mscrX}(\mscrU)_{\tperfd} \to \mcalO^d_{\mscrX_{\infty}}(\mscrV)\}_{\mscrV \subseteq \pi^{-1}(\mscrU)}\) and then the sheaf property of \(\pi_*\mcalO^d_{\mscrX_{\infty}}\) gives a morphism of \(\mcalO^d_{\mscrX}(\mscrU)\)-algebras
    \begin{equation*}
        \mcalO^d_{\mscrX, \perfd}(\mscrU) = \mcalO_{\mscrX}(\mscrU)_{\tperfd} \to \lim_{\mscrV \subseteq \pi^{-1}(\mscrU)} \mcalO^d_{\mscrX_{\infty}}(\mscrV) \cong \mcalO^d_{\mscrX_{\infty}}(\pi^{-1}(\mscrU)) = \pi_*\mcalO^d_{\mscrX_{\infty}}(\mscrU)
    \end{equation*}
    for any affine open subset \(\mscrU \subseteq \mscrX\). In particular, using the sheaf property again, we have a morphism \(\mcalO^d_{\mscrX, \perfd} \to \pi_*\mcalO^d_{\mscrX_{\infty}}\) of \(\mcalO^d_{\mscrX}\)-algebras.
\end{proof}

\begin{definition} \label{DefCatPerfd}
    Let \(\mscrX\) be an adic quasi-compact separated formal scheme such that \(p\) is topologically nilpotent in \(\mcalO_{\mscrX}\).
    We define the following categories.
    \begin{enumerate}
        \item The category \(\Perfd_{\mscrX}\) of all perfectoid formal schemes over \(\mscrX\).
        \item The category \(\Perfd^{\Aff}_{\mscrX}\) of all affine morphisms \(\mscrX_{\infty}\to\mscrX\), where \(\mscrX_{\infty}\) is a perfectoid formal scheme.
        \item The category \(\AffPerfd_{\mscrX}\) of all morphisms \(\Spf(P) \to \mscrX\) from an affine perfectoid formal scheme \(\Spf(P)\).
        \item The category \(\AffPerfd_{\mscrX}^{\bigstar}\) is the full subcategory of \(\AffPerfd_{\mscrX}\) such that \(\Spf(P) \to \mscrX\) factors through an affine open subset \(\mscrU\) of \(\mscrX\).
    \end{enumerate}
    In particular, we have the following inclusions of full subcategories:
    \begin{equation*}
        \Perfd_{\mscrX} \supset \Perfd^{\Aff}_{\mscrX} \supset \AffPerfd_{\mscrX} \supset \AffPerfd_{\mscrX}^{\bigstar},
    \end{equation*}
    where the second containment follows from the fact that any morphism \(\pi \colon \Spf(P) \to \mscrX\) from an affine formal scheme \(\Spf(P)\) to a separated formal scheme \(\mscrX\) is affine.\footnote{For any affine open formal subscheme \(j \colon \mscrU \hookrightarrow \mscrX\), the separatedness of \(\mscrX\) ensures the affineness of the morphism \(j\). So the base change \(\pi^{-1}(\mscrU) = \mscrU \times_{\mscrX} \Spf(P) \to \Spf(P)\) is also an affine morphism and then \(\pi^{-1}(\mscrU)\) is an affine open subscheme of \(\Spf(P)\).}
\end{definition}

\begin{proposition} \label{GlobalPerfectoidizationSheafLimit}
    Let \(\mscrX\) be an adic quasi-compact separated formal scheme such that \(p\) is topologically nilpotent in \(\mcalO_{\mscrX}\).
    Let \(\Perfd^{pre}_{\mscrX}\) be any full subcategory of \(\Perfd_{\mscrX}\) which contains \(\AffPerfd_{\mscrX}^{\bigstar}\).
    Then the morphism induced from \Cref{UniversalityOfGlobalPerfd} gives an isomorphism
    \begin{equation} \label{AbsPerfdLimit}
        \mcalO^d_{\mscrX, \perfd} \xrightarrow{\cong} \lim_{\mscrX_{\infty} \xrightarrow{\pi} \mscrX} \pi_*\mcalO^d_{\mscrX_{\infty}}
    \end{equation}
    in the \(\infty\)-category \(\CAlg_{\mcalO^d_{\mscrX}}(\Shv(\mscrX_{\Zar}, \mcalD(\setZ)))\) where the limit is taken over \(\Perfd^{pre}_{\mscrX}\).
    This isomorphism is functorial in the following sense: For any morphism \(f \colon \mscrY \to \mscrX\) and full subcategories \(\Perfd^{pre}_{\mscrX}\) and \(\Perfd^{pre}_{\mscrY}\) as above such that for any morphism \(\pi \colon \mscrY_{\infty} \to \mscrY\) in \(\Perfd^{pre}_{\mscrY}\) the composition \(f \circ \pi\) belongs to \(\Perfd^{pre}_{\mscrX}\), the morphisms
    \begin{equation*}
        \mcalO^d_{\mscrX, \perfd} \to f_*\mcalO^d_{\mscrY, \perfd} \quad \text{and} \quad \lim_{\mscrX_{\infty} \xrightarrow{\pi} \mscrX} \pi_*\mcalO^d_{\mscrX_{\infty}} \to \lim_{\mscrY_{\infty} \xrightarrow{\pi'} \mscrY} f_*\pi'_*\mcalO^d_{\mscrY_{\infty}}
    \end{equation*}
    commute with the above isomorphisms.
\end{proposition}

\begin{proof}

    By \Cref{UniversalityOfGlobalPerfd}, we obtain a morphism \(\mcalO^d_{\mscrX,\perfd} \to \lim_{f \in \Perfd^{pre}_{\mscrX}} f_*\mcalO^d_{\mscrX_{\infty}}\).
    The functoriality of this morphism follows from the following functoriality: For any morphism \(f \colon \mscrY \to \mscrX\) and an object \(\pi \colon \mscrY_{\infty} \to \mscrY\) in \(\Perfd^{pre}_{\mscrY}\), the composition \(f \circ \pi\) belongs to \(\Perfd^{pre}_{\mscrX}\) and this gives a commutative diagram
    \begin{center}
        \begin{tikzcd}
            {\mcalO^d_{\mscrX, \perfd}} \arrow[rr] \arrow[d] &  & \lim_{\mscrX_{\infty} \xrightarrow{\pi} \mscrX} \pi_*\mcalO^d_{\mscrX_{\infty}} \arrow[d] \\
            {f_*\mcalO^d_{\mscrY, \perfd}} \arrow[rr]        &  & \lim_{\mscrY_{\infty} \xrightarrow{\pi'} \mscrY} f_*\pi'_*\mcalO^d_{\mscrY_{\infty}}     
        \end{tikzcd}
    \end{center}
    of \(\CAlg\)-valued sheaves.
    
    We will check that the morphism \(\mcalO^d_{\mscrX, \perfd} \to \lim_{\pi \in \Perfd^{pre}_{\mscrX}} \pi_*\mcalO^d_{\mscrX_{\infty}}\) satisfies the universality.
    Take \(\mcalF \in \CAlg_{\mcalO^d_{\mscrX}}(\Shv(\mscrX_{\Zar}, \mcalD(\setZ)))\) with a compatible system \(\{\tau_\pi \colon \mcalF \to \pi_*\mcalO^d_{\mscrX_{\infty}}\}_{\pi \in \Perfd^{pre}_{\mscrX}}\) of \(\mcalO^d_{\mscrX}\)-algebra morphisms, i.e., a natural transformation \(\tau \colon c_{\mcalF} \to \mcalO\) in \(\Fun(\Perfd^{pre}_{\mscrX}, \CAlg_{\mcalO^d_{\mscrX}}(\Shv(\mscrX_{\Zar}, \mcalD(\setZ))))\) where \(c_{\mcalF}\) is the constant functor \(\mscrX_{\infty} \mapsto \mcalF\) and \(\mcalO\) is the functor \(\mscrX_{\infty} \mapsto \pi_*\mcalO^d_{\mscrX_{\infty}}\).

    First we construct a morphism \(\mcalF \to \mcalO^d_{\mscrX, \perfd}\) of \(\mcalO^d_{\mscrX}\)-algebras.
    For any affine open subset \(\mscrU\) of \(\mscrX\), take any complete adic perfectoid \(\mcalO_{\mscrX}(\mscrU)\)-algebra \(P\).
    This gives rise to an object \(\pi \colon \Spf(P) \to \mscrU \hookrightarrow \mscrX\) in \(\AffPerfd_{\mscrX}^{\bigstar} \subseteq \Perfd^{pre}_{\mscrX}\).
    Then the natural transformation \(\tau \colon c_{\mcalF} \to \mcalO\) induces a morphism \(\tau_\pi(\mscrU) \colon \mcalF(\mscrU) \to \mcalO^d_{\Spf(P)}(\pi^{-1}(\mscrU)) \cong P\) of \(\mcalO_{\mscrX}(\mscrU)\)-algebras where the last isomorphism follows from \Cref{AffineVanishing}.
    This is compatible with any morphism \(P \to P'\) of complete adic perfectoid \(\mcalO_{\mscrX}(\mscrU)\)-algebras and then this gives a unique morphism
    \begin{equation*}
        \lim \tau_{\Spf(P) \to \mscrU \hookrightarrow \mscrX}(\mscrU) \colon \mcalF(\mscrU) \to \lim_{\mcalO_{\mscrX}(\mscrU) \to P} P = \mcalO^d_{\mscrX, \perfd}(\mscrU)
    \end{equation*}
    of \(\mcalO_{\mscrX}(\mscrU)\)-algebras by \Cref{SectionDerivedPushout}.
    Also this construction is compatible with the restriction morphisms, so the sheaf property gives a unique morphism \(\widetilde{\tau} \colon \mcalF \to \mcalO^d_{\mscrX, \perfd}\) of \(\mcalO^d_{\mscrX}\)-algebras such that the following diagram is commutative:
    \begin{equation} \label{FtoOperfdProjection}
        \begin{tikzcd}
            \mcalF(\mscrU) \arrow[r, "\tau_{\pi}(\mscrU)"] \arrow[rd, "\widetilde{\tau}(\mscrU)"] & P                                                                     \\
            \mcalO^d_{\mscrX}(\mscrU) \arrow[u] \arrow[r]                                         & {\mcalO^d_{\mscrX, \perfd}(\mscrU)} \arrow[u, "\mathrm{projection}"']
        \end{tikzcd}
    \end{equation}
    for any affine open subset \(\mscrU\) and any morphism \(\pi \colon \Spf(P) \to \mscrU \hookrightarrow \mscrX\) from the formal spectrum of a complete adic perfectoid \(\mcalO_{\mscrX}(\mscrU)\)-algebra \(P\) where the upper morphism \(\mcalF(\mscrU) \to P\) is induced from \(\tau_\pi \colon \mcalF \to \pi_*\mcalO^d_{\Spf(P)}\).

    Let \(\pi\colon \mscrX_{\infty} \to \mscrX\) be an object in \(\Perfd^{pre}_{\mscrX}\).
    Consider the composition
    \begin{equation*}
        \mcalF \xrightarrow{\widetilde{\tau}} \mcalO^d_{\mscrX, \perfd} \to \pi_*\mcalO^d_{\mscrX_{\infty}}
    \end{equation*}
    of \(\widetilde{\tau}\) defined above and the morphism induced from \Cref{UniversalityOfGlobalPerfd}.
    To prove that \(\widetilde{\tau}\) is the desired universal morphism, it is enough to show that the composition is the given morphism \(\mcalF \xrightarrow{\tau_{\pi}} \pi_*\mcalO^d_{\mscrX_{\infty}}\) for \(\mscrX_{\infty} \xrightarrow{\pi} \mscrX\).
    Since it is a morphism of sheaves, it is enough to show that the morphism \(\mcalF(\mscrU) \xrightarrow{\widetilde{\tau}(\mscrU)} \mcalO^d_{\mscrX, \perfd}(\mscrU) \to \mcalO^d_{\mscrX_{\infty}}(\pi^{-1}(\mscrU))\) is the morphism \(\tau_{\pi}(\mscrU) \colon \mcalF(\mscrU) \to \pi_*\mcalO^d_{\mscrX_{\infty}}(\mscrU)\) for any affine open subset \(\mscrU\).
    Consider the following diagram, in which the commutativity of the left triangle is not yet known:
    \begin{center}
        \begin{tikzcd}
            \mcalF(\mscrU) \arrow[rd, "\widetilde{\tau}(\mscrU)"'] \arrow[rr, "\tau_\pi(\mscrU)"] &                                                                              & \mcalO^d_{\mscrX_{\infty}}(\pi^{-1}(\mscrU)) \arrow[rr] &                                                       & \mcalO^d_{\mscrX_{\infty}}(\mscrV) \\
             & {\mcalO^d_{\mscrX, \perfd}(\mscrU)} \arrow[rr, "\text{projection}"] \arrow[ru] &                                                         & \mcalO_{\mscrX_{\infty}}(\mscrV) \arrow[ru, "\cong"'] &                                   
        \end{tikzcd}
    \end{center}
    for any affine open subset \(\mscrV \subseteq \pi^{-1}(\mscrU)\) of \(\mscrX_{\infty}\).
    We check some commutativity of this diagram: Using \Cref{UniversalityOfGlobalPerfd}, the right square is commutative because the morphism \(\mcalO^d_{\mscrX, \perfd}(\mscrU) \to \mcalO_{\mscrX_{\infty}}(\mscrV)\) is the projection to the complete adic perfectoid \(\mcalO_{\mscrX}(\mscrU)\)-algebra \(\mcalO_{\mscrX_{\infty}}(\mscrV)\).
    On the other hand, the diagram \eqref{FtoOperfdProjection} shows that the outer square is also commutative since the upper composition map \(\mcalF(\mscrU) \xrightarrow{\tau_{\pi}(\mscrU)} \mcalO^d_{\mscrX_{\infty}}(\pi^{-1}(\mscrU)) \to \mcalO^d_{\mscrX_{\infty}}(\mscrV)\) is the one indexed by \(\mscrV = \Spf(\mcalO^d_{\mscrX_{\infty}}(\mscrV)) \hookrightarrow \pi^{-1}(\mscrU) \xrightarrow{\pi} \mscrU \hookrightarrow \mscrX\) which belongs to \(\AffPerfd_{\mscrX}^{\bigstar} \subseteq \Perfd^{pre}_{\mscrX}\).
    Consequently, the sheaf condition on \(\mcalO^d_{\mscrX_{\infty}}\) implies that the left triangle is commutative.
    This gives the desired universality.
\end{proof}

\begin{corollary} \label{GlobalPerfectoidizationSectionLimit}
    Let \(\mscrX\) be an adic quasi-compact separated formal scheme such that \(p\) is topologically nilpotent in \(\mcalO_{\mscrX}\).
    Then we have a functorial isomorphism
    \begin{equation*}
        \mcalO_{\mscrX, \perfd} \xrightarrow{\cong} \lim_{\mscrX_{\infty} \xrightarrow{\pi} \mscrX} R\pi_*\mcalO_{\mscrX_{\infty}}
    \end{equation*}
    in \(\CAlg(\mcalD(\mscrX_{\Zar}, \mcalO_{\mscrX}))\), where the limit is taken over an arbitrary full subcategory \(\Perfd^{pre}_{\mscrX}\) of \(\Perfd_{\mscrX}\) which contains \(\AffPerfd_{\mscrX}^{\bigstar}\).
    In particular, for any open subset \(\mscrU\) of \(\mscrX\), the isomorphism
    \begin{equation*}
        R\Gamma(\mscrU, \mcalO_{\mscrX, \perfd}) \xrightarrow{\cong} \lim_{\Spf(P) \xrightarrow{\pi} \mscrU} P
    \end{equation*}
    holds where the limit is taken over all morphisms \(\Spf(P) \to \mscrU\) from affine perfectoid formal schemes.
\end{corollary}

\begin{proof}
    The isomorphism \eqref{AbsPerfdLimit} in \Cref{GlobalPerfectoidizationSheafLimit} can be regarded as an isomorphism
    \begin{equation*}
        R\Gamma(\mcalO_{\mscrX, \perfd}) \xrightarrow{\cong} R\Gamma(\lim_{\mscrX_{\infty} \xrightarrow{\pi} \mscrX} R\pi_*\mcalO_{\mscrX_{\infty}})
    \end{equation*}
    by \Cref{PushforwardDerivedStructureSheaf} and \Cref{RGammaCommutesLimits}.
    Since \(R\Gamma\) is fully faithful, we have the desired isomorphism.
    In the second isomorphism, we may assume that \(\mscrU = \mscrX\), and then it follows from the isomorphism \(R\Gamma(\mscrX, \mcalO_{\mscrX, \perfd}) \cong \lim_{\Spf(P) \to \mscrX} R\Gamma(\Spf(P), \mcalO_{\Spf(P)})\) as above and the vanishing of higher cohomology for affine formal schemes (\Cref{AffineVanishing}).
\end{proof}

\subsection{Relation with perfectization of formal schemes} \label{SectionPerfectization}

Recently, in \cite{bhatt2025Aspects}, Bhatt introduced the notion of the \emph{perfectization} \(\mfrakX^{\pfd}\) of a stack \(\mfrakX\), and constructed an object \(\mcalO_{\mfrakX, \pfd}\) of \(\mcalD_{\qcoh}(\mfrakX)\) which is very similar to our absolute perfectoidization sheaf \(\mcalO_{\mscrX, \perfd}\) for a formal scheme \(\mscrX\).

\begin{notation}
    Throughout this section, we will use the following terminology:
    \begin{itemize}
        \item The (1-)category of \(p\)-nilpotent semiperfectoid rings is denoted by \(\SPerfd^{\nil}\). We equip it with the fpqc topology and regard it as a site. Note that \(\SPerfd^{\nil}\) forms a basis for the fpqc topology on the category of \(p\)-nilpotent rings.
        \item A \emph{stack} is an \(\Ani\)-valued fpqc sheaf on \(\SPerfd^{\nil}\) (or equivalently, an \(\Ani\)-valued fpqc sheaf on the category of \(p\)-nilpotent rings).
        \item Any formal scheme \(\mfrakX\) has an associated stack that sends \(S \in \SPerfd^{\nil}\) to the set of morphisms \(\Spec(S) \to \mfrakX\) of formal schemes. We will identify \(\mfrakX\) with the associated stack.
        \item Any stack \(\mfrakX\) on \(\SPerfd^{\nil}\) admits the \(\infty\)-category
        \begin{equation*}
            \mcalD_{\qcoh}(\mfrakX) \defeq \lim_{\Spec(S) \to \mfrakX} \mcalD(S),
        \end{equation*}
        where the limit runs over \(\SPerfd^{\nil}_{/\mfrakX}\).
        This is called the \emph{\(\infty\)-category of quasi-coherent sheaves} on \(\mfrakX\). See \Cref{CompatibilityQCDerived} for this terminology.
        \item Any morphism \(f \colon \mfrakY \to \mfrakX\) of stacks induces the \emph{pullback} functor \(f^* \colon \mcalD_{\qcoh}(\mfrakX) \to \mcalD_{\qcoh}(\mfrakY)\) and it admits a right adjoint \(f_* \colon \mcalD_{\qcoh}(\mfrakY) \to \mcalD_{\qcoh}(\mfrakX)\) called the \emph{pushforward}.
    \end{itemize}
\end{notation}

We record the following well-known compatibility of quasi-coherent sheaves on a formal scheme and quasi-coherent sheaves on the associated stack:

\begin{lemma} \label{CompatibilityQCDerived}
    Let \(\mfrakX\) be a quasi-compact separated bounded \(p\)-adic formal scheme.
    Then there exists a natural equivalence
    \begin{equation*}
        \mcalD_{\qcoh}(\mfrakX_{\Zar}, \mcalO_{\mfrakX}) \xrightarrow{\simeq} \mcalD_{\qcoh}(\mfrakX)
    \end{equation*}
    of \(\infty\)-categories, where the left-hand side is the \(\infty\)-category of quasi-coherent sheaves on \(\mfrakX\) in the sense of \Cref{DefDerivedQcohFormalScheme} and the right-hand side is the \(\infty\)-category of quasi-coherent sheaves on the stack associated to \(\mfrakX\) as above.
\end{lemma}

\begin{proof}
    This is well-known, but we sketch the proof for the convenience of the reader.
    By \Cref{DerivedQCohEquivGeneral}, we have the equivalence
    \begin{equation*}
        \mcalD_{\qcoh}(\mfrakX_{\Zar}, \mcalO_{\mfrakX}) \xrightarrow{\simeq} \lim_{n \geq 1} \mcalD_{\qcoh}(\mfrakX_n, \mcalO_{\mfrakX_n}),
    \end{equation*}
    where \(\mfrakX_n\) is the closed subscheme of \(\mfrakX\) defined by \(p^n\).
    Since \(\mfrakX_n\) is a quasi-compact separated scheme, we have the equivalence \(\mcalD_{\qcoh}(\mfrakX_n, \mcalO_{\mfrakX_n}) \xrightarrow{\simeq} \lim_{\Spec(S) \to \mfrakX_n} \mcalD(S)\) where the limit runs over the category of affine schemes over \(\mfrakX_n\).
    Composing the equivalences above, we have the equivalence
    \begin{equation*}
        \mcalD_{\qcoh}(\mfrakX_{\Zar}, \mcalO_{\mfrakX}) \xrightarrow{\simeq} \lim_{n \geq 1} \lim_{\Spec(S) \to \mfrakX_n} \mcalD(S) \simeq \lim_{\Spec(S) \to \mfrakX} \mcalD(S),
    \end{equation*}
    where the last limit runs through the category of \(p\)-nilpotent rings \(S\) over \(\mfrakX\) and the last equivalence follows from the fact that any morphism \(\Spec(S) \to \mfrakX\) factors through some \(\mfrakX_n\) since \(S\) is \(p\)-nilpotent.
    Composing the equivalences above, we obtain the desired equivalence since \(\SPerfd^{\nil}\) forms a basis for the fpqc topology on the category of \(p\)-nilpotent rings.
\end{proof}

In what follows, we briefly recall and summarize the construction of the perfectization \(\mfrakX^{\pfd}\) and the sheaf \(\mcalO_{\mfrakX, \pfd}\).
For simplicity, in this section we treat only the case of bounded \(p\)-adic formal schemes, and do not consider more general adic formal schemes.
Any results on perfectization are due to Bhatt \cite{bhatt2025Aspects} except for \Cref{ComparisonOXperfd}.

\begin{definition}[{\cite{bhatt2025Aspects}*{\S 4}}] \label{DefPerfectization}
    Let \(\mfrakX\) be a bounded \(p\)-adic formal scheme.
    The \emph{perfectization} \(\mfrakX^{\pfd}\) of \(\mfrakX\) is defined as the functor
    \begin{equation*}
        \mfrakX^{\pfd} \colon \SPerfd^{\nil} \to \Set \hookrightarrow \Ani
    \end{equation*}
    which sends \(S\in \SPerfd^{\nil}\) to the set\footnote{For general stack \(\mfrakX\), the value \(\mfrakX^{\pfd}(S)\) may be a \(1\)-truncated anima. However, in the case of bounded \(p\)-adic formal schemes, it is a set (\cite{bhatt2025Aspects}*{\S 4.2}).} \(\mfrakX^{\pfd}(S)\) of diagrams \(\Spec(S) \xrightarrow{i} \Spf(P) \xrightarrow{u} \mfrakX\) of \(p\)-adic formal schemes, where \(P\) is a \(p\)-adic perfectoid ring and \(i\) induces an isomorphism \(P^{\flat} \xrightarrow{\cong} S^{\flat}\).
    This operation \(\mfrakX \mapsto \mfrakX^{\pfd}\) is functorial in \(\mfrakX\) and commutes with fiber products.

    This \(\mfrakX^{\pfd}\) actually becomes an fpqc sheaf and, forgetting \(\Spf(P)\), we have the natural morphism
    \begin{equation*}
        \pi \colon \mfrakX^{\pfd} \to \mfrakX
    \end{equation*}
    of fpqc sheaves on \(\SPerfd^{\nil}\).
\end{definition}

\begin{proposition}[{\cite{bhatt2025Aspects}*{Proposition 4.4.3 and Corollary 4.4.4}}] \label{GlobalSectionsPerfectization}
    Let \(\mfrakX\) be a bounded \(p\)-adic formal scheme. Then the pullback functors induce an equivalence
    \begin{equation*}
        \mcalD_{\qcoh}(\mfrakX^{\pfd}) \xrightarrow{\simeq} \lim_{\Spf(P) \in \AffPerfd_{\mfrakX}} \mcalD^{\comp{p}}(P)
    \end{equation*}
    of \(\infty\)-categories, where the limit runs over the category \(\AffPerfd_{\mfrakX}\) of affine perfectoid formal schemes over \(\mfrakX\) (\Cref{DefCatPerfd}).
    In particular, we have the isomorphism
    \begin{equation*}
        R_{\perfd} \xrightarrow{\cong} R\Gamma(\Spf(R)^{\pfd}, \mcalO)
    \end{equation*}
    in \(\CAlg(\mcalD(R))\) for any bounded \(p\)-adic ring \(R\).
\end{proposition}

\begin{proof}
    The first statement is \cite{bhatt2025Aspects}*{Proposition 4.4.3}.
    In \cite{bhatt2025Aspects}*{Corollary 4.4.4}, it is assumed that there exists a morphism from a perfectoid ring to \(R\), but the same proof works for the absolute perfectoidization without this assumption since \(R_{\perfd}\) is the limit of all perfectoid \(R\)-algebras.
\end{proof}

The perfectization \(\mfrakX^{\pfd}\) of \(\mfrakX\) admits a well-behaved cover, which is called a \emph{perfectoid atlas} for \(\mfrakX\):

\begin{theorem}[{\cite{bhatt2025Aspects}*{Theorem 4.3.4}}] \label{PerfectoidAtlas}
    Let \(\mfrakX\) be a bounded \(p\)-adic formal scheme.
    Then there exists a representable fpqc cover \(\mfrakU \to \mfrakX^{\pfd}\) of fpqc sheaves on \(\SPerfd^{\nil}\) such that \(\mfrakU\) is representable by a \(p\)-adic perfectoid formal scheme.
    If \(\mfrakX\) is separated, we can take \(\mfrakU\) to be affine.
\end{theorem}

\begin{lemma} \label{QCQSPerfectization}
    Let \(\mfrakX\) be a quasi-compact separated bounded \(p\)-adic formal scheme and let \(\pi \colon \mfrakX^{\pfd} \to \mfrakX\) be the natural morphism.
    Then this morphism \(\pi\) is a qcqs morphism of stacks.
\end{lemma}

\begin{proof}
    This is an easy consequence of the existence of a perfectoid atlas for \(\mfrakX\) in \Cref{PerfectoidAtlas}, but we include the proof for completeness.

    Since \(\mfrakX\) is separated, we can take an affine perfectoid atlas \(\mfrakU \to \mfrakX^{\pfd}\) of \(\mfrakX\) by \Cref{PerfectoidAtlas}, which shows that \(\pi \colon \mfrakX^{\pfd} \to \mfrakX\) is quasi-compact.

    To show that \(\pi\) is quasi-separated, it suffices to show that the diagonal morphism \(\Delta_{\pi} \colon \mfrakX^{\pfd} \to \mfrakX^{\pfd} \times_{\mfrakX} \mfrakX^{\pfd}\) is quasi-compact under the assumption that \(\mfrakX\) is affine.
    Since the functor \(\mfrakX \mapsto \mfrakX^{\pfd}\) commutes with fiber products and \(\mfrakU \simeq \mfrakU^{\pfd}\) holds (\cite{bhatt2025Aspects}*{Lemma 4.3.8}), we have a pullback square
    \begin{equation*}
        \begin{tikzcd}
            (\mfrakU \times_{\mfrakX} \mfrakU)^{\pfd} \simeq \mfrakU \times_{\mfrakX^{\pfd}} \mfrakU \arrow[d] \arrow[r] &  \mfrakU \times_{\mfrakX} \mfrakU \arrow[d]    \\
            \mfrakX^{\pfd} \arrow[r, "\Delta_{\pi}"]                                                                     & \mfrakX^{\pfd} \times_{\mfrakX} \mfrakX^{\pfd}
        \end{tikzcd}
    \end{equation*}
    of stacks whose right vertical morphism is a representable fpqc cover.
    Since \(\mfrakX\) (resp., \(\mfrakU\)) is affine (resp., affine perfectoid), the fiber product \(\mfrakU \times_{\mfrakX} \mfrakU\) is also affine and moreover can be written as \(\Spf(S)\) for some \(p\)-adic semiperfectoid ring \(S\).
    Then the perfectoidization \(S_{\perfd}\) exists and the perfectization \((\mfrakU \times_{\mfrakX} \mfrakU)^{\pfd}\) is isomorphic to \(\Spf(S_{\perfd})\) by \cite{bhatt2025Aspects}*{Example 4.2.1}.
    This implies that the upper morphism in the above diagram is quasi-compact and thus \(\Delta_{\pi}\) is also quasi-compact, as desired.
\end{proof}

\begin{definition}[{\cite{bhatt2025Aspects}*{Notation 4.6.1}}] \label{DefSheafPerfectization}
    Let \(\mfrakX\) be a quasi-compact separated bounded \(p\)-adic formal scheme and let \(\pi \colon \mfrakX^{\pfd} \to \mfrakX\) be the natural morphism.
    The pushforward \(\pi_* \colon \mcalD_{\qcoh}(\mfrakX^{\pfd}) \to \mcalD_{\qcoh}(\mfrakX)\) gives rise to the object
    \begin{equation*}
        \mcalO_{\mfrakX, \pfd} \defeq \pi_*(\mcalO_{\mfrakX^{\pfd}}) \in \mcalD_{\qcoh}(\mfrakX) \simeq \mcalD_{\qcoh}(\mfrakX_{\Zar}, \mcalO_{\mfrakX})
    \end{equation*}
    of quasi-coherent object on \(\mfrakX\), where the equivalence follows from \Cref{CompatibilityQCDerived}.
    This \(\mcalO_{\mfrakX, \pfd}\) has a natural structure of an \(\mcalO_{\mfrakX}\)-algebra in \(\mcalD(\mfrakX_{\Zar}, \mcalO_{\mfrakX})\).
\end{definition}

This \(\mcalO_{\mfrakX, \pfd}\) is isomorphic to our absolute perfectoidization \(\mcalO_{\mfrakX, \perfd}\) for a bounded \(p\)-adic formal scheme \(\mfrakX\):

\begin{theorem} \label{ComparisonOXperfd}
    Let \(\mfrakX\) be a quasi-compact separated bounded \(p\)-adic formal scheme.
    Then there exists a natural isomorphism
    \begin{equation*}
        \mcalO_{\mfrakX, \perfd} \xrightarrow{\cong} \mcalO_{\mfrakX, \pfd}
    \end{equation*}
    in \(\CAlg(\mcalD(\mfrakX_{\Zar}, \mcalO_{\mfrakX}))\).
\end{theorem}

\begin{proof}
    Using \Cref{RGammaCommutesLimits} and \Cref{SectionDerivedPushout}, it suffices to show that there exist natural isomorphisms
    \begin{equation*}
        R\Gamma(\mfrakU, \mcalO_{\mfrakX, \pfd}) \xrightarrow{\cong} R_{\perfd}
    \end{equation*}
    in \(\CAlg_R\) for any affine open subset \(\mfrakU = \Spf(R)\) of \(\mfrakX\).
    Set the inclusion \(\iota \colon \Spf(R) = \mfrakU \hookrightarrow \mfrakX\) and the natural morphisms \(\pi_{\mfrakU} \colon \mfrakU^{\pfd} \to \mfrakU\) and \(\pi \colon \mfrakX^{\pfd} \to \mfrakX\).
    Since \(\pi\) is qcqs by \Cref{QCQSPerfectization}, the flat base change theorem produces a commutative diagram
    \begin{equation*}
        \begin{tikzcd}
            \mcalD_{\qcoh}(\mfrakX^{\pfd}) \arrow[r, "\pi_*"] \arrow[d, "(\iota^{*})^{\pfd}"] & \mcalD_{\qcoh}(\mfrakX) \arrow[d, "\iota^*"] \\
            \mcalD_{\qcoh}(\mfrakU^{\pfd}) \arrow[r, "(\pi_{\mfrakU})_*"]                     & \mcalD_{\qcoh}(\mfrakU)                     
        \end{tikzcd}
    \end{equation*}
    of \(\infty\)-categories, where \((\iota^*)^{\pfd}\) is the pullback along the morphism \(\iota^{\pfd} \colon \mfrakU^{\pfd} \to \mfrakX^{\pfd}\) induced by \(\iota\).
    Applying this diagram to \(\mcalO_{\mfrakX^{\pfd}}\), we have a natural isomorphism
    \begin{equation*}
        \restr{\mcalO_{\mfrakX, \pfd}}{\mfrakU} = \iota^*(\mcalO_{\mfrakX, \pfd}) \xrightarrow{\cong} \mcalO_{\mfrakU, \pfd}
    \end{equation*}
    in \(\CAlg(\mcalD_{\qcoh}(\mfrakU))\).
    Taking the global section, we have a natural isomorphism
    \begin{equation*}
        R\Gamma(\mfrakU, \mcalO_{\mfrakX, \pfd}) \xrightarrow{\cong} R\Gamma(\mfrakU, \mcalO_{\mfrakU, \pfd}) \xrightarrow{\cong} R\Gamma(\mfrakU^{\pfd}, \mcalO) \cong R_{\perfd}
    \end{equation*}
    in \(\CAlg_R\), where the third isomorphism follows from \Cref{GlobalSectionsPerfectization}.
    This is the desired isomorphism.
\end{proof}

\begin{remark} \label{CoherenceFormalSchemes}
    Using the \(p\)-adic Riemann--Hilbert functor (\cite{bhatt2025Aspects}*{Theorem 7.3.6}), if \(\mfrakX\) is a topologically finitely presented \(p\)-adic formal scheme over a perfectoid valuation ring \(\mcalO_K\), then \(\mcalO_{\mfrakX, \perfd} \in \mcalD_{\qcoh}(\mfrakX)\) is almost coherent with respect to the basic setup \((\mcalO_K, \sqrt{p\mcalO_K})\).
    Namely, each cohomology sheaf \(\mscrH^i(\mcalO_{\mfrakX, \perfd})\) is an almost coherent \(\mcalO_{\mfrakX}\)-module in the sense that there exists an affine open covering \(\{\mfrakU_i = \Spf(R_i)\}_{i \in I}\) of \(\mfrakX\) such that \(\mscrH^i(\restr{\mcalO_{\mfrakX, \perfd}}{\mfrakU_i})\) is an almost finitely generated \(R_i\)-module. See \cite{bhatt2025Aspects}*{\S 7.1} and \cite{zavyalov2025Almost}*{\S 4.5 and \S 4.8} for comprehensive study of almost coherent sheaves on formal schemes.
    
    In \cite{bhatt2025Aspects}*{Construction 7.4.1 and Theorem 7.4.2}, this almost coherence implies that the localization \(\mcalO_{\mfrakX, \perfd}[1/p]\) is an object of \(\mcalD^b_{\coh}(\mfrakX_{\eta})\) under the identification \(\mcalD^b_{\acoh}(\mfrakX)[1/p] \simeq \mcalD^b_{\coh}(\mfrakX_{\eta})\), where \(\mfrakX_{\eta}\) is the rigid generic fiber of \(\mfrakX\).
    See also \Cref{CoherenceAlgebraicLocalization} for this coherence result in the algebraic setting not for analytic setting.
\end{remark}

Finally, we give a relationship between the absolute graded perfectoidization and the perfectization of stacks.
Before stating the result, we observe the perfectization of quotient stacks:

\begin{lemma} \label{QuotientStackPerfectization}
    Let \(\mfrakX\) be a quasi-compact separated bounded \(p\)-adic formal scheme with a \(\mfrakD\)-equivariant structure for some group scheme \(\mfrakD\) over \(\Spf(\setZ_p)\).
    Then \(\mfrakX^{\pfd}\) admits a natural \(\mfrakD^{\pfd}\)-equivariant structure and the natural morphism
    \begin{equation*}
        \mfrakX^{\pfd}/\mfrakD^{\pfd} \to (\mfrakX/\mfrakD)^{\pfd}
    \end{equation*}
    is an isomorphism of stacks on \(\SPerfd^{\nil}\).
\end{lemma}

\begin{proof}
    Pick any \(S \in \SPerfd^{\nil}\) and set \(\mfrakY \defeq \mfrakX/\mfrakD\).
    We will write the untilts \(i \colon \Spec(S) \xrightarrow{\text{untilt}} \Spf(P)\) as \((P, i)\) for simplicity.
    By the definition of the perfectization \(\mfrakY^{\pfd}\), we have natural isomorphisms
    \begin{align*}
        \mfrakY^{\pfd}(S) & = \{(\Spec(S) \xrightarrow{\text{untilt}} \Spf(P) \to \mfrakY)\} \simeq \bigsqcup_{(P, i)} \mfrakY(\Spf(P)) \\
        & \simeq \bigsqcup_{(P, i)} \colim_{\Delta^{\opposite}} (\mfrakX \times \mfrakD^{\times \bullet})(\Spf(P)) \simeq \colim_{\Delta^{\opposite}} \bigsqcup_{(P, i)} (\mfrakX \times \mfrakD^{\times \bullet})(\Spf(P)) \\
        & \simeq \colim_{\Delta^{\opposite}} (\mfrakX^{\pfd} \times_{\Spf(\setZ_p)^{\pfd}} (\mfrakD^{\pfd})^{\times \bullet})(S) \simeq (\mfrakX^{\pfd}/\mfrakD^{\pfd})(S)
    \end{align*}
    of anima, naturally.
    So the natural morphism \(\mfrakX^{\pfd}/\mfrakD^{\pfd} \to (\mfrakX/\mfrakD)^{\pfd}\) is an isomorphism, as desired.
\end{proof}

Next, we check the following categorical equivalence:

\begin{lemma} \label{GeometrizationGradedModule}
    Let \(R\) be a \(G\)-graded bounded \(p\)-adic ring for some torsion-free abelian group \(G\).
    Write \(\mfrakX \defeq \Spf(R)\) with a \(\mfrakD \defeq \Spf(\setZ_p[G])\)-equivariant structure induced from the \(G\)-graded structure on \(R\).
    Take the quotient stack \(\mfrakY \defeq \mfrakX/\mfrakD\).
    Then taking the derived global section gives a natural equivalence
    \begin{equation*}
        R\Gamma(\mfrakY, -) \colon \mcalD_{\qcoh}(\mfrakY) \xrightarrow{\simeq} \mcalD_{\graded{G}}^{\comp{p}}(R)
    \end{equation*}
    of \(\infty\)-categories, where the right-hand side is the \(\infty\)-category of derived gradedwise \(p\)-complete \(G\)-graded \(R\)-modules defined in \cite{ishizuka2026Derived}.
    Moreover, we have a commutative diagram
    \begin{equation*}
        \begin{tikzcd}
        \mcalD_{\qcoh}(\mfrakY) \arrow[d, "f^*"'] \arrow[r, "{R\Gamma(\mfrakY, -)}"] & \mcalD_{\graded{G}}^{\comp{p}}(R) \arrow[d, "\dcomp{p}{-}"] \\
        \mcalD_{\qcoh}(\mfrakX) \arrow[r, "{R\Gamma(\mfrakX, -)}"]                  & \mcalD^{\comp{p}}(R),                                       
        \end{tikzcd}
    \end{equation*}
    where the left vertical morphism is the pullback along the natural morphism \(f \colon \mfrakX \to \mfrakY\) and the right vertical morphism is the derived \(p\)-completion functor.
\end{lemma}

\begin{proof}
    Let \(f \colon \mfrakX \to \mfrakY\) be the canonical morphism.
    Consider the adjunction
    \begin{equation*}
        f^* \colon \mcalD_{\qcoh}(\mfrakY) \rightleftarrows \mcalD_{\qcoh}(\mfrakX) \colon f_* .
    \end{equation*}
    Since \(f\) is an fpqc cover, \(f^*\) is conservative. Moreover, since we have an fpqc descent \(\mcalD_{\qcoh}(\mfrakY) \xrightarrow{\simeq} \lim_{\Delta} \mcalD_{\qcoh}(\mfrakX^{\bullet})\), any \(f^*\)-split cosimplicial object in \(\mcalD_{\qcoh}(\mfrakY)\) admits a limit and \(f^*\) preserves this limit.
    Thus the Barr--Beck--Lurie theorem (\cite{lurie2017Higher}*{Proposition 4.7.3.3 and Theorem 4.7.3.5}) identifies \(\mcalD_{\qcoh}(\mfrakY)\) with the \(\infty\)-category of comodules \(\coMod_{K}(\mcalD_{\qcoh}(\mfrakX))\) over the comonad
    \begin{equation*}
        K\defeq f^*f_* \colon \mcalD_{\qcoh}(\mfrakX)\longrightarrow \mcalD_{\qcoh}(\mfrakX).
    \end{equation*}
    on \(\mcalD_{\qcoh}(\mfrakX)\). Under the equivalence \(R\Gamma(\mfrakX,-) \colon \mcalD_{\qcoh}(\mfrakX) \xrightarrow{\simeq} \mcalD^{\comp{p}}(R)\) in \Cref{DerivedQCohEquivAffine}, we compute this comonad as follows. Let
    \begin{equation*}
        p_1, p_2 \colon \mfrakX \times_{\mfrakY} \mfrakX \longrightarrow \mfrakX
    \end{equation*}
    be the two projections.
    Under the isomorphism \(\mfrakX \times_{\mfrakY} \mfrakX \cong \mfrakX \times \mfrakD \cong \Spf(R[G])\), we may assume that \(p_1\) is the projection to the first factor and \(p_2\) is the morphism induced by the \(\mfrakD\)-action on \(\mfrakX\), which corresponds to the coaction \(\rho_R \colon R \longrightarrow R[G]\) corresponding to the \(G\)-grading on \(R\).
    By the flat base change theorem, \(K\) is equivalent to \(p_{2,*}p_1^*\).

    Hence, for the commutative diagram
    \begin{equation*}
        \begin{tikzcd}
        \mcalD_{\qcoh}(\mfrakX) \arrow[rr, "p_1^*"] \arrow[d, "\simeq"] \arrow[rrrr, "K"', bend left] \arrow[d, "{R\Gamma(\mfrakX, -)}"'] &  & \mcalD_{\qcoh}(\mfrakX \times \mfrakD) \arrow[rr, "{p_{2, *}}"] \arrow[d, "\simeq"'] \arrow[d, "{R\Gamma(\mfrakX \times \mfrakD, -)}"] &  & \mcalD_{\qcoh}(\mfrakX) \arrow[d, "\simeq"'] \arrow[d, "{R\Gamma(\mfrakX, -)}"] \\
        \mcalD^{\comp{p}}(R) \arrow[rr, "{\dcomp{p}{- \otimes^L_R R[G]}}"] \arrow[rrrr, "\mcalT^p", bend right]                           &  & {\mcalD^{\comp{p}}(R[G])} \arrow[rr, "{\rho_{R, *}}"]                                                                                  &  & {\mcalD^{\comp{p}}(R),}                                                        
        \end{tikzcd}
    \end{equation*}
    the comonad \(K\) on \(\mcalD_{\qcoh}(\mfrakX)\) corresponds to the comonad \(\mcalT^p \colon \rho_{R, *}(- \otimes^L_R R[G])\) on \(\mcalD^{\comp{p}}(R)\) given in the proof of the categorical equivalence
    \begin{equation*}
        \mcalD_{\graded{G}}^{\comp{p}}(R) \xrightarrow{\simeq} \coMod_{G}^{\comp{p}}(R) \defeq \coMod_{\mcalT^p}(\mcalD^{\comp{p}}(R))
    \end{equation*}
    in \cite{ishizuka2026Derived}*{Theorem 6.23, Corollary 6.6, and Definition 6.11}.
    It follows the categorical equivalence
    \begin{equation*}
        R\Gamma(\mfrakY, -) \colon \mcalD_{\qcoh}(\mfrakY) \xrightarrow{\simeq} \coMod_{K}(\mcalD_{\qcoh}(\mfrakX)) \xrightarrow{\simeq} \coMod_{\mcalT^p}(\mcalD^{\comp{p}}(R)) \xleftarrow{\simeq} \mcalD_{\graded{G}}^{\comp{p}}(R).
    \end{equation*}
    
    It remains to check the compatibility with pullback to \(\mfrakX\). Under the comonadic identification above, the pullback functor \(f^*\colon \mcalD_{\qcoh}(\mfrakY) \to \mcalD_{\qcoh}(\mfrakX)\) corresponds to the forgetful functor \(\coMod_K(\mcalD_{\qcoh}(\mfrakX)) \to \mcalD_{\qcoh}(\mfrakX)\).
    On the other hand, under the equivalences above, this forgetful functor is identified with the composition
    \begin{equation*}
        \mcalD_{\graded{G}}^{\comp{p}}(R) \xrightarrow{\simeq} \coMod_{G}^{\comp{p}}(R) \to \mcalD^{\comp{p}}(R),
    \end{equation*}
    which is precisely the derived \(p\)-completion functor \(\dcomp{p}{-}\) by \cite{ishizuka2026Derived}*{Construction 6.2}.
    Therefore, the desired diagram commutes.
\end{proof}

\begin{proposition} \label{StackyGradedPerfd}
    Let \(R\) be a \(G\)-graded bounded \(p\)-adic ring for some torsion-free abelian group \(G\) with \(G = G[1/p]\).
    Take a natural representable fpqc cover
    \begin{equation*}
        f \colon \mfrakX \defeq \Spf(R) \to \mfrakY \defeq \mfrakX/\mfrakD
    \end{equation*}
    of fpqc stacks on \(\SPerfd^{\nil}\) induced from the \(G\)-graded structure on \(R\), where \(\mfrakD \defeq \Spf(\setZ_p[G])\) is the \(p\)-adic formal completion of the associated affine flat group scheme over \(\Spf(\setZ_p)\).
    Then we have a natural isomorphism
    \begin{equation*}
        R\Gamma(\mfrakY^{\pfd}, \mcalO_{\mfrakY^{\pfd}}) \xrightarrow{\cong} R_{\grpfd}
    \end{equation*}
    in \(\CAlg(\mcalD_{\graded{G}}^{\comp{p}}(R))\).
\end{proposition}

\begin{proof}
    Consider the commutative diagram of stacks
    \begin{equation} \label{DiagramStackyPullback}
        \begin{tikzcd}
        \mfrakX^{\pfd} \arrow[d, "\pi_{\mfrakX}"'] \arrow[r, "f^{\pfd}"] & \mfrakY^{\pfd} \arrow[d, "\pi_{\mfrakY}"] \\
        \mfrakX \arrow[r, "f"]                                       & \mfrakY,                                
        \end{tikzcd}
    \end{equation}
    where the upper morphism \(f^{\pfd}\) is the morphism induced by \(f\) and the vertical morphisms are the natural morphisms from the perfectization.
    The lower horizontal morphism \(f \colon \mfrakX \to \mfrakY\) is a \(\mfrakD\)-torsor and the upper horizontal morphism \(f^{\pfd} \colon \mfrakX^{\pfd} \to \mfrakY^{\pfd}\) is a \(\mfrakD^{\pfd}\)-torsor by \Cref{QuotientStackPerfectization}.
    We will show that the above diagram is a pullback square of stacks.

    Take the induced morphism
    \begin{equation*}
        g \colon \mfrakX^{\pfd} \to \mfrakX \times_{\mfrakY} \mfrakY^{\pfd}.
    \end{equation*}
    To show this is an isomorphism, we first consider the following commutative diagram
    \begin{equation*}
        \begin{tikzcd}
        \mfrakD^{\pfd} \times_{\Spf(\setZ_p)^{\pfd}} \mfrakX^{\pfd} \arrow[d] \arrow[r, "g'"] & \mfrakD \times_{\Spf(\setZ_p)} \mfrakX^{\pfd} \arrow[d] \arrow[r]                          & \mfrakX^{\pfd} \arrow[d, "g"]                     \\
        \mfrakD \times_{\Spf(\setZ_p)} \mfrakX^{\pfd} \arrow[d] \arrow[r]               & \mfrakD \times_{\Spf(\setZ_p)} \mfrakX \times_{\mfrakY} \mfrakY^{\pfd} \arrow[d] \arrow[r] & \mfrakX \times_{\mfrakY} \mfrakY^{\pfd} \arrow[d] \\
        \mfrakX^{\pfd} \arrow[r, "g"]                                                   & \mfrakX \times_{\mfrakY} \mfrakY^{\pfd} \arrow[r]                                          & \mfrakY^{\pfd}                                   
        \end{tikzcd}
    \end{equation*}
    whose squares are all pullback squares: This follows from the quotient stacks \(\mfrakX \to \mfrakY\) and \(\mfrakX^{\pfd} \to \mfrakY^{\pfd}\).
    By the fpqc descent, it suffices to show that the morphism
    \begin{equation*}
        g' \colon \mfrakD^{\pfd} \times_{\Spf(\setZ_p)^{\pfd}} \mfrakX^{\pfd} \to \mfrakD \times_{\Spf(\setZ_p)} \mfrakX^{\pfd}
    \end{equation*}
    induced by \(g\) is an isomorphism.
    Pick any \(S \in \SPerfd^{\nil}\).
    Note that for any untilt \(P \to S\) of \(S\), we have identifications of the sets of multiplicative characters of \(G\);
    \begin{equation*}
        \Hom(G, P^{\times}) \cong \Hom(G, (P^{\flat})^{\times}) \cong \Hom(G, (S^{\flat})^{\times})
    \end{equation*}
    since \(G = G[1/p]\) holds.
    Therefore, the functor of points \(\mfrakD^{\pfd}(S)\) is identified with the set of pairs \((P \to S, G \to (S^\flat)^{\times})\) where \(P \to S\) is an untilt of \(S\) and \(G \to (S^{\flat})^{\times}\) is a multiplicative character of \(G\).
    So \(g'(S)\) identifies with the target and source of the morphism as the same set
    \begin{equation*}
        \{(G \to (S^\flat)^{\times}, \Spec(S) \to \Spf(P) \to \mfrakX)\}
    \end{equation*}
    where \(\Spec(S) \to \Spf(P)\) is an untilt of \(S\) and \(G \to (S^\flat)^{\times}\) is a multiplicative character of \(G\).
    This shows that \(g'(S)\) is an isomorphism for any \(S \in \SPerfd^{\nil}\) and thus \(g'\) is an isomorphism, as desired.

    Since \(\mfrakX^{\pfd}\) has a representable fpqc cover from an affine perfectoid formal scheme and \(\mfrakX \to \mfrakY\) is a \(\mfrakD\)-torsor, the pullback square \eqref{DiagramStackyPullback} shows that \(\mfrakY^{\pfd}\) admits a representable fpqc cover from an affine perfectoid formal scheme.
    In particular, \(\pi_{\mfrakY} \colon \mfrakY^{\pfd} \to \mfrakY\) is qcqs and we can apply the flat base change theorem.
    
    Then we have a canonical isomorphism
    \begin{equation*}
        f^*\pi_{\mfrakY, *}(\mcalO_{\mfrakY^{\pfd}}) \xrightarrow{\cong} \pi_{\mfrakX, *}(f^{\pfd, *}(\mcalO_{\mfrakY^{\pfd}})) \cong \pi_{\mfrakX, *}(\mcalO_{\mfrakX^{\pfd}}) 
    \end{equation*}
    in \(\mcalD_{\qcoh}(\mfrakX)\).
    Using \Cref{GeometrizationGradedModule}, the above isomorphism corresponds to the isomorphism
    \begin{equation*}
        \dcomp{p}{R\Gamma(\mfrakY^{\pfd}, \mcalO_{\mfrakY^{\pfd}})} \xrightarrow{\cong} R_{\perfd}
    \end{equation*}
    in \(\CAlg_R\).
    Using \Cref{CompletionOfAbsoluteGradedPerfd} and the derived gradedwise \(p\)-completeness of \(R\Gamma(\mfrakY^{\pfd}, \mcalO_{\mfrakY^{\pfd}})\), we have the isomorphism
    \begin{equation*}
        R\Gamma(\mfrakY^{\pfd}, \mcalO_{\mfrakY^{\pfd}}) \xrightarrow{\cong} R_{\grpfd}
    \end{equation*}
    in \(\CAlg(\mcalD_{\graded{G}}^{\comp{p}}(R))\) and thus the desired isomorphism.
\end{proof}

\section{Algebraization of absolute perfectoidization}

In this section, we will work in the following setting:

\begin{notation} \label{SettingFibersequence}
    Let \(X\) be a quasi-compact scheme whose canonical morphism \(X \to \Spec(H^0(X, \mcalO_{X}))\) is universally closed and let \(L\) be an ample line bundle on \(X\) together with a fixed inverse line bundle \(L^{-1}\) and an isomorphism \(L \otimes_{\mcalO_X} L^{-1} \xrightarrow{\cong} \mcalO_X\).
    Set the section ring
    \begin{equation*}
        R \defeq \bigoplus_{n \geq 0} H^0(X, L^{\otimes n}) \quad \text{with its homogeneous ideal} \quad R_+ \defeq \bigoplus_{n > 0} H^0(X, L^{\otimes n})
    \end{equation*}
    and the projective spectrum \(\Proj(R)\) of \(R\).
    By the following lemma (\Cref{ProjIsom}), we will identify \(X\) with \(\Proj(R)\) and \(L\) on \(X\) with \(\mcalO_{\Proj(R)}(1)\) on \(\Proj(R)\).

    Moreover, we fix a finitely generated homogeneous ideal \(I\) of \(R\) containing \(p\) and we equip \(R\) with the \(I\)-adic topology.
    We will write \(\hProj(R)\) for the \(\widetilde{I}\)-adic formal completion of the projective spectrum of \(R\), where \(\widetilde{I}\) is the ideal sheaf on \(\Proj(R)\) associated to \(I\).
    This formal scheme is adic, quasi-compact, and separated such that \(p\) is topologically nilpotent.

    We will denote by \(X_n\) the closed subscheme of \(X\) defined by \(\widetilde{I}^n\) and by \(i_n \colon X_n \to \mscrX\) the morphism of ringed spaces whose underlying morphism of topological spaces is the identity and whose morphism of structure sheaves is the natural morphism \(\mcalO_{\mscrX} \to i_{n, *}(\restr{\mcalO_X}{X_n}) \cong \mcalO_X/\widetilde{I}^n\) for every \(n \geq 1\).
\end{notation}

\begin{lemma} \label{ProjIsom}
    In the above setting (\Cref{SettingFibersequence}), the natural morphism \(r \colon X \to \Proj(R)\) is an isomorphism and this induces an isomorphism \(\mcalO_{\Proj(R)}(n) \xrightarrow{\cong} L^{\otimes n}\) of line bundles on \(X\) for every \(n \in \setZ\).
\end{lemma}

\begin{proof}
    The first assertion is \citeSta{0C6J}.

    Using \citeSta{01QJ} for the quasi-coherent sheaf \(L^{\otimes k}\) on \(X\), we have an isomorphism
    \begin{equation*}
        r^*\widetilde{R'(k)} \xrightarrow{\cong} L^{\otimes k}
    \end{equation*}
    for every \(k \in \setZ\), where \(\widetilde{R'(k)}\) is the quasi-coherent sheaf on \(\Proj(R)\) associated to the \(\setZ\)-graded \(R\)-module \(R'(k) \defeq \bigoplus_{n \in \setZ} H^0(X, L^{\otimes n+k})\).
    Since \(\widetilde{R'(k)} \cong \widetilde{R(k)} = \mcalO_{\Proj(R)}(k)\), we obtain the isomorphism \(\mcalO_{\Proj(R)}(k) \xrightarrow{\cong} L^{\otimes k}\) of \(\mcalO_X\)-modules for every \(k \in \setZ\).
\end{proof}

\begin{definition} \label{AssociatedSheafFormalGradedRecall}
    In the above setting (\Cref{SettingFibersequence}), following \Cref{AssociatedSheafFormalGraded}, we can associate a quasi-coherent \(\mcalO_{\hProj(R)}\)-module \(M^{\Delta}\) to any \(\setQ\)-graded \(R\)-module \(M\) by the formula
    \begin{equation*}
        M^{\Delta} \defeq \varprojlim_{n} \widetilde{M/I^nM}; \quad D_+(f) \mapsto (\grcomp{I}{M[1/f]})_0
    \end{equation*}
    in \(\Shv(\hProj(R)_{\Zar}, \mcalO_{\hProj(R)})\), where \(\widetilde{M/I^nM}\) is the quasi-coherent \(\mcalO_{\Proj(R/I^n)}\)-module associated to the \(\setQ\)-graded \(R/I^n\)-module \(M/I^nM\).
    This is the completion of the quasi-coherent \(\mcalO_{\Proj(R)}\)-module \(\widetilde{M}\) associated to \(M\) with respect to the closed subscheme \(V_+(I)\) of \(\Proj(R)\) defined by \(I\) (\Cref{AssociatedSheafFormalGraded}).
    See \Cref{DefCompletionQcoh} for the case of affine formal schemes.
\end{definition}

\begin{lemma} \label{TwistShiftCompletion}
    In the above setting (\Cref{SettingFibersequence}), any \(\setQ\)-graded \(R\)-module \(M\) admits the canonical isomorphism
    \begin{equation*}
        M^{\Delta} \otimes_{\mcalO_{\hProj(R)}} \mcalO_{\hProj(R)}(m) \eqdef M^{\Delta}(m) \xrightarrow{\cong} M(m)^{\Delta}
    \end{equation*}
    of \(\mcalO_{\hProj(R)}\)-modules for every \(m \in \setZ\), where \(\mcalO_{\hProj(R)}(m)\) is the \(\widetilde{I}\)-completion of \(\mcalO_{\Proj(R)}(m)\).
\end{lemma}

\begin{proof}
    Since \(\mcalO_{\Proj(R)}(\pm 1)\cong L^{\pm 1}\) are invertible line bundles, there exists a finite open covering \(\{D_+(f_i)\}_{i=1}^r\) of \(\Proj(R)\) with a trivialization \(\restr{\mcalO_{\Proj(R)}(\pm 1)}{D_+(f_i)} \xrightarrow{\cong} \mcalO_{D_+(f_i)}\) of \(\mcalO_{D_+(f_i)}\)-modules.
    In particular, we have isomorphisms
    \begin{equation*}
        (R[1/f_i])_{1} \xrightarrow{\cong} R_{(f_i)}; \quad \exists \varpi_i \mapsto 1, \quad \text{and} \quad (R[1/f_i])_{-1} \xrightarrow{\cong} R_{(f_i)}; \quad \exists \varpi_i' \mapsto 1,
    \end{equation*}
    which send elements \(\varpi_i \in (R[1/f_i])_1\) and \(\varpi_i' \in (R[1/f_i])_{-1}\) to \(1 \in R_{(f_i)}\) for every \(i\).
    We fixed an isomorphism \(L \otimes_{\mcalO_X} L^{-1} \xrightarrow{\cong} \mcalO_X\) and thus the \(R_{(f_i)}\)-module isomorphism
    \begin{equation*}
        R_{(f_i)} \xrightarrow{\cong} R_{(f_i)} \otimes_{R_{(f_i)}} R_{(f_i)} \xrightarrow{\cong} (R[1/f_i])_1 \otimes_{R_{(f_i)}} (R[1/f_i])_{-1} \xrightarrow{\cong} R_{(f_i)}
    \end{equation*}
    sends \(1 \in R_{(f_i)}\) to \(\varpi_i \otimes \varpi_i'\) and thus we have \(\varpi_i \otimes \varpi_i' = 1\) in \(R_{(f_i)}\).
    Write \(\varpi_i'\) as \(\varpi_i^{-1} \in (R[1/f_i])_{-1}\) for every \(i\).

    Consider the following well-defined morphism of \(R_{(f_i)}\)-modules:
    \begin{equation*}
        (M[1/f_i])_m \to M_{(f_i)} \otimes_{R_{(f_i)}} (R[1/f_i])_m; \quad \frac{x}{f^k} \mapsto x \varpi_i^{-\deg(x)} \otimes \frac{\varpi_i^{\deg(x)}}{f^k},
    \end{equation*}
    where \(x \in M\) is a homogeneous element and \(k \in \setZ_{\geq 0}\).
    This induces the inverse morphism of the canonical morphism and thus we obtain the isomorphism
    \begin{equation*}
        \widetilde{M}(m)(D_+(f_i)) = M_{(f_i)} \otimes_{R_{(f_i)}} (R[1/f_i])_m \xrightarrow{\cong} (M[1/f_i])_m = \widetilde{M(m)}(D_+(f_i))
    \end{equation*}
    of \(R_{(f_i)}\)-modules for every \(i\) and \(m \in \setZ\).
    Since \(D_+(f_i)\) covers \(\Proj(R)\), we obtain the isomorphism
    \begin{equation*}
        \widetilde{M}(m) \xrightarrow{\cong} \widetilde{M(m)}
    \end{equation*}
    of quasi-coherent \(\mcalO_{\Proj(R)}\)-modules.
    Applying this for \(M/I^nM\) and taking the limit over \(n\), we obtain the isomorphism
    \begin{equation*}
        \lim_{n \geq 1} (\widetilde{M/I^nM} \otimes_{\mcalO_{\hProj(R)}} \mcalO_{\hProj(R)}(m)) \cong \lim_{n \geq 1} (\widetilde{M/I^nM}(m)) \xrightarrow{\cong} \lim_{n \geq 1} \widetilde{M/I^nM(m)} = M(m)^{\Delta}.
    \end{equation*}
    Using \Cref{TensorCommutesWithLimitInv} below, the left-hand side is isomorphic to \(M^{\Delta} \otimes_{\mcalO_{\hProj(R)}} \mcalO_{\hProj(R)}(m)\), and thus we obtain the desired isomorphism.
\end{proof}

\begin{lemma} \label{TensorCommutesWithLimitInv}
    Fix an adic formal scheme \(\mscrX\).
    Let \(\{\mcalF_i\}_{i \in I}\) be a set of objects in \(\mcalD(\mscrX, \mcalO_{\mscrX})\) with a small index set \(I\) and let \(\mcalP\) be a perfect complex of \(\mcalO_{\mscrX}\)-modules (\citeSta{08CL}).
    Then the canonical morphism
    \begin{equation*}
        (\lim_{i \in I} \mcalF_i) \otimes^L_{\mcalO_{\mscrX}} \mcalP \to \lim_{i \in I} (\mcalF_i \otimes^L_{\mcalO_{\mscrX}} \mcalP)
    \end{equation*}
    is an isomorphism in \(\mcalD(\mscrX, \mcalO_{\mscrX})\).
\end{lemma}

\begin{proof}
    The canonical morphism is given globally. Since \(\mcalP\) is locally isomorphic to a finite complex of locally finite free modules, we can reduce to the case that \(\mcalP\) is so.
    Using the induction on the length of complexes, we can assume that \(\mcalP\) is a locally finite free module.
    Again locally, it is finite free, and then the statement follows from the fact that tensoring with a finite free module commutes with arbitrary limits in \(\mcalD(\mscrX, \mcalO_{\mscrX})\).
\end{proof}

\begin{proposition} \label{ExtensionPerfdProj}
    In the above setting (\Cref{SettingFibersequence}), take a homogeneous element \(f \in R_+\) and take any perfectoid \(R_{(f)}\)-algebra \(Q\).
    Then there exists a \(\setZ[1/p]\)-graded perfectoid \(R[1/f]\)-algebra \(Q'\) such that \(Q'_n \cong Q \otimes_{R_{(f)}} (R[1/f])_n\) as \(R_{(f)}\)-modules for every \(n \in \setZ\).
\end{proposition}

\begin{proof}
    Since \(L\) is an invertible \(\mcalO_X\)-module, the base change
    \begin{equation*}
        Q_n \defeq Q \otimes_{R_{(f)}} H^0(D_+(f), L^{\otimes n}) \cong Q \otimes_{R_{(f)}} (R[1/f])_n
    \end{equation*}
    is an invertible \(Q\)-module for every \(n \in \setZ\), where the isomorphism follows from the isomorphism \(L^{\otimes n} \cong \mcalO_{\Proj(R)}(n)\) in \Cref{ProjIsom}.
    Using the fixed isomorphism \(L \otimes_{\mcalO_X} L^{-1} \xrightarrow{\cong} \mcalO_X\), we have the isomorphism of invertible \(R_{(f)}\)-modules;
    \begin{equation*}
        H^0(D_+(f), L^{\otimes n}) \otimes_{R_{(f)}} H^0(D_+(f), L^{\otimes m}) \xrightarrow{\cong} H^0(D_+(f), L^{\otimes (n+m)})
    \end{equation*}
    for every \(n, m \in \setZ\).
    Taking the tensor product with \(Q\) over \(R_{(f)}\), we obtain the isomorphism of invertible \(Q\)-modules
    \begin{equation*}
        Q_n \otimes_Q Q_m \xrightarrow{\cong} Q_{n+m}
    \end{equation*}
    for every \(n, m \in \setZ\).
    This endows
    \begin{equation*}
        P \defeq \bigoplus_{n \in \setZ} Q_n
    \end{equation*}
    with the structure of a \(\setZ\)-graded ring together with a \(\setZ\)-graded ring homomorphism
    \begin{equation*}
        \bigoplus_{n \in \setZ} H^0(D_+(f), L^{\otimes n}) \to P.
    \end{equation*}
    By the identification \(L^{\otimes n}\cong \mcalO_{\Proj(R)}(n)\) from \Cref{ProjIsom}, the left-hand side is isomorphic to \(R[1/f]\) as a \(\setZ\)-graded \(R\)-algebra, and thus we can regard \(P\) as a \(p\)-adically gradedwise complete \(\setZ\)-graded \(R[1/f]\)-algebra such that \(P_0 = Q\) and satisfies the assumption in \Cref{EnlargeGradedPerfdNew}.

    By \Cref{EnlargeGradedPerfdNew}, there exists a \(\setZ[1/p]\)-graded perfectoid \(R[1/f]\)-algebra \(Q'\) such that \(Q'_0 \cong Q\) as \(R_{(f)}\)-algebras, as desired.
\end{proof}

\begin{theorem} \label{PerfdProjLimit}
    In the above setting (\Cref{SettingFibersequence}), the canonical morphism
    \begin{equation} \label{PerfdProjLimitEq}
        \mcalO_{\hProj(R), \perfd} \to \lim_{P \in \Perfd^{\wedge I}_{\graded{\setZ[1/p]}}(R)} P^{\Delta}
    \end{equation}
    in \(\CAlg(\mcalD(\hProj(R)_{\Zar}, \mcalO_{\hProj(R)}))\) is an isomorphism.
    In particular, we have an isomorphism
    \begin{equation*}
        R\Gamma(\hProj(R), \mcalO_{\hProj(R), \perfd}) \cong R\Gamma(\hProj(R), \lim_{P \in \Perfd^{\wedge I}_{\graded{\setZ[1/p]}}(R)} P^{\Delta})
    \end{equation*}
    of \(\setE_{\infty}\)-\(R\Gamma(\hProj(R), \mcalO_{\hProj(R)})\)-algebras in \(\mcalD(\setZ)\).
\end{theorem}

\begin{proof}
    For notational simplicity, we will write \(\mscrX \defeq \hProj(R)\).
    As in the proof of \Cref{GlobalPerfectoidizationSectionLimit}, it suffices to show that the morphism
    \begin{equation*}
        \mcalO^d_{\hProj(R), \perfd} \to \lim_{P \in \Perfd^{\wedge I}_{\graded{\setZ[1/p]}}(R)} R\Gamma(P^{\Delta})
    \end{equation*}
    in \(\CAlg_{\mcalO^d_{\mscrX}}(\Shv(\mscrX_{\Zar}, \mcalD(\setZ)))\) is an isomorphism by \Cref{PushforwardDerivedStructureSheaf} and \Cref{RGammaCommutesLimits}.
    For each homogeneous element \(f \in R_+\), the projection induces the morphism
    \begin{align} \label{PerfdProjLimitEq2}
        & R\Gamma(D_+(f), \mcalO_{\mscrX, \perfd}) \cong \mcalO_{\mscrX}(D_+(f))_{\perfd} \cong \lim_{Q \in \Perfd^{\wedge I}(R_{(f)})} Q \\
        & \to \lim_{Q' \in \Perfd^{\wedge I}_{\graded{\setZ[1/p]}}(R[1/f])} Q'_0 \cong \lim_{P \in \Perfd^{\wedge I}_{\graded{\setZ[1/p]}}(R)} (\grcomp{I}{P[1/f]})_0 = \lim_{P \in \Perfd^{\wedge I}_{\graded{\setZ[1/p]}}(R)} (P^{\Delta}(D_+(f))) \nonumber \\
        & \cong \lim_{P \in \Perfd^{\wedge I}_{\graded{\setZ[1/p]}}(R)} R\Gamma(D_+(f), P^{\Delta}) \cong R\Gamma(D_+(f), \lim_{P \in \Perfd^{\wedge I}_{\graded{\setZ[1/p]}}(R)} P^{\Delta}) \nonumber
    \end{align}
    of \(\setE_{\infty}\)-\(R_{(f)}\)-algebras where the first isomorphism follows from \Cref{SectionDerivedPushout} and the last two isomorphisms follow from \Cref{RGammaCommutesLimits} and \Cref{AffineVanishing}.
    The morphism above is an isomorphism for any \(f \in R_+\) by \Cref{ExtensionPerfdProj} and thus we obtain the isomorphism
    \begin{equation*}
        \mcalO^d_{\mscrX, \perfd} \xrightarrow{\cong} \lim_{P \in \Perfd^{\wedge I}_{\graded{\setZ[1/p]}}(R)} R\Gamma(P^{\Delta})
    \end{equation*}
    in \(\CAlg_{\mcalO^d_{\mscrX}}(\Shv(\mscrX_{\Zar}, \mcalD(\setZ)))\), which is the desired isomorphism.
    Taking the global section, the second isomorphism follows.
\end{proof}

\begin{definition} \label{DefTwistedPerfdProj}
    In the above setting (\Cref{SettingFibersequence}), we define the twisted absolute perfectoidization
    \begin{equation*}
        \mcalO_{\hProj(R), \perfd}(m) \defeq \mcalO_{\hProj(R), \perfd} \otimes^L_{\mcalO_{\hProj(R)}} \mcalO_{\hProj(R)}(m) \in \mcalD(\hProj(R)_{\Zar}, \mcalO_{\hProj(R)})
    \end{equation*}
    for every \(m \in \setZ\).
\end{definition}

\begin{corollary} \label{TwistAbsPerfd}
    In the above setting (\Cref{SettingFibersequence}), we have an isomorphism
    \begin{equation*}
        \mcalO_{\hProj(R), \perfd}(m) \xrightarrow{\cong} \lim_{P \in \Perfd^{\wedge I}_{\graded{\setZ[1/p]}}(R)} P(m)^{\Delta}
    \end{equation*}
    in \(\mcalD(\hProj(R), \mcalO_{\hProj(R)})\) for every \(m \in \setZ\).
\end{corollary}

\begin{proof}
    Taking the derived tensor product with \(\mcalO_{\mscrX}(m)\) over \(\mcalO_{\mscrX}\), the isomorphism \eqref{PerfdProjLimitEq} induces an isomorphism
    \begin{equation*}
        \mcalO_{\mscrX, \perfd}(m) \xrightarrow{\cong} (\lim_{P \in \Perfd^{\wedge I}_{\graded{\setZ[1/p]}}(R)} P^{\Delta}) \otimes^L_{\mcalO_{\mscrX}} \mcalO_{\mscrX}(m) \xrightarrow{\cong} \lim_{P \in \Perfd^{\wedge I}_{\graded{\setZ[1/p]}}(R)} (P^{\Delta}(m))
    \end{equation*}
    in \(\mcalD(\mscrX, \mcalO_{\mscrX})\), where the second isomorphism follows from \Cref{TensorCommutesWithLimitInv} since \(\mcalO_{\mscrX}(m)\) is an invertible sheaf.
    Using the canonical isomorphism in \Cref{TwistShiftCompletion}, we have the isomorphism
    \begin{equation*}
        \mcalO_{\mscrX, \perfd}(m) \xrightarrow{\cong} \lim_{P \in \Perfd^{\wedge I}_{\graded{\setZ[1/p]}}(R)} P(m)^{\Delta}
    \end{equation*}
    in \(\mcalD(\mscrX, \mcalO_{\mscrX})\) for every \(m \in \setZ\) as desired.
    \qedhere

    
\end{proof}

\begin{theorem} \label{AlgebraizableOXperfd}
    Keep the setting of \Cref{SettingFibersequence}.
    Assume further that the ideal \(I\) is generated by a weakly proregular sequence in \(R_0\) (\Cref{DefWeaklyProregular}).
    Fix an open subscheme \(U\) of \(X\) and its \(I\)-adic formal completion \(\mscrU\) as an open formal subscheme of \(\mscrX\).
    Then there exists an object \(\mcalO_{U,\perfd}\in \mcalD_{\qcoh}(U)\) together with an isomorphism
    \begin{equation*}
        \lim_{n \geq 1} Li_n^*(\mcalO_{U, \perfd}) \cong \mcalO_{\mscrU, \perfd}
    \end{equation*}
    in \(\mcalD(\mscrU_{\Zar}, \mcalO_{\mscrU})\), which is natural in \(U\).
\end{theorem}

\begin{proof}
    It is enough to show the existence of \(\mcalO_{U, \perfd}\) in the case that \(U = X\) since the restriction of \(\mcalO_{X, \perfd}\) to \(U\) will satisfy the desired property.
    
    Taking the derived functor of the exact functor \(\widetilde{(-)} \colon \Mod_{\graded{\setZ}}(R) \to \QCoh(X)\), we have a functor
    \begin{equation*}
        \widetilde{(-)} \colon \mcalD_{\graded{\setZ}}(R) \to \mcalD_{\qcoh}(X)
    \end{equation*}
    of derived \(\infty\)-categories.
    Set
    \begin{align*}
        \mcalO_{X, \perfd} & \defeq \widetilde{R_{\tgrpfd}} \in \mcalD_{\qcoh}(X) \quad \text{and} \\
        \mcalO_{X, \perfd}^{\wedge} & \defeq \lim_{n \geq 1} \restr{\mcalO_{X, \perfd}}{X_n} \in \mcalD_{\qcoh}(\mscrX),
    \end{align*}
    where the latter denotes the object corresponding to the system \(\{\restr{\mcalO_{X,\perfd}}{X_n}\}_{n\geq 1}\) in \(\lim_{n \geq 1} \mcalD_{\qcoh}(X_n)\) along the equivalence in \Cref{DerivedQCohEquivGeneral} together with the weak proregularity of \(I\).
    Note that the associated ideal sheaf \(\widetilde{I}\) of \(I\) on \(X\) is locally generated by a weakly proregular sequence since \(I\) is generated by a weakly proregular sequence in \(R_0\).

    For any homogeneous element \(f \in R_+\) and \(D_+(f) \subseteq \mscrX\), we have isomorphisms
    \begin{align*}
        R\Gamma(D_+(f), \mcalO_{X, \perfd}^{\wedge}) & \cong \lim_{n \geq 1} R\Gamma(D_+(f), \restr{\mcalO_{X, \perfd}}{X_n}) \cong \lim_{n \geq 1} R\Gamma(D_+(f), \mcalO_{X, \perfd} \otimes^L_{\mcalO_X} \mcalO_{X_n}) \\
        & \cong \lim_{n \geq 1} ((R_{\tgrpfd})_{(f)} \otimes^L_{R_{(f)}} R_{(f)}/(I_{(f)})^n)
    \end{align*}
    of \(R_{(f)}\)-algebras since \(X_n\) is the closed subscheme of \(X\) defined by \(I^n\) (not necessarily \(\Proj(R/I^n)\)).
    The systems \(\{(I_{(f)})^n\}_{n \geq 1}\) and \(\{(I^n[1/f])_0\}_{n \geq 0}\) of ideals of \(R_{(f)}\) are cofinal with each other by \Cref{ComparisonTopologyFormalProj}.
    Therefore, using the degree \(0\) weakly proregular sequence \(\underline{x} = x_1, \dots, x_r\) generating the ideal \(I\), the above isomorphisms continue as the following;
    \begin{align*}
        & \lim_{n \geq 1} ((R_{\tgrpfd})_{(f)} \otimes^L_{R_{(f)}} (R/I^n)_{(f)}) \cong \lim_{n \geq 1} ((R_{\tgrpfd})_{(f)} \otimes^L_{R_{(f)}} (R/^L \underline{x^n})_{(f)}) \\
        & \cong \lim_{n \geq 1} (((R_{\tgrpfd}/^L \underline{x^n})[1/f])_0) \cong (\grlim_{n \geq 1} ((R_{\tgrpfd}/^L \underline{x^n})[1/f]))_0 \cong (\dgrcomp{I}{R_{\tgrpfd}[1/f]})_0 \\
        & \cong ((R[1/f])_{\tgrpfd})_0 = \lim_{Q' \in \Perfd^{\wedge I}_{\graded{\setZ[1/p]}}(R[1/f])} Q'_0
    \end{align*}
    of \(R_{(f)}\)-algebras, where the second isomorphism holds for \(I\) generated by degree \(0\) elements and the fifth isomorphism is \Cref{CommutativityLimitsPerfd}.
    The last term appears in \eqref{PerfdProjLimitEq2} in the proof of \Cref{PerfdProjLimit}, and thus this is isomorphic to \(R\Gamma(D_+(f), \mcalO_{\mscrX, \perfd})\).
    Since they are functorial isomorphisms, we can conclude there exists an isomorphism
    \begin{equation*}
        \mcalO_{X, \perfd}^{\wedge} \cong \mcalO_{\mscrX, \perfd}
    \end{equation*}
    in \(\mcalD_{\qcoh}(\mscrX_{\Zar}, \mcalO_{\mscrX})\).
\end{proof}

\begin{corollary} \label{QuasiProjCase}
Let $B$ be a Noetherian ring and $X$ be a quasi-projective scheme over $\Spec(B)$.\footnote{Namely, \(X\) is an open subscheme of a projective scheme over \(\Spec(B)\).}
Let $I$ be an ideal of $B$ containing $p$.
Then the absolute perfectoidization of the structure sheaf $\cO_{\mscrX}$ of $I$-adic completion $\mscrX$ of $X$ is algebraizable.
\end{corollary}

\begin{proof}
    Take an open immersion \(X \hookrightarrow X'\) to a projective scheme \(X'\) over \(\Spec(B)\).
    The structure morphism \(X' \to \Spec(B)\) factors through \(X' \to \Spec(H^0(X', \mcalO_{X'}))\).
    Since \(X' \to \Spec(B)\) is projective and \(\Spec(H^0(X', \mcalO_{X'})) \to \Spec(B)\) is separated, \(X' \to \Spec(H^0(X', \mcalO_{X'}))\) is projective.
    Moreover, \(H^0(X', \mcalO_{X'})\) is Noetherian by \citeSta{02O5}, so \(X'\) satisfies the condition in \Cref{AlgebraizableOXperfd}.
\end{proof}

Following \Cref{CoherenceFormalSchemes}, we can show that \(\mcalO_{X, \perfd}\) is almost coherent when \(X\) comes from a Noetherian setting and in particular, the algebraic localization \(\mcalO_{X, \perfd}[1/p] \in \mcalD_{\qcoh}(X[1/p])\) is coherent:

\begin{proposition} \label{CoherenceAlgebraicLocalization}
    Keep the setting of \Cref{SettingFibersequence} and assume that \(R_0 = H^0(X, \mcalO_X)\) is a \(p\)-torsion-free perfectoid valuation ring.
    Assume further that there exists a flat projective scheme \(X'\) over a \(p\)-torsion-free discrete valuation ring \(R'_0 \defeq H^0(X', \mcalO_{X'})\) and an ample line bundle \(L'\) on \(X'\) such that \(R'_0 \to R_0\) is a flat extension, \(X \cong X' \times_{\Spec(H^0(X', \mcalO_{X'}))} \Spec(R_0)\) and \(L \cong L' \otimes_{\mcalO_{X'}} \mcalO_X\).
    
    Then \(\mcalO_{X, \perfd} \in \mcalD_{\qcoh}(X)\) is bounded and almost coherent with respect to the basic setup \((R_0, \sqrt{pR_0})\).
    In particular, \(\mcalO_{X, \perfd}[1/p] \in \mcalD_{\qcoh}(X[1/p])\) belongs to \(\mcalD_{\coh}^b(X[1/p])\), where \(X[1/p]\) is the open subscheme of \(X\) defined by the non-vanishing of \(p\) and \(\mcalO_{X, \perfd}[1/p]\) is the pullback of \(\mcalO_{X, \perfd}\) to \(X[1/p]\).
\end{proposition}

\begin{proof}
    Set \(R' = \bigoplus_{n \geq 0} H^0(X', L'^{\otimes n})\).
    Using the flat base change theorem, we have an isomorphism
    \begin{equation*}
        R' \otimes_{R'_0} R_0 \xrightarrow{\cong} \bigoplus_{n \geq 0} H^0(X, L^{\otimes n}) = R.
    \end{equation*}
    Since \(X'\) is projective over a discrete valuation ring \(R'_0\), the section ring \(R'\) is finitely generated over \(R'_0\) and thus \(R\) is finitely presented over \(R_0\).

    Using \cite{bhatt2025Aspects}*{Theorem 7.3.6} for \(\mfrakX = \Spf(R)\) as in \Cref{CoherenceFormalSchemes}, we know that the absolute perfectoidization \(R_{\perfd}\) is almost coherent as an \(R_0\)-module.
    By \Cref{CompletionOfAbsoluteGradedPerfd}, \(R_{\grpfd}/^L p \in \mcalD_{\graded{\setZ[1/p]}}(R)\) is almost coherent.
    The same arguments of \cite{bhatt2025Aspects}*{Proposition 7.1.6} in the graded setting, i.e., in \(\mcalD_{\graded{\setZ[1/p]}}^{\comp{p}}(R)\), show that \(R_{\grpfd} \in \mcalD_{\graded{\setZ[1/p]}}^{\comp{p}}(R)\) is almost coherent.
    Since \(R_{\perfd}\) is concentrated in degree \([0, \dim R]\) (\cite{bhatt2025Aspects}*{Proposition 4.7.4}), so is \(R_{\grpfd}\) in \(\mcalD_{\graded{\setZ[1/p]}}(R)\).
    In particular, \(R_{\grpfd} \in \mcalD_{\graded{\setZ[1/p]}}(R)\) is bounded and almost coherent.


    The construction of \(\mcalO_{X, \perfd}\) shows the isomorphism
    \begin{equation*}
        \restr{\mscrH^i(\mcalO_{X, \perfd})}{D_+(f)} \cong (H^i(R_{\grpfd})[1/f])_0
    \end{equation*}
    holds for any homogeneous element \(f \in R_+\) and \(i \in \setZ\).
    So the finiteness of \(R_{\grpfd}\) implies that these cohomology sheaves \(\mscrH^i(\mcalO_{X, \perfd})\) are bounded and almost coherent sheaves on \(X\).
    The above arguments deduce that \(\mcalO_{X, \perfd} \in \mcalD_{\qcoh}(X)\) is bounded and almost coherent with respect to the basic setup \((R_0, \sqrt{pR_0})\).
\end{proof}

    

\appendix

\section{Generalities on projective spectra}

\subsection{\texorpdfstring{Projective spectrum of \(\setQ\)-graded rings}{Projective spectrum of Q-graded rings}}

\begin{construction} \label{ConstProjMulti}
    Let \(n\) be a positive integer, and let
    \[
    R \defeq \bigoplus_{\mbfq = (q_1, \dots, q_n) \in \setQ^n} R_\mbfq
    \]
    be a \(\setQ^n\)-graded ring.
    For each standard basis vector \(\mbfe_i \defeq (0, \dots, 0, 1, 0, \dots, 0) \in \setQ^n\), we define a \(\setQ\)-graded ring \(R_i \defeq \bigoplus_{q \in \setQ} R_{q\mbfe_i}\), which is a graded subring of \(R\).
    Choose homogeneous ideals \(R_{i,+} \subseteq R_i\) such that \(R_{i,+} \cap R_{i,0} = 0\) for each \(i = 1, \dots, n\). This defines a homogeneous ideal \(R_+ \defeq R_{1,+} \cdots R_{n,+}\) of \(R\) such that \(R_+ \cap R_0 = 0\).
    We define a subset of \(\Spec(R)\) by
    \begin{align*}
        \Proj(R; R_+) & \defeq \Proj(R; R_{1, +}, \dots, R_{n, +}) \\
        & \defeq \set{\mfrakp \in \Spec(R)}{\text{$\mfrakp$ is homogeneous and $R_{i, +} \nsubseteq \mfrakp \cap R_i$ for all $i = 1, \dots, n$}}.
    \end{align*}
    For homogeneous elements \(f_i \in R_{i,+}\) for each \(i=1,\dots,n\), writing the product \(\mbff \defeq f_1 \cdots f_n\), we define a subset
    \begin{equation*}
        D_+(\mbff) \defeq D_+(f_1 \cdots f_n) \defeq \bigcap_{i=1}^n D_+(f_i) = \set{\mfrakp \in \Proj(R; R_+)}{f_1 \cdots f_n \notin \mfrakp} \subseteq \Spec(R).
    \end{equation*}
    The collection \(\set{D_+(f_1 \cdots f_n)}{f_1 \in R_{1, +}, \dots, f_n \in R_{n, +}}\) covers \(\Proj(R; R_+)\) and is stable under taking finite intersections.
    Thus we equip \(\Proj(R; R_+)\) with the Zariski topology for which this collection forms a basis.
    For each basic open subset \(D_+(\mbff)\), we define
    \begin{equation*}
        \mcalO_{\Proj(R;R_+)}(D_+(\mbff)) \defeq R_{(\mbff)} \defeq (R[1/\mbff])_0,
    \end{equation*}
    the degree-\(0\) part of the localization of \(R\) at \(f_1 \cdots f_n\).
    By the following lemma (\Cref{SectionMultigradedProj}), it follows that \(\Proj(R; R_+)\) is a separated scheme over \(\Spec(R_0)\), and we call this scheme the \emph{projective spectrum of \(R\) with respect to \(R_{1, +}, \dots, R_{n, +}\)}.
\end{construction}

\begin{lemma} \label{SectionsGradedMapMulti}
    Let \(R\) be a \(G\)-graded ring and let \(R_+\) be a homogeneous ideal of \(R\) such that \(R_+ \cap R_0 = 0\).
    For any homogeneous ideal \(I\) of \(R\) such that \(I \subseteq R_+\), the ideal \(\sqrt{I} \cap R_+\) is the intersection of \(R_+\) and all homogeneous prime ideals of \(R\) which contain \(I\) but not \(R_+\).
    In particular, in a \(\setQ^n\)-graded ring \(R\), any containment \(D_+(\mbfg) \subseteq D_+(\mbff)\) induces a canonical graded homomorphism \(R[1/\mbff] \to R[1/\mbfg]\).
    Thus we obtain a ringed space \((D_+(\mbff), \mcalO_{D_+(\mbff)})\) from these sections and the corresponding restriction maps.
\end{lemma}

\begin{proof}
    It suffices to show that if \(x \in R_+\) is a nonzero homogeneous element such that \(x \notin \sqrt{I}\), then there exists a homogeneous prime ideal \(\mfrakp\) of \(R\) such that \(R_+ \nsubseteq \mfrakp \supseteq I\) and \(x \notin \mfrakp\).
    Set \(\Sigma\) to be the set of all homogeneous ideals of \(R\) which do not meet the multiplicative subset \(\{x, x^2, \dots\}\) in \(R_+\).
    Since \(x \neq 0\), the set \(\Sigma\) is nonempty. Moreover, we can apply Zorn's lemma to \(\Sigma\) to obtain a maximal homogeneous ideal \(I_0\) in \(\Sigma\).
    We can show that \(I_0\) is a prime ideal: let \(f,g \in R\) be homogeneous elements such that both are not contained in \(I_0\).
    Then the homogeneous ideals \(I_0 + (f)\) and \(I_0 + (g)\) have an intersection with \(\{x, x^2, \dots\}\) in \(R_+\) because \(I_0\) is maximal in \(\Sigma\).
    There exist \(i,j \in I_0\) and \(a,b \in R\) such that \((i+fa)(j+gb) = x^k\) for some \(k \geq 1\). Since \(I_0\) does not have an intersection with \(\{x, x^2, \dots\}\), we have \(fg \notin I_0\), and hence \(I_0\) is prime.

    The last assertion follows from the fact that if \(D_+(g_1 \cdots g_n) \subseteq D_+(f_1 \cdots f_n)\), then our first assertion implies that \(g_1 \cdots g_n \in \sqrt{f_1 \cdots f_n R}\) and therefore we have a canonical map \(R[1/f_1 \cdots f_n] \to R[1/g_1 \cdots g_n]\) and this is a graded ring homomorphism.
\end{proof}

\begin{lemma} \label{SectionMultigradedProj}
    Let \(R\) be a \(\setQ^n\)-graded ring and let \(R_{i, +}\) be a homogeneous ideal of \(R_i\) such that \(R_{i, +} \cap R_{i, 0} = 0\) for each \(i = 1, \dots, n\).
    For any homogeneous element \(f_i \in R_{i, +}\), we set \(\mbff \defeq f_1 \cdots f_n\) and the map
    \begin{equation} \label{AffineProjMulti}
        D_+(\mbff) \to \Spec(R_{(\mbff)})\ ;\ \mfrakp \mapsto \mfrakp_{(\mbff)} \defeq \mfrakp[1/\mbff] \cap R_{(\mbff)}
    \end{equation}
    is a homeomorphism and induces an isomorphism of ringed spaces 
    \[
    (D_+(\mbff), \mcalO_{D_+(\mbff)}) \cong (\Spec(R_{(\mbff)}), \mcalO_{\Spec(R_{(\mbff)})}).
    \]
    In particular, for any standard affine open subsets \(D_+(\mbff)\) and \(D_+(\mbfg)\) of \(\Proj(R, R_+)\), the intersection \(D_+(\mbff) \cap D_+(\mbfg) = D_+(\mbff \mbfg)\) is also a standard affine open subset, and the restriction maps induce a surjective homomorphism \(R_{(\mbff)} \otimes R_{(\mbfg)} \twoheadrightarrow R_{(\mbff \mbfg)}\).
\end{lemma}

\begin{proof}
    For \(\mbff = f_1 \cdots f_n\), set \(\deg(f_i) = n_i/m_i\) in \(R_i\) for some \(n_i, m_i \in \setZ\) with \(n_i \geq 0\) and \(m_i \neq 0\). For simplicity, we may assume that all denominators \(m_i\) are equal, say \(m\), and we can write \(\deg(f_i) = n_i/m\) for each \(i = 1, \dots, n\).
    Take the product \(\widetilde{n} \defeq \prod_{i=1}^{n} n_i\).
    First we construct an inverse map.
    Let \(\mfrakq \in \Spec(R_{(\mbff)})\) be a prime ideal.
    Fix rational numbers \(r_i = k_i/l \in \setQ\) with \(k_i, l \in \setZ\), \(k_i \geq 0\), and \(l \neq 0\), and set \(k \defeq \prod_{i=1}^{n} k_i\). Throughout this proof, we will use this notation.
    We define an additive subgroup \(\mfrakq_{(r_1, \dots, r_n)}\) of \(R_{(r_1, \dots, r_n)}\) as follows:
    \begin{equation} \label{Degree0PrimeMapMulti}
        \widetilde{\mfrakq}_{(r_1, \dots, r_n)} \defeq \setlr{x \in R_{(r_1, \dots, r_n)}}{\frac{x^{l \widetilde{n}}}{\prod_{i=1}^{n}f_i^{m k_i \prod_{j \neq i} n_j}} \in \mfrakq \subseteq R_{(\mbff)}}
    \end{equation}
    which does not depend on the choices of \(k_i, l, m, n_i\).
    Then the direct sum \(\widetilde{\mfrakq} \defeq \bigoplus_{(r_1, \dots, r_n) \in \setQ^n} \widetilde{\mfrakq}_{(r_1, \dots, r_n)}\) is a graded prime ideal of \(R\) which contains none of the \(f_i\):
    For \(\deg(f_i) = n_i/m\), if \(f_i\) belongs to \(\widetilde{\mfrakq}\), then we have \(1 = f_i^{mn_i}/f_i^{n_im} \in \mfrakq\) but this is a contradiction.
    Let \(x,y \in R\) be homogeneous elements such that \(xy \in \widetilde{\mfrakq}\).
    Set \(\deg(x) = (k_1'/l', \dots, k'_n/l')\) and \(\deg(y) = (k''_1/l'', \dots, k''_n/l'')\) for some \(k'_i, k''_i \in \setZ\) and \(l', l'' \in \setZ\) with \(k'_i, k''_i \geq 0\) and \(l', l'' \neq 0\).
    Then \(\deg(x) + \deg(y) = ((k'_il'' + k''_il')/l'l'')_{i=1}^n\) and the containment \(xy \in \widetilde{\mfrakq}_{\deg(x) + \deg(y)}\) means that
    \begin{align*}
        R_{(\mbff)} \supset \mfrakq \ni \frac{(xy)^{l'l''\widetilde{n}}}{\prod_{i=1}^n f_i^{m(k'_i l'' + k''_i l')\prod_{j \neq i} n_j}} & = \frac{x^{l'l''\widetilde{n}} \cdot y^{l'l''\widetilde{n}}}{\prod_{i=1}^{n} \parenlr{f_i^{mk'_il''\prod_{j \neq i} n_j} \cdot f_i^{mk''_il'\prod_{j \neq i} n_j}}} \\
        & = \parenlr{\frac{x^{l\widetilde{n}}}{\prod_{i=1}^n f_i^{mk'_i\prod_{j \neq i} n_j}}}^{l''} \cdot \parenlr{\frac{y^{l'\widetilde{n}}}{\prod_{i=1}^{n}f_i^{mk''_i\prod_{j \neq i} n_j}}}^{l'}
    \end{align*}
    and so each factor belongs to \(R_{(\mbff)}\).
    This implies \(x \in \widetilde{\mfrakq}_{\deg(x)}\) or \(y \in \widetilde{\mfrakq}_{\deg(y)}\).
    Thus we obtain a continuous map \(\Spec(R_{(\mbff)}) \to D_+(\mbff)\), since this construction is compatible with inclusions of prime ideals.

    We show that this construction gives an inverse map of \eqref{AffineProjMulti}.
    Let \(\mfrakq\) be a prime ideal of \(R_{(\mbff)}\). We will show \(\mfrakq = (\widetilde{\mfrakq})_{(\mbff)}\).
    Let \(x/\mbff^N \in R_{(\mbff)}\) be any element such that \(x \in R_{(r_1, \dots, r_n)}\). Since this is a degree \(0\) element, we know \(k_i/l = Nn_i/m\).
    Assume that \(x/\mbff^N\) belongs to \(\mfrakq\), then we have
    \begin{equation*}
        \frac{x^{l\widetilde{n}}}{\prod_{i=1}^{n}f_i^{mk_i\prod_{j \neq i} n_j}} = \frac{x^{l\widetilde{n}}}{\prod_{i=1}^{n}f_i^{lN\widetilde{n}}} = \parenlr{\frac{x}{\mbff^N}}^{l\widetilde{n}} \in \mfrakq 
    \end{equation*}
    and therefore \(x \in \widetilde{\mfrakq}_{(r_1, \dots, r_n)}\) by \eqref{Degree0PrimeMapMulti} which shows \(x/\mbff^N \in (\widetilde{\mfrakq})_{(\mbff)}\).
    Conversely, the condition \(x \in \widetilde{\mfrakq}_{(r_1, \dots, r_n)}\) means that \(x^{l\widetilde{n}}/\prod_{i=1}^{n} f_i^{mk_i\prod_{j \neq i} n_j} \in \mfrakq\). So we have
    \begin{equation*}
        \mfrakq \ni \frac{x^{l\widetilde{n}}}{\prod_{i=1}^n f_i^{mk_i\prod_{j \neq i} n_j}} = \frac{x^{l\widetilde{n}}}{\prod_{i=1}^n f_i^{Nl\widetilde{n}}} = \parenlr{\frac{x}{\mbff^N}}^{l\widetilde{n}}
    \end{equation*}
    and therefore \(x/\mbff^N \in \mfrakq\). In conclusion, we can show that the map \eqref{AffineProjMulti} is a homeomorphism.

    More explicitly, we can show that the map \(D_+(\mbff) \to \Spec(R_{(\mbff)})\), which is a restriction of the continuous map \(\Spec(R[1/\mbff]) \to \Spec(R_{(\mbff)})\), is open.
    Take any standard open subset \(D_+(\mbfg) \subseteq D_+(\mbff)\) for some element \(\mbfg \defeq g_1 \cdots g_n \in R\), where \(g_i \in R_{+, i}\) and \(\deg(g_i) = n'_i/m'\) for some \(n'_i, m' \in \setZ\) with \(n'_i \geq 0\) and \(m' \neq 0\).
    For any prime ideal \(\mfrakp \in D_+(\mbff)\), it is contained in \(D_+(\mbfg)\) if and only if \(\mbfg \notin \mfrakp\) if and only if \(g_i^{M_i}/f_i^{N_i} \notin \mfrakp[1/\mbff]\) for all \(i = 1, \dots, n\) and for any, (or equivalently, some) \(M_i, N_i \in \setZ\) with \(M_i \geq 0\) and \(N_i \neq 0\).
    Set
    \[
    \mbfg_{(\mbff)} \defeq \prod_{i=1}^{n} (g_i^{m'n_i}/f_i^{mn'_i}) \in R_{(\mbff)},
    \]
    which is a degree-\(0\) element.
    Then the previous argument shows that \(\mfrakp \in D_+(\mbfg)\) is equivalent to \(\mbfg_{(\mbff)} \notin \mfrakp_{(\mbff)}\) and thus \(D_+(\mbfg)\) bijectively corresponds to the standard open subset \(D(\mbfg_{(\mbff)})\) in \(\Spec(R_{(\mbff)})\).

    Using this correspondence, the containment \(D_+(\mbfg) \subseteq D_+(\mbff)\) corresponds to the canonical morphism \(R_{(\mbff)} \to R_{(\mbff)}[1/\mbfg_{(\mbff)}]\) on \(\Spec(R_{(\mbff)})\). By \Cref{SectionsGradedMapMulti}, we have an identification \(R[1/\mbff \mbfg] = R[1/\mbfg]\). In particular, the graded map \(R[1/\mbff \mbfg_{(\mbff)}] \hookrightarrow R[1/\mbff \mbfg] = R[1/\mbfg]\) induces a map \(R_{(\mbff)}[1/\mbfg_{(\mbff)}] \hookrightarrow R_{(\mbfg)}\) since \(\mbfg_{(\mbff)}\) is a degree \(0\) element.
    For any element \(y/\mbfg^N \in R_{(\mbfg)}\), we can write
    \begin{equation*}
        \frac{y}{\mbfg^N} = \frac{\prod_{i=1}^{n}f_i^{Nmn'_i}}{\prod_{i=1}^{n} g_i^{Nm'n_i}} \cdot \frac{y\prod_{i=1}^{n} g_i^{N(m'n_i-1)}}{\prod_{i=1}^{n}f_i^{mn'_iN}} = \frac{1}{\mbfg_{(\mbff)}^N} \cdot \frac{y\prod_{i=1}^{n} g_i^{N(m'n_i-1)}}{\prod_{i=1}^{n}f_i^{mn'_iN}} \in R_{(\mbff)}[1/\mbfg_{(\mbff)}]
    \end{equation*}
    and this shows that the map \(R_{(\mbff)}[1/\mbfg_{(\mbff)}] \hookrightarrow R_{(\mbfg)}\) is surjective.
    Therefore, the homeomorphism \eqref{AffineProjMulti} is compatible with the sections and then an isomorphism of ringed spaces.

    In the last assertion, we have to check the surjectivity of the map \(R_{(\mbff)} \otimes R_{(\mbfg)} \to R_{(\mbff \mbfg)}\): Take any element \(x/\mbff^N\mbfg^N\) in \(R_{(\mbff \mbfg)}\). Then we have
    \begin{equation*}
        \frac{x}{\mbff^N\mbfg^N} = \frac{x}{\mbff^N} \cdot \frac{\prod_{i=1}^{n}g_i^{(m'n_i-1)N}}{\prod_{i=1}^{n}f_i^{mn'_iN}} \cdot \parenlr{\frac{\prod_{i=1}^{n}f_i^{mn'_i}}{\prod_{i=1}^{n} g_i^{m'n_i}}}^N
    \end{equation*}
    which is in the image of the map \(R_{(\mbff)} \otimes R_{(\mbfg)} \to R_{(\mbff \mbfg)}\). This completes the proof.
\end{proof}

\subsection{\texorpdfstring{Projective spectrum of pro-\(G\)-graded rings}{Projective spectrum of pro-G-graded rings}}

Next, we will define the projective spectrum of pro-\(\setQ^n\)-graded rings. Before that we record the following lemma.

\begin{lemma} \label{GradedWiseCompletionLocalization}
    Let \(R\) be a \(G\)-graded topological ring and let \(f\) be a homogeneous element of \(R\).
    Then the canonical morphism
    \begin{equation*}
        (R[1/f])^{\wedge_{\graded}} \to (R^{\wedge_{\graded}}[1/f])^{\wedge_{\graded}}
    \end{equation*}
    is an isomorphism of \(G\)-graded topological rings.
\end{lemma}

\begin{proof}
    This follows from the universal property of gradedwise completion and localization at \(f\).
\end{proof}

\begin{definition} \label{ProjConstProGraded}
    Let \((R, R_{\graded})\) be a pro-\(\setQ^n\)-graded ring such that \(R_{\graded}\) is a \(\setQ^n\)-graded adic ring with a homogeneous ideal of definition \(J\).
    As in \Cref{ConstProjMulti}, after fixing homogeneous ideals \(R_{i, +}\) of \(R_i \defeq \oplus_{q \in \setQ} R_{q \mbfe_i} \subseteq R_{\graded}\) such that \(R_{i, +} \cap R_{i, 0} = 0\) for each \(i = 1, \dots, n\), we define a homogeneous ideal \(R_{\graded, +} \defeq R_{1, +} \cdots R_{n, +}\) of \(R_{\graded}\) such that \(R_{\graded, +} \cap R_0 = 0\). Throughout this subsection, we use this notation.

    The \emph{projective spectrum of \((R, R_{\graded})\) with respect to \(R_{1, +}, \dots, R_{n, +}\)} is defined as the formal completion
    \begin{equation*}
        \Proj(R; R_{\graded, +}) \defeq \hProj(R_{\graded}; R_{\graded, +})
    \end{equation*}
    of the projective spectrum \(\Proj(R_{\graded}; R_{\graded, +})\) of \(R_{\graded}\) along the closed subscheme \(V_+(J)\).
    More explicitly, since the ideal sheaf \(\widetilde{J}\) corresponding to \(V_+(J)\) is defined by \(D_+(\mbff) \mapsto J_{(\mbff)}\), the underlying set is
    \begin{equation*}
        \abs{\Proj(R; R_{\graded, +})} = \set{\mfrakp \in V(J) \subseteq \Spec(R_{\graded})}{\text{$\mfrakp$ is homogeneous and $R_{i, +} \nsubseteq \mfrakp \cap R_i$ for all $i = 1, \dots, n$}}
    \end{equation*}
    with an open cover \(\set{D_+(\mbff)^{\wedge}}{f_1 \in R_{1,+}, \dots, f_n \in R_{n,+}}\) where \(D_+(\mbff)^{\wedge} \defeq \set{\mfrakp \in \Proj(R; R_{\graded, +})}{\mbff \notin \mfrakp}\) and the structure sheaf is defined by
    \begin{equation*}
        \mcalO_{\Proj(R; R_{\graded, +})}(D_+(\mbff)^{\wedge}) = \lim_{n \geq 1} R_{(\mbff)}/(J_{(\mbff)})^nR_{(\mbff)} \eqdef R_{(\mbff)}^{\wedge},
    \end{equation*}
    which is the degree \(0\) part of the gradedwise \(J\)-adic completion \((R_{\graded}[1/\mbff])^{\wedge_{\graded}}\) as follows (\Cref{ComparisonTopologyFormalProj}).
\end{definition}

\begin{lemma} \label{ComparisonTopologyFormalProj}
    Keep the notation in \Cref{ProjConstProGraded}.
    For any \(\mbff \in R_{\graded, +}\), the \(J_{(\mbff)}\)-adic topology on \(R_{(\mbff)}\) coincides with the relative topology induced by the \(J\)-adic topology on \(R[1/\mbff]\).
\end{lemma}

\begin{proof}
    For notational simplicity, we do not distinguish ideals of \(R\) and their extensions in \(R[1/\mbff]\).
    Note that the relative topology on \(R_{(\mbff)}\) is defined by the collection \(\{(J^n)_0\}_{n \geq 1}\).
    The containment \((J_0)^k \subseteq (J^k)_0\) is clear, so it remains to show the converse.
    Since \(R_{(\mbff)}\) does not change if we replace \(\mbff\) by a power of \(\mbff\), we may assume that \(\mbff\) is a product of homogeneous elements \(f_i \in R_{i, +}\) such that \(\deg(f_i) = d_i\mbfe_i\) for some \(d_i \in \setZ \setminus \{0\}\), for every \(i = 1, \dots, n\).

    Set \(d \defeq \prod_{i=1}^n \abs{d_i}\) and \(m(d) \defeq d^n(d-1)^n + 1\).
    We can show
    \begin{equation*}
        (J^{km(d)})_0 \subseteq (J_0)^k
    \end{equation*}
    for any \(k \geq 1\) as follows:
    Take homogeneous elements \(h_j \in J\) for \(j = 1, \dots, km(d)\) such that \(h_1 \cdots h_{km(d)}\) is a degree \(0\) element.
    We can write \(h_j = x_j/\mbff^N\) for some \(x_j \in J\) and \(N \geq 0\) and
    \begin{equation*}
        \deg(x_j) = (n_{j, 1}d_1 + r_{j, 1}, \dots, n_{j, n}d_n + r_{j, n})
    \end{equation*}
    for some \(n_{j, i} \in \setZ\) and \(r_{j, i} \in \setZ\) with \(0 \leq r_{j, i} < d_i\) for each \(j = 1, \dots, km(d)\) and each \(i = 1, \dots, n\).
    Since
    \begin{equation*}
        \sum_{j=1}^{km(d)} \deg(h_j) = \sum_{j=1}^{km(d)} (\deg(x_j) - N \deg(\mbff)) = \sum_{j=1}^{km(d)} ((n_{j, 1} - N)d_1 + r_{j, 1}, \dots, (n_{j, n} - N)d_n + r_{j, n}) = (0, \dots, 0),
    \end{equation*}
    we have
    \begin{equation*}
        \sum_{j=1}^{km(d)} \overline{r_{j, i}} = \sum_{a=0}^{k-1} \sum_{j=1}^{m(d)} \overline{r_{am(d) + j, i}} = 0 \in \setZ/d_i\setZ
    \end{equation*}
    for each \(i = 1, \dots, n\).

    Fix an integer \(a\) with \(0 \leq a \leq k-1\).
    Using \(m(d) \geq d \cdot d^{n-1}(d-1)^n + 1\), we can find a subset \(I_{a, 1}\) of \(\{1, \dots, m(d)\}\) such that \(\abs{I_{a, 1}} = d^{n-1}(d-1)^n + 1\) and all \(r_{j, 1}\) are the same for \(j \in I_{a, 1}\).
    Next, using \(\abs{I_{a, 1}} = d^{n-1}(d-1)^n + 1 \geq d \cdot d^{n-2}(d-1)^{n-1} + 1\), we can find a subset \(I_{a, 2}\) of \(I_{a, 1}\) such that \(\abs{I_{a, 2}} = d^{n-2}(d-1)^{n-1} + 1\) and all \(r_{j, 2}\) are the same for \(j \in I_{a, 2}\).
    Repeating this process, we can find a sequence of subsets
    \begin{equation*}
        \{1, \dots, m(d)\} \supseteq I_{a, 1} \supseteq I_{a, 2} \supseteq \cdots \supseteq I_{a, n-1} \supseteq I_{a, n} \eqdef I_a
    \end{equation*}
    such that \(\abs{I_{a, i}} = d^{n-i}(d-1)^{n-i+1} + 1\) and all \(r_{j, i}\) are the same for \(j \in I_{a, i}\) for each \(i = 1, \dots, n\).
    Therefore, for each \(a = 0, \dots, k-1\) and each \(i = 1, \dots, n\), we have
    \begin{equation*}
        \sum_{j' \in I_a} r_{am(d) + j', i} = N_{a, i}d_i
    \end{equation*}
    for some \(N_{a, i} \in \setZ\) since \(d_i\) divides \(d\) and all \(r_{j, i}\) are the same for \(j \in I_{a, n}\) and \(\abs{I_a} = d\).

    Define
    \begin{equation} \label{ElementInJ0}
        N_a \defeq \prod_{j' \in I_a} x_{am(d) + j'} \in J \subseteq R_{\graded} \quad \text{and} \quad
        D_a \defeq \prod_{i=1}^n f_i^{\sum_{j' \in I_a} n_{j', i} + N_{a, i}} \in R_{\graded, +}
    \end{equation}
    The degree of the \(i\)-th component of \(D_a\) is
    \begin{align*}
        & d_i \cdot \parenlr{(\sum_{j' \in I_a} n_{am(d) + j', i}) + N_{a, i}} = (\sum_{j' \in I_a} n_{am(d) + j', i}d_i) + N_{a, i}d_i \\
        & = \sum_{j' \in I_a} (n_{am(d) + j', i} d_i + r_{am(d) + j', i}) = \sum_{j' \in I_a} \deg(x_{am(d) + j'})_i,
    \end{align*}
    where \(\deg(-)_i\) denotes the \(i\)-th component of the degree, which is equal to the degree of the \(i\)-th component of \(N_a\) in \eqref{ElementInJ0} and therefore we can take an element
    \begin{equation*}
        F \defeq \prod_{a = 0}^{k-1} \frac{N_a}{D_a} \in (J_0)^k \subseteq R_{(\mbff)}.
    \end{equation*}
    Moreover, the degree of the \(i\)-th component of the denominator \(\prod_{a=0}^{k-1} D_a\) of \(F\) is less than or equal to
    \begin{equation*}
        \sum_{a=0}^{k-1} \sum_{j = 1}^{m(d)} \deg(x_{am(d) + j})_i = \sum_{j=1}^{km(d)} \deg(x_j)_i = d_iNkm(d) = \deg(\mbff^{Nkm(d)}),
    \end{equation*}
    where the last equality follows from the fact that \(\deg(h_1 \cdots h_{km(d)}) = \deg(x_1 \cdots x_{km(d)}) - \deg(\mbff^{Nkm(d)}) = (0, \dots, 0)\).
    Consequently, \(F\) divides \(h_1 \cdots h_{km(d)}\) in \(R[1/\mbff]\) and therefore in \(R_{(\mbff)}\) which shows that \(h_1 \cdots h_{km(d)}\) belongs to \((J_0)^k\) in \(R_{(\mbff)}\) and therefore \((J^{km(d)})_0 \subseteq (J_0)^k\).
\end{proof}

\begin{lemma} \label{ProjProGraded}
    The projective spectrum \(\Proj(R; R_{\graded, +})\) is a formal scheme which is separated over \(\Spf(R_0)\).
    Moreover, for any \(\mbff \in R_{\graded, +}\), the open subscheme \(D_+(\mbff)^{\wedge}\) is isomorphic to the affine formal spectrum \(\Spf(R_{(\mbff)}^{\wedge}) \cong \Spf((\grcomp{J}{R_{\graded}[1/\mbff]})_0)\).
\end{lemma}

\begin{proof}
    Since \(\Proj(R_{\graded}; R_{\graded, +})\) is a separated scheme over \(\Spec(R_0)\) by \Cref{ConstProjMulti}, its formal completion \(\widehat{\Proj(R_{\graded}; R_{\graded, +})}\) is separated over \(\Spf(R_0)\) as well, where \(R_0\) admits the restriction of the \(J\)-adic topology on \(R_{\graded}\).
    Since \(D_+(\mbff)^{\wedge}\) is the completion of the open subscheme \(D_+(\mbff)\) of \(\Proj(R_{\graded}; R_{\graded, +})\) along \(V_+(J)\), it is isomorphic to the \(J_{(\mbff)}\)-adic completion of \(\Spec(R_{(\mbff)})\) by \Cref{SectionMultigradedProj}, which is written as \(\Spf(R_{(\mbff)}^{\wedge})\).
    By \Cref{ComparisonTopologyFormalProj}, this is isomorphic to the formal spectrum of \((\grcomp{J}{R_{\graded}[1/\mbff]})_0\).
\end{proof}

\begin{construction} \label{AssociatedSheafFormalGraded}
    Let \((R, R_{\graded})\) be a pro-\(\setQ^n\)-graded ring with a homogeneous ideal of definition \(J\) and let \(M\) be a \(\setQ^n\)-graded \(R_{\graded}\)-module.
    Then we can define an adic quasi-coherent sheaf \(M^{\Delta}\) on \(\Proj(R; R_{\graded, +})\) by
    \begin{equation*}
        M^{\Delta} \defeq \widetilde{M}^{\wedge}; \quad D_+(f) \mapsto \lim_{n \geq 1} M_{(f)}/(J_{(f)})^nM_{(f)} \eqdef M_{(f)}^{\wedge},
    \end{equation*}
    for any \(f \in R_{\graded, +}\), i.e., the completion of the quasi-coherent sheaf \(\widetilde{M}\) on \(\Proj(R_{\graded}; R_{\graded, +})\) associated to \(M\) along \(V_+(J)\).
    Using \Cref{ComparisonTopologyFormalProj}, the isomorphism
    \begin{equation*}
        M^{\Delta} \cong \lim_{n \geq 1} \widetilde{M/J^nM}
    \end{equation*}
    holds.
\end{construction}

\section{Perfectoid and prismatic schemes}

\subsection{Perfectoid and prismatic formal schemes}




\begin{definition} \label{DefPrismaticScheme}
    A \emph{pseudo-prismatic scheme} (resp. a \emph{prismatic scheme}) is a pair \((\mscrA, \mcalI)\) of an (adic) formal scheme \(\mscrA\) with a \(\delta\)-structure and an ideal sheaf \(\mcalI\) of \(\mscrA\) such that the following conditions hold:
    \begin{enumerate}
        \item The ideal sheaf \((p, \mcalI)\) generated by \(p\) and \(\mcalI\) is topologically nilpotent on \(\mscrA\),
        \item \(\mcalI\) is an invertible \(\mcalO_{\mscrA}\)-module,
        \item there exists an affine open covering \(\mscrA = \cup_{i \in I} \mscrU_i\) such that \(p\) is contained in the ideal generated by \(\mcalI(\mscrU_i)\) and \(\varphi(\mcalI(\mscrU_i))\) in \(\mcalO_{\mscrA}(\mscrU_i)\) (resp., \(p \in \mcalI(\mscrA) + \varphi(\mcalI(\mscrA))\mcalO_{\mscrA}(\mscrA)\)), and
        \item Zariski locally, \(\mcalO_{\mscrA}/\mcalI\) has bounded \(p^{\infty}\)-torsion.
    \end{enumerate}
    If \(\mscrA\) is perfect as a \(\delta\)-scheme, we say that \((\mscrA, \mcalI)\) is a \emph{(pseudo) perfect prismatic scheme} respectively.
    We often denote the closed subscheme \(V(\mcalI)\) in \(\mscrA\) by \(\overline{\mscrA}\), which is a bounded \(p\)-adic formal scheme.

    A morphism of (pseudo) prismatic schemes \(f \colon (\mscrA, \mcalI) \to (\mscrB, \mcalJ)\) is a morphism of formal schemes \(\mscrA \to \mscrB\) which is compatible with the given \(\delta\)-structures and such that the map \(\mcalO_{\mscrB} \to f_*\mcalO_{\mscrA}\) sends \(\mcalJ\) to \(f_*\mcalI\).
\end{definition}

\begin{construction} \label{DecompPerfdFormalScheme}
    We define an ideal sheaf \(\mcalO_{\mscrX}[p^\infty]\) by the formula \(\mscrU \mapsto H^0(\mscrU, \mcalO_{\mscrX})[p^\infty]\). Since \(\mscrX\) can be covered by affine perfectoid formal schemes, this ideal sheaf is equal to the kernel \(\mcalO_{\mscrX}[p]\) of the map \(\mcalO_{\mscrX} \xrightarrow{\times p} \mcalO_{\mscrX}\).
    Taking the quotient, we obtain a \(p\)-torsion-free \(\mcalO_{\mscrX}\)-algebra sheaf \(\mcalO_{\mscrX,tf} \defeq \mcalO_{\mscrX}/\mcalO_{\mscrX}[p^\infty]\).
    On the other hand, taking induced reduced structure on \(\mscrX_{p=0}\), we have a perfect \(\mcalO_{\mscrX}\)-algebra sheaf \(\mcalO_{\mscrX_{p=0}, \red}\).
    Similarly, we can define \(\mcalO_{(\mscrX_{tf})_{p=0}, \red}\).
    Using the decomposition of perfectoid rings (e.g., \cite{bhatt2019Regular}*{Remark 3.9}), we have an exact sequence of \(\mcalO_{\mscrX}\)-modules
    \begin{equation*}
        0 \to \mcalO_{\mscrX} \to \mcalO_{\mscrX_{tf}} \times \mcalO_{\mscrX_{p=0}, \red} \to \mcalO_{(\mscrX_{tf})_{p=0}, \red} \to 0.
    \end{equation*}
\end{construction}

\begin{lemma} \label{PseudoPerfectoid}
    Let \(\mscrX\) be a perfectoid formal scheme with a morphism to an affine perfectoid formal scheme \(\Spf(P)\).
    Take an element \(\varpi \in P\) such that \(p = \varpi^p u\) for some unit \(u \in P^{\times}\).
    The global section \(R \defeq H^0(\mscrX, \mcalO_{\mscrX})\) satisfies the following properties:
    \begin{enumerate}
        \item the \(p\)-power map \(R/\varpi \xrightarrow{a \mapsto a^p} R/p\) is injective,
        \item \(R\) has bounded \(p^\infty\)-torsion, and
        \item the \(p\)-power map \(R[p^\infty] \xrightarrow{a \mapsto a^p} R[p^\infty]\) is bijective.
    \end{enumerate}
\end{lemma}

\begin{proof}
    (1): Using the decomposition (\Cref{DecompPerfdFormalScheme}), we have an exact sequence
    \begin{center}
        \begin{tikzcd}
            0 \arrow[r] & H^0(\mcalO_{\mscrX}) \arrow[r] \arrow[d, "\times p"] & {H^0(\mcalO_{\mscrX_{tf}}) \times H^0(\mcalO_{\mscrX_{p=0}, \red})} \arrow[r, "\delta"] \arrow[d, "\times p"] & \Image(\delta) \arrow[r] \arrow[d, "\times p"] & 0 \\
            0 \arrow[r] & H^0(\mcalO_{\mscrX}) \arrow[r]                       & {H^0(\mcalO_{\mscrX_{tf}}) \times H^0(\mcalO_{\mscrX_{p=0}, \red})} \arrow[r, "\delta"]                       & \Image(\delta) \arrow[r]                       & 0
        \end{tikzcd}
    \end{center}
    where \(\delta\) is the map \(H^0(\mcalO_{\mscrX_{tf}}) \times H^0(\mcalO_{\mscrX_{p=0}, \red}) \to H^0(\mcalO_{(\mscrX_{tf})_{p=0}, \red})\).
    The snake lemma yields an exact sequence
    \begin{equation*}
        H^0(\mcalO_{\mscrX_{p=0}, \red}) \xrightarrow{\delta} \Image(\delta) \to H^0(\mcalO_{\mscrX})/p \to H^0(\mcalO_{\mscrX_{tf}})/p \times H^0(\mcalO_{\mscrX_{p=0}, \red})
    \end{equation*}
    since \(H^0(\mcalO_{\mscrX_{p=0}, \red})\) and \(\Image(\delta) \subseteq H^0(\mcalO_{(\mscrX_{tf})_{p=0}, \red})\) are \(p\)-torsion.
    The same argument works for \(\varpi\) instead of \(p\). So we have a commutative diagram
    \begin{center}
        \begin{tikzcd}
            {H^0(\mcalO_{\mscrX_{p=0}, \red})} \arrow[d, "F"] \arrow[d, "\cong"'] \arrow[r] & \Image(\delta) \arrow[d, "F", hook] \arrow[r] & H^0(\mcalO_{\mscrX})/\varpi \arrow[d, "F"] \arrow[r] & {H^0(\mcalO_{\mscrX_{tf}})/\varpi \times H^0(\mcalO_{\mscrX_{p=0}, \red})} \arrow[d, "F"] \\
            {H^0(\mcalO_{\mscrX_{p=0}, \red})} \arrow[r]                                    & \Image(\delta) \arrow[r]                      & H^0(\mcalO_{\mscrX})/p \arrow[r]                     & {H^0(\mcalO_{\mscrX_{tf}})/p \times H^0(\mcalO_{\mscrX_{p=0}, \red})}                    
        \end{tikzcd}
    \end{center}
    whose rows are exact. Since the \(p\)-th power map on perfect \(\mcalO_{\mscrX}\)-algebras is an isomorphism, it suffices to show that the map \(H^0(\mcalO_{\mscrX_{tf}})/\varpi \xrightarrow{a \mapsto a^p} H^0(\mcalO_{\mscrX_{tf}})/p\) is injective.
    Since this \(\mcalO_{\mscrX_{tf}}\) is \(p\)-torsion-free, we have an exact sequence \(0 \to \mcalO_{\mscrX_{tf}} \xrightarrow{\times p} \mcalO_{\mscrX_{tf}} \to \mcalO_{\mscrX_{tf}}/p \to 0\).
    Taking global sections, we have an inclusion \(H^0(\mcalO_{\mscrX_{tf}})/p \subseteq H^0(\mcalO_{\mscrX_{tf}}/p)\).
    Since the \(p\)-power map \(\mcalO_{\mscrX_{tf}}/\varpi \xrightarrow{a \mapsto a^p} \mcalO_{\mscrX_{tf}}/p\) is an isomorphism, the induced map \(H^0(\mcalO_{\mscrX_{tf}}/\varpi) \xrightarrow{a \mapsto a^p} H^0(\mcalO_{\mscrX_{tf}}/p)\) is injective.
    This is compatible with the inclusion and hence we have the injectivity \(H^0(\mcalO_{\mscrX_{tf}})/\varpi \xrightarrow{a \mapsto a^p} H^0(\mcalO_{\mscrX_{tf}})/p\).

    (2): Since perfectoid rings are reduced, \(\mscrX\) is reduced, and hence so is its ring of global sections \(R\). If \(x \in R\) is \(p^N\)-torsion, then \(\varpi^{N/p^n}x\) is zero for all \(n \geq 0\). This shows that \(R\) has bounded \(p^\infty\)-torsion.

    (3): Take an affine open covering \(\mscrX = \cup_{i \in I} \mscrU_i\) such that each \(R_i \defeq H^0(\mscrU_i, \mcalO_{\mscrX})\) is a perfectoid ring. We have an exact sequence of topological rings
    \begin{equation*}
        0 \to R \to \prod_{i \in I} R_i \to \prod_{i, j \in I} R_{ij},
    \end{equation*}
    where \(R_{ij} \defeq H^0(\mscrU_i \cap \mscrU_j, \mcalO_{\mscrX})\).
    Since \(R_i\) is a perfectoid ring, \(R_i[p^\infty] = R_i[p]\) and the \(p\)-power map \(R_i[p] \xrightarrow{a \mapsto a^p} R_i[p]\) is bijective for every \(i \in I\).
    If \(x \in R[p^\infty]\) satisfies \(x^p = 0\), then the same holds for its restriction \(x_i \defeq \restr{x}{\mscrU_i}\) for every \(i \in I\). It follows that \(x_i\) is zero in \(R_i[p^\infty]\) and hence \(x = 0\), which shows the injectivity.
    For any \(x \in R[p^\infty]\), there exists \(y_i \in R_i[p]\) such that \(x_i = y_i^p\) for every \(i \in I\).
    Then \(y_i^p - y_j^p = 0\) holds in \(R_{ij}\) and in particular \((y_i - y_j)^p = 0\) in \(R_{ij}\) since \(y_i\) and \(y_j\) are \(p\)-torsion.
    Using again that \(\mscrX\) is reduced, we have \(y_i - y_j = 0\) in \(R_{ij}\) and hence there exists a unique \(y \in R\) such that \(y_i = \restr{y}{\mscrU_i}\) for every \(i \in I\).
    Since \(y^p\) is zero after restricting to every \(R_i\), we have \(y \in R[p^\infty]\) and this shows the surjectivity.
\end{proof}

\begin{lemma} \label{LocallyPrismaticScheme}
    Let \((\mscrA, \mcalI)\) be a pseudo-prismatic scheme. Then the following conditions hold:
    \begin{enumerate}
        \item There exists an affine open covering \(\mscrA = \cup_{i \in I} \mscrU_i\) such that for each \(i \in I\), the pair \((\mcalO_{\mscrA}(\mscrU_i), \mcalI(\mscrU_i))\) is a bounded prism.
        \item If \((\mscrA, \mcalI)\) is perfect, the pair \((A, \mcalI(\mscrU))\) is a perfect prism for any affine formal open subset \(\mscrU = \Spf(A)\) of \(\mscrA\) and the closed subscheme \(\overline{\mscrA}\) is a perfectoid formal scheme.
        \item Any map of pseudo-prismatic schemes \((\mscrA, \mcalI) \to (\mscrB, \mcalJ)\) satisfies rigidity, that is, the induced map \(\mcalJ \to f_*\mcalI\) is surjective.
        \item If there is a map \((\mscrA, \mcalI) \to (\Spf(A), I)\) of pseudo-prismatic schemes, then \((\mscrA, \mcalI)\) is a prismatic scheme, i.e., \(p \in \mcalI(\mscrA) + \varphi(\mcalI(\mscrA))\mcalO_{\mscrA}(\mscrA)\).
        \item If \((\mscrA, \mcalI)\) is a perfect prismatic scheme, then the pair of global sections \((\mcalO_{\mscrA}(\mscrA), \mcalI(\mscrA))\) is a perfect prism. In particular, \(\mcalI\) is isomorphic to \(\mcalO_{\mscrA}\).
    \end{enumerate}
\end{lemma}

\begin{proof}
    (1): Let \(\mscrU = \Spf(A)\) be an affine open subset, and set \(I \defeq \mcalI(\mscrU)\). Note that the complete topological ring \(A\) is complete with respect to the \((p, I)\)-adic topology and any completed localization admits a unique \(\delta\)-structure (\cite{bhatt2022Prismsa}*{Remark 2.16}). Moreover, \(A/I\) has bounded \(p^{\infty}\)-torsion by the quasi-compactness of \(\mscrU=\Spf(A)\) and the local boundedness of \(p^{\infty}\)-torsion on Zariski open subsets.

    (2): By (1), it suffices to show \(p \in I + \varphi(I)A\) under the assumption that \(\mscrA = \Spf(A)\) is affine. This follows from the same argument as in \cite{takaya2026Relative}*{Proposition 4.22}.

    (3): This follows from (1) and the rigidity of morphisms of prisms (\cite{bhatt2022Prismsa}*{Lemma 3.5}).


    (4): Since we have a map of rings \(A \to \mcalO_{\mscrA}(\mscrA)\) and this sends \(I\) to \(\mcalI(\mscrA)\), the prismatic condition on \((A, I)\) implies the desired condition on \((\mscrA, \mcalI)\).

    (5): The prismatic condition on \((\mcalO_{\mscrA}(\mscrA), \mcalI(\mscrA))\) follows from the fact that \((\mscrA, \mcalI)\) is a prismatic scheme. Since it is a perfect prism, \(\mcalI(\mscrA)\) is generated by a non-zero-divisor \(d\) in \(\mcalO_{\mscrA}(\mscrA)\). This shows that \(\mcalI\) is generated by \(d\) and thus \(\mcalI\) is isomorphic to \(\mcalO_{\mscrA}\) as an \(\mcalO_{\mscrA}\)-module.
\end{proof}

\subsection{Tilting and untilting of formal schemes}

\begin{lemma} \label{TiltSheaf}
    Let \(\mscrX\) be a formal scheme such that \(p\) is topologically nilpotent.
    On the topological space \(\abs{\mscrX}\), we define a presheaf of topological rings of characteristic \(p\) by
    \begin{equation}
        \mcalO_{\mscrX}^\flat \colon \mscrU \mapsto \mcalO_{\mscrX}(\mscrU)^{\flat} \defeq \lim_{F} (\mcalO_{\mscrX}(\mscrU)/(p))
    \end{equation}
    where \(\mcalO_{\mscrX}(\mscrU)^{\flat}\) is endowed with the limit topology.
    Then this defines a sheaf of topological rings on \(\abs{\mscrX}\).
\end{lemma}

\begin{proof}
    Note that we have a canonical isomorphism
    \begin{equation} \label{TiltDescription}
        \lim_{a \mapsto a^p} \mcalO_{\mscrX}(\mscrU) \xrightarrow{\cong} \mcalO_{\mscrX}(\mscrU)^{\flat}; \quad a = (a^{(n)}) \mapsto (\overline{a^{(n)}})
    \end{equation}
    of multiplicative monoids by \cite{bhatt2018Integral}*{Lemma 3.2 (i)} since \(\mcalO_{\mscrX}(\mscrU)\) is \(p\)-adically complete.
    Take any open covering \(\{\mscrU_i \hookrightarrow \mscrU\}_{i \in I}\).
    Assume that an element \(a \in \mcalO_{\mscrX}(\mscrU)^\flat\) goes to \(0\) in \(\mcalO_{\mscrX}(\mscrU_i)^\flat\) for all \(i \in I\).
    Using \eqref{TiltDescription}, the element \(a\) can be written as \(a = (a^{(n)}) \in \lim \mcalO_{\mscrX}(\mscrU)\) which goes to \(0\) in \(\lim \mcalO_{\mscrX}(\mscrU_i)\) for all \(i \in I\) and then so does \(a^{(n)} \in \mcalO_{\mscrX}(\mscrU)\) in \(\mcalO_{\mscrX}(\mscrU_i)\).
    Since \(\mcalO_{\mscrX}\) is a sheaf, it follows \(a^{(n)} = 0\) in \(\mcalO_{\mscrX}(\mscrU)\) and \(a = 0\) in \(\mcalO_{\mscrX}(\mscrU)^\flat\).
    Moreover, take an element \((a_i) \in \prod_{i \in I} \mcalO_{\mscrX}(\mscrU_i)^{\flat}\) such that \(a_i = a_j\) in \(\mcalO_{\mscrX}(\mscrU_i \cap \mscrU_j)^\flat\) for all \(i \neq j\).
    Similarly using \eqref{TiltDescription}, there exists a unique element \(a^{(n)}\) in \(\mcalO_{\mscrX}(\mscrU)\) such that its image in \(\mcalO_{\mscrX}(\mscrU_i)\) is \(a_i^{(n)}\) in each \(n \geq 0\) for \(a_i = (a_i^{(n)}) \in \mcalO_{\mscrX}(\mscrU_i)^\flat\).
    The uniqueness ensures that the sequence \((a^{(n)})\) in the limit \(\lim_{a \mapsto a^p} \mcalO_{\mscrX}(\mscrU)\) and this goes to \(a_i\) on each \(\mscrU_i\).
    This shows that \(\mcalO_{\mscrX}^{\flat}\) is a sheaf of topological rings on \(\abs{\mscrX}\).
\end{proof}

\begin{definition} \label{DefTiltPerfd}
    Let \(\mscrX\) be a perfectoid formal scheme. The \emph{tilt of \(\mscrX\)}, denoted \(\mscrX^{\flat}\), is the topologically ringed space whose underlying topological space is \(\abs{\mscrX}\) and whose structure sheaf is given by \(\mcalO_{\mscrX}^\flat\) in \Cref{TiltSheaf}.
\end{definition}

\begin{lemma}
    Let \(\mscrX\) be a perfectoid formal scheme. Then the tilt \(\mscrX^\flat\) is a perfect formal scheme of characteristic \(p\).
\end{lemma}

\begin{proof}
    To show that \(\mscrX^\flat\) is a perfect formal scheme of characteristic \(p\), we may assume that the perfectoid formal scheme \(\mscrX = \Spf(R)\) is affine and it suffices to show that \(\mscrX^\flat\) is isomorphic to \(\Spf(R^\flat)\) where \(R^\flat \defeq \lim_{F} (R/(p))\) is the tilt of the complete adic perfectoid ring \(R\).
    Take an element \(\varpi \in R\) such that \(\varpi\) has a compatible system \(\{\varpi^{1/p^n}\}_{n \geq 0}\) in \(R\) and \(p = \varpi^p u\) for some unit \(u\) in \(R\).
    Since \(R\) is \(p\)-adically complete, we can set an ideal of definition \(J\) of \(R\) such that \(p \in J\).
    The underlying space \(\abs{\mscrX^\flat}\) is equal to \(\abs{\mscrX} = \Spec(R/J)\).
    For a perfectoid ring \(R\), we have an isomorphism \(R/\varpi R \cong R^\flat/\varpi^\flat R^\flat\) (see for example \cite{CS2024Purity}*{(2.1.2.2)}).
    In particular, there exists a finitely generated ideal \(J^\flat\) of \(R^\flat\) such that \(\varpi^\flat \in J^\flat\), \(R/J \cong R^\flat/J^\flat\) and the adic topology on \(R^\flat\) is the \(J^\flat\)-adic topology by \cite{takaya2026Relative}*{Lemma 4.5}.
    So the underlying space \(\abs{\Spf(R^\flat)}\) is the same as \(\Spec(R^\flat/J^\flat)\).
    Through the isomorphism, we can get a homeomorphism \(\abs{\mscrX^\flat} \approx \abs{\Spf(R^\flat)}\).

    Next, we will show that the structure sheaf \(\mcalO_{\mscrX^\flat}\) is isomorphic to the sheaf of topological rings \(\mcalO_{\Spf(R^\flat)}\) on \(\abs{\Spf(R^\flat)}\).
    Take an element \(f \in R\). Using the isomorphism \(R/\varpi R \cong R^\flat/\varpi^\flat R^\flat\), we can take (a lift of) the corresponding element \(f^\flat\) in \(R^\flat\). Then \(D(f) \subseteq \abs{\mscrX^\flat}\) corresponds to \(D(f^\flat) \subseteq \abs{\Spf(R^\flat)}\).
    We know that
    \begin{equation*}
        \mcalO_{\mscrX}^\flat(D_+(f)) = \mcalO_{\mscrX}(D(f))^\flat = (\widehat{R[1/f]})^\flat \cong \lim_{F} R[1/f]/J[1/f]
    \end{equation*}
    where the last isomorphism follows from \Cref{TiltAdicPerfd} and it has a map from \(R^\flat \cong \lim_{F} R^\flat/\varpi^\flat R^\flat\).
    This induces a map \(\widehat{R^\flat[1/f^\flat]} \to (\widehat{R[1/f]})^\flat\) of complete adic rings.
    Here, we have isomorphisms
    \begin{align*}
        \widehat{R^\flat[1/f^\flat]} & = \lim_{n \geq 0} R^\flat[1/f^\flat]/(J^\flat[1/f^\flat])^n \cong \lim_{n \geq 0} (R^\flat/(J^\flat)^{[p^n]})[1/f^\flat] \\
        & \cong \lim_{F} (R^\flat/J^\flat)[1/f^\flat] \cong \lim_{F} (R/J)[1/f] \cong (\widehat{R[1/f]})^\flat
    \end{align*}
    through the isomorphism \(R/J \cong R^\flat/J^\flat\) where the second isomorphism follows from the containment \((J^\flat)^{p^{nN}} \subseteq (J^\flat)^{[p^n]} \subseteq (J^\flat)^{p^n}\) for the finitely generated ideal \(J^\flat = (f_1, \dots, f_N)\) and the third isomorphism follows from the following commutative diagram
    \begin{center}
        \begin{tikzcd}
            {R^\flat/(J^\flat)^{[p^{n}]}}                        & R^\flat/J^\flat \arrow[l, "F^n"'] \arrow[l, "\cong"]                     \\
            {R^\flat/(J^\flat)^{[p^{n+1}]}} \arrow[u, two heads] & R^\flat/J^\flat \arrow[l, "F^{n+1}"'] \arrow[u, "F"'] \arrow[l, "\cong"]
        \end{tikzcd}
    \end{center}
    where \((J^\flat)^{[p^n]}\) is the ideal generated by the \(p^n\)-th power of the elements of \(J^\flat\).
    This shows that the map \(\mcalO_{\Spf(R^\flat)}(D_+(f^\flat)) = \widehat{R^\flat[1/f^\flat]} \to \mcalO_{\mscrX}^\flat(D_+(f))\) is an isomorphism.\footnote{We can also deduce this from \cite{takaya2026Relative}*{Corollary 4.11} instead of \Cref{TiltAdicPerfd}}.
\end{proof}

\begin{lemma} \label{TiltAdicPerfd}
    Let \(R\) be an adic perfectoid ring and let \(J\) be an ideal of definition with \(p \in J\).
    Set the \(J\)-adic completion \(\widehat{R}\).
    Then we have isomorphisms
    \begin{equation*}
        \widehat{R}^\flat \xrightarrow{\cong} \lim_{a \mapsto a^p} \widehat{R} \xrightarrow{\cong} \lim_{F} R/J
    \end{equation*}
    of multiplicative monoids.
\end{lemma}

\begin{proof}
    Since \(\widehat{R}\) is a \(p\)-adically complete ring, the first isomorphism \(\widehat{R}^\flat \cong \lim \widehat{R}\) is given by \cite{bhatt2018Integral}*{Lemma 3.2 (i)}.
    For each \(n \geq 0\), we can define a map
    \begin{equation*}
        \lim_{F} R/J \to \widehat{R}; \quad (\overline{a_n}) \mapsto \lim_{k \to \infty} a_{n+k}^{p^k}
    \end{equation*}
    which is well-defined because \(J\) is an ideal of definition and \(p \in J\) as in the \(p\)-adic case.
    This defines the inverse map
    \[
    \lim_{F} R/J \to \lim_{a \mapsto a^p} \widehat{R}; \quad
    (\overline{a_n}) \mapsto \bigl(\lim_{k \to \infty} a_{n+k}^{p^k}\bigr)_{n \geq 0}.
    \]
\end{proof}

\begin{lemma} \label{WittSheaf}
    Let \(\mscrX\) be a formal scheme. On the topological space \(\abs{\mscrX}\), we define a presheaf of complete topological rings by
    \begin{equation*}
        W\mcalO_{\mscrX} \colon \mscrU \mapsto W(\mcalO_{\mscrX}(\mscrU)) = \lim_{n} W_n(\mcalO_{\mscrX}(\mscrU))
    \end{equation*}
    where \(W\) denotes the Witt vector functor, and \(W(\mcalO_{\mscrX}(\mscrU))\) is endowed with the linear topology induced by the systems of neighborhoods \(\{W(I)\}_{I \in \mcalB}\) and \(\{V^kW(\mcalO_{\mscrX}(\mscrU))\}_{k \geq 1}\), where \(\mcalB\) is the set of open ideals of \(\mcalO_{\mscrX}(\mscrU)\).
    Then \(W\mcalO_{\mscrX}\) is a sheaf of complete topological rings on \(\abs{\mscrX}\).
\end{lemma}

\begin{proof}
    The completeness of \(W(\mcalO_{\mscrX}(\mscrU))\) follows from the same argument as in \cite{ishizuka2025Graded}*{Example 3.8 (2)}.
    As a set, \(W(\mcalO_{\mscrX}(\mscrU))\) is the same as the countable product \(\prod_{n \geq 0} \mcalO_{\mscrX}(\mscrU)\).
    So the sheaf property of \(\mcalO_{\mscrX}\) implies that \(W\mcalO_{\mscrX}\) is a sheaf of topological rings.
\end{proof}

\begin{definition} \label{AinfConst}
    Let \(\mscrX\) be a perfect formal scheme.\footnote{Here, perfect formal scheme means a perfectoid formal scheme of characteristic \(p\).} We can define a topologically ringed space whose underlying space is \(\abs{\mscrX}\) and whose structure sheaf is given by \(W\mcalO_{\mscrX}\).
    We can show that this pair \(W(\mscrX) \defeq (\abs{\mscrX}, W\mcalO_{\mscrX})\) is a perfect formal \(\delta\)-scheme.
    This is called the \emph{Witt formal scheme of \(\mscrX\)}.
\end{definition}

\begin{proof}
    Since \(\mcalO_{\mscrX}\) is a sheaf of perfect topological rings, \(W\mcalO_{\mscrX}(\mscrU) = W(\mcalO_{\mscrX}(\mscrU))\) is a perfect topological \(\delta\)-ring for any open subset \(\mscrU \subseteq \mscrX\) by \Cref{WittSheaf}. So it suffices to show that \(W(\mscrX)\) is a formal scheme.
    We will show that if \(\mscrX = \Spf(R)\) is an affine perfect formal scheme, then \(W(\mscrX) = \Spf(W(R))\) where \(W(R)\) has the topology induced from \(\{W(J^n), p^nW(R)\}_{n \geq 0}\) since \(R\) is an adic perfect ring with an ideal of definition \(J\).
    The underlying space of \(\Spf(W(R))\) is the same as the set of open prime ideals of \(W(R)\) and then this is the same as \(\Spec(W(R)/(p, W(J))) = \Spec(R/J) = \Spf(R)\). So the underlying space of \(W(\mscrX)\) is the same as \(\abs{\mscrX}\).

    For each \(f \in R\), we know that \(W\mcalO_{\mscrX}(D(f)) = W(\mcalO_{\mscrX}(D(f))) = W(\widehat{R[1/f]})\) by \Cref{WittSheaf}.
    Along the isomorphism \(R/J \cong W(R)/(p, W(J))\), the element \(f\) corresponds to the Teichm\"uller lift \([f]\) in \(W(R)\).
    Then the canonical map \(W(R) \to W(\widehat{R[1/f]})\) sends \([f]\) to a unit and thus induces a morphism of complete topological rings \((W(R)[1/[f]])^{\wedge} \to W(\widehat{R[1/f]})^{\wedge}\).
    Using the \(p\)-adic completeness and perfectness of \(R\) and \(\widehat{R[1/f]}\), it is enough to show that the induced map
    \begin{equation*}
        (W_n(R)[1/[f]])^{\wedge} \to W_n(\widehat{R[1/f]})
    \end{equation*}
    is an isomorphism of complete topological rings for any \(n \geq 1\) where the topology comes from a set of ideals \(\{W_n(J^k)\}_{k \geq 0}\).
    So we will show that the induced map \((W_n(R)/W_n(J^k))[1/[f]] \to W_n(\widehat{R[1/f]})/W_n(J^k)\) is an isomorphism for any \(k \geq 0\), namely, the map
    \begin{equation*}
        W_n(R/J^k)[1/[f]] \to W_n((R/J^k)[1/f])
    \end{equation*}
    is an isomorphism of rings. It follows from a more general result on Witt rings below (\Cref{WittLocalization}) and this completes the proof.
\end{proof}

\begin{lemma} \label{WittLocalization}
    Let \(R\) be a ring and let \(f\) be an element of \(R\).
    Then the induced map
    \begin{equation*}
        W_n(R)[1/[f]] \to W_n(R[1/f])
    \end{equation*}
    is an isomorphism of rings for any \(n \geq 1\).
\end{lemma}

\begin{proof}
    If an element \((a_0, \dots, a_{n-1})/[f]^i\) in \(W_n(R)[1/[f]]\) is zero on the right-hand side, then the multiplication
    \begin{equation*}
        (1/f^i, 0, \dots, 0) \cdot (a_0, \dots, a_{n-1}) = (a_0/f^i, a_1/f^{pi}, \dots, a_{n-1}/f^{p^{n-1}i}) \in W_n(R[1/f])
    \end{equation*}
    is zero. So there exists a sufficiently large \(N\) such that \(f^{p^jN} a_j = 0\) in \(R\) for all \(j = 0, \dots, n-1\).
    Then \([f^N] \cdot (a_0, \dots, a_{n-1}) = 0\) holds in \(W_n(R)\) and this shows the injectivity.

    For the surjectivity, let \((b_0, \dots, b_{n-1}) \in W_n((R[1/f]))\). Then there exists \(N\) and \(a_0, \dots, a_{n-1}\) in \(W_n(R)\) such that
    \begin{equation*}
        (b_0, \dots, b_{n-1}) = (a_0/f^N, a_1/f^{pN}, \dots, a_{n-1}/f^{p^{n-1}N}) \in W_n((R[1/f]))
    \end{equation*}
    which is the image of \((a_0, \dots, a_{n-1})/[f^N]\) in \(W_n(R)[1/[f]]\). This finishes the proof.
\end{proof}

\begin{definition} \label{AinfScheme}
    Let \(\mscrX\) be a perfectoid formal scheme.
    We define the \emph{\(A_{\inf}\)-formal scheme \(A_{\inf}(\mscrX)\) of \(\mscrX\)} to be the Witt scheme \(W(\mscrX)^\flat\) of the tilting \(\mscrX^\flat\) of \(\mscrX\). More explicitly, \(A_{\inf}(\mscrX)\) is the perfect formal \(\delta\)-scheme whose underlying space is \(\abs{\mscrX}\) and whose structure sheaf is given by the Witt sheaf \(W\mcalO_{\mscrX^\flat}\):
    \begin{equation*}
        \mcalO_{A_{\inf}(\mscrX)} \colon \mscrU \mapsto W(\mcalO_{\mscrX}(\mscrU)^\flat)
    \end{equation*}
    by \Cref{WittSheaf} and \Cref{TiltSheaf}.
\end{definition}

\begin{definition} \label{DefThetaMap}
    Using Fontaine's theta map (\cite{bhatt2018Integral}*{\S 3.1}) on each section, we obtain a morphism \(\theta \colon \mcalO_{A_{\inf}(\mscrX)} \to \mcalO_{\mscrX}\) of sheaves of rings on \(\abs{\mscrX}\):
    \begin{align*}
        \theta(\mscrU) \colon \mcalO_{A_{\inf}(\mscrX)}(\mscrU) = W(\mcalO_{\mscrX}(\mscrU)^\flat) & \to \mcalO_{\mscrX}(\mscrU) \\
        (a_0, a_1, \dots) & \mapsto \sharp(a_0) + p\sharp(a_1)^{1/p} + p^2\sharp(a_2)^{1/p^2} + \cdots
    \end{align*}
    where \(\sharp \colon \mcalO_{\mscrX}(\mscrU)^\flat \to \mcalO_{\mscrX}(\mscrU)\) is the sharp map used in \cite{bhatt2018Integral}*{Lemma 3.2 (i)}.
    Since this map is surjective on every affine open subset by perfectoidness, we conclude that \(\theta\) is a surjective morphism of sheaves on \(\abs{\mscrX}\).
\end{definition}

\begin{proposition} \label{AinfSchemeProperties}
    In the above situation, the kernel \(\mcalI \defeq \ker(\theta)\) defines a quasi-coherent ideal sheaf on \(A_{\inf}(\mscrX)\).
    Then we have the following properties:
    \begin{enumerate}
        \item The pair \((\mcalO_{A_{\inf}(\mscrX)}(\mscrU), \mcalI(\mscrU))\) is the perfect prism corresponding to the perfectoid ring \(\mcalO_{\mscrX}(\mscrU)\) on each affine open subset \(\mscrU\).
        \item We can identify \(\mscrX\) with the closed formal subscheme \(V(\mcalI)\) of \(A_{\inf}(\mscrX)\) defined by the ideal sheaf \(\mcalI\).
        \item The pair \((A_{\inf}(\mscrX), \mcalI)\) is a pseudo-perfect prismatic scheme and its associated closed subscheme \(V(\mcalI)\) is \(\mscrX\).
        \item The construction of \((A_{\inf}(\mscrX), \mcalI)\) is functorial in \(\mscrX\).
        \item If \(\mscrX\) admits a morphism to a formal spectrum \(\Spf(R)\) for some perfectoid \(R\), then \((A_{\inf}(\mscrX), \mcalI)\) is a perfect prismatic scheme and \(\mcalI\) is generated by a non-zero-divisor \(d\) in \(\ker(W(R^\flat) \to R)\).
    \end{enumerate}
    In particular, this gives rise to a functor from the category of perfectoid formal schemes to the category of perfect pseudo-prismatic schemes.
\end{proposition}

\begin{proof}
    All properties follow from the fact that the theta map \(\theta\) induces an isomorphism \(\mcalO_{A_{\inf}(\mscrX)}(\mscrU)/\mcalI(\mscrU) \xrightarrow{\cong} \mcalO_{\mscrX}(\mscrU)\) on each affine open subset \(\mscrU\).
    Note that the last assertion follows from the fact that the map \(\mscrX \to \Spf(R)\) induces a map of prisms \((A_{\inf}(R), (d)) \to (\mcalO_{A_{\inf}(\mscrX)}(\mscrU), \mcalI(\mscrU))\).
\end{proof}

\begin{proposition} \label{PerfectWittCorr}
    Let \((\mscrA, \mcalI)\) be a perfect pseudo-prismatic scheme.
    Then there exists a canonical isomorphism of formal schemes
    \begin{equation*}
        \theta \colon \mscrA \xrightarrow{\cong} A_{\inf}(\overline{\mscrA}) = W((\overline{\mscrA})^{\flat})
    \end{equation*}
    where \(A_{\inf}(\overline{\mscrA})\) is the \(A_{\inf}\)-formal scheme of the perfectoid formal scheme \(\overline{\mscrA}\) defined in \Cref{AinfScheme}.
    This induces a closed immersion \(\theta \colon \underline{\mscrA} \hookrightarrow W(\underline{\mscrA}^{\flat})\), which is Zariski locally identified with Fontaine's theta map and this can be identified with the closed immersion \(\underline{\mscrA} = V(\mcalI) \hookrightarrow \mscrA\) defined by the ideal sheaf \(\mcalI\) in \(\mscrA\).
    If the induced perfectoid formal scheme \(\overline{\mscrA}\) has a map \(\overline{\mscrA} \to \Spf(R)\) for some perfectoid \(R\), then \((\mscrA, \mcalI)\) is isomorphic to the perfect prismatic scheme \((W(\overline{\mscrA}^{\flat}), \ker(\theta_R))\) where \(\ker(\theta_R)\) is an ideal sheaf of \(W(\overline{\mscrA}^{\flat})\) generated by the kernel of the map \(A_{\inf}(R) \twoheadrightarrow R\).
\end{proposition}

\begin{proof}
    Since \(\overline{\mscrA}\) is a perfectoid formal scheme by \Cref{LocallyPrismaticScheme}, we can take the \(A_{\inf}\)-formal scheme \(A_{\inf}(\overline{\mscrA})\) whose underlying space is \(\abs{\overline{\mscrA}}\) and the structure sheaf is given by \(\mscrU \mapsto W(\mcalO_{\overline{\mscrA}}(\mscrU)^\flat)\) for any open subset \(\mscrU\) in \(\overline{\mscrA}\). 
    Since \((p, \mcalI)\) is topologically nilpotent, the underlying space \(\abs{\overline{\mscrA}}\) is equal to the underlying space \(\abs{V(p, \mcalI)} = \abs{\mscrA}\).
    Moreover, for any affine open subset \(\mscrU\), we know that \(\mcalO_{\overline{\mscrA}}(\mscrU)\) is isomorphic to \(\mcalO_{\mscrA}(\mscrU)/\mcalI(\mscrU)\).

    Since \(\mscrA\) is perfect as a formal \(\delta\)-scheme, we know that there exists a functorial isomorphism of topological rings
    \begin{equation*}
        A_{\inf}(\mcalO_{\overline{\mscrA}}(\mscrU)) = W(\mcalO_{\overline{\mscrA}}(\mscrU)^\flat) = W((\mcalO_{\mscrA}(\mscrU)/\mcalI(\mscrU))^{\flat}) \xrightarrow{\cong} \mcalO_{\mscrA}(\mscrU)
    \end{equation*}
    for any affine formal open subset \(\mscrU\) of \(\mscrA\) since \((\mcalO_{\mscrA}(\mscrU), \mcalI(\mscrU))\) is the corresponding perfect prism to the perfectoid ring \(\mcalO_{\overline{\mscrA}}(\mscrU)\).
    This shows the first claim.
    Composing the modulo \(\mcalI\) map, the above map is equal to Fontaine's theta map \(\theta(\mscrU)\) defined in \Cref{DefThetaMap}, and its kernel \(\ker(\theta(\mscrU))\) is equal to \(\mcalI(\mscrU)\) through the isomorphism above. We obtain the desired closed immersion.

    If there is a map \(\overline{\mscrA} \to \Spf(R)\) for some perfectoid ring \(R\), then the rigidity of maps of prisms ensures that \(\ker(\theta(\mscrU))\) is generated by a generator of the kernel of the map \(A_{\inf}(R) \to R\) for any affine formal open subset \(\mscrU\) of \(\mscrA\). Therefore, we have a map of prismatic schemes as in \Cref{LocallyPrismaticScheme}.
\end{proof}

\begin{corollary} \label{pTorsfPrismScheme}
    Let \((\mscrA, \mcalI)\) be a perfect prismatic formal scheme.
    Then \(H^1(\mscrA, \mcalO_{\mscrA})\) is \(p\)-torsion-free.
\end{corollary}

\begin{proof}
    Since \(\mscrA\) is \(p\)-torsion-free, we have an exact sequence \(0 \to \mcalO_{\mscrA} \xrightarrow{\times p} \mcalO_{\mscrA} \to \mcalO_{\mscrA}/p \to 0\).
    Then the \(p\)-torsion part \(H^1(\mscrA,\mcalO_{\mscrA})[p]\) is isomorphic to the image of the map \(H^0(\mcalO_{\mscrA}) \to H^0(\mcalO_{\mscrA}/p)\).
    For any element \(f\) in \(H^0(\mcalO_{\mscrA}/p)\) and for any affine open subset \(\mscrU\) of \(\mscrA\), we can take a lift \([f]\) in \(\mcalO_{\mscrA}(\mscrU) \cong W(\mcalO_{\mscrA}(\mscrU)/p)\) by using the Teichm\"uller lift under the isomorphism coming from \Cref{PerfectWittCorr}.
    Since this choice is canonical, we can glue this element and we get a lift of \(f\) on the global section.
    So the surjectivity follows and this shows that \(H^1(\mscrA, \mcalO_{\mscrA})[p] = 0\).
\end{proof}

\begin{corollary} \label{EquivCatPerfdFormal}
    The following two categories are equivalent.
    \begin{itemize}
        \item The category of perfectoid formal schemes.
        \item The category of pseudo-perfect prismatic schemes.
    \end{itemize}
    The functor is \(\mscrX \mapsto (A_{\inf}(\mscrX), \mcalI)\) in \Cref{AinfSchemeProperties} and the inverse functor is \((\mscrA, \mcalI) \mapsto \overline{\mscrA}\) in \Cref{DefPrismaticScheme}.

    Let \(R\) be an adic perfectoid ring and let \((A, I)\) be the corresponding prism.
    Then the equivalence of categories above restricts to the following equivalence of categories:
    \begin{itemize}
        \item The category of perfectoid formal schemes over \(\Spf(R)\).
        \item The category of perfect prismatic schemes over \((A, I)\).
    \end{itemize}
\end{corollary}

\begin{proof}
    Both assertions follow from \Cref{AinfSchemeProperties} and \Cref{PerfectWittCorr}.
\end{proof}

\subsection{Formal schemes associated to graded perfectoid rings and prisms}

\begin{proposition} \label{ProjConstGradedPerfd}
    Keep the notation in \Cref{ProjConstProGraded}.
    Assume that \((R, R_{\graded})\) is a pro-\(\setQ^n\)-graded perfectoid ring.
    Then the projective spectrum \(\Proj(R; R_{\graded, +})\) is a perfectoid formal scheme over \(\Spf(R_0)\).
\end{proposition}

\begin{proof}
    Fix a homogeneous ideal of definition \(I\) of \(R_{\graded}\).
    By \Cref{ProjProGraded}, \(\Proj(R; R_{\graded, +})\) is a formal scheme which is separated over \(\Spf(R_0)\) and has an affine open covering \(\{\Spf((\grcomp{I}{R_{\graded}[1/\mbff]})_0)\}\), where \(\mbff \in R_{\graded, +}\).
    Since \(R_{\graded}\) is a graded perfectoid ring, so is \(\grcomp{I}{R_{\graded}[1/\mbff]}\) by \Cref{LocalizationGradedPerfd}(1).
    Its degree \(0\) part \((\grcomp{I}{R_{\graded}[1/\mbff]})_0\) is a perfectoid ring by \Cref{graded-perfectoid-equiv}.
    So the formal scheme \(\Proj(R; R_{\graded, +})\) is locally isomorphic to the formal spectrum of a perfectoid ring and thus this is a perfectoid formal scheme.
\end{proof}

\begin{proposition} \label{ProjConstGradedPrism}
    Let \((A, I, A_{\graded})\) be a pro-\(\setQ_{\geq 0}^n\)-graded perfect prism.
    As in \Cref{ProjConstProGraded}, fix homogeneous ideals \(A_{1,+}, \dots, A_{n,+}\), and set
    \[
    A_{\graded,+} \defeq A_{1,+}\cdots A_{n,+}.
    \]
    Then the pair of the projective spectrum \(\Proj(A; A_{\graded, +})\) and an ideal sheaf \(I^{\Delta}\) associated to a homogeneous ideal of definition \(I\) is a perfect prismatic scheme over \((\Spf(A_0), I_0)\).
    The associated perfectoid formal scheme \(\overline{\Proj(A; A_{\graded, +})}\) is the perfectoid formal scheme \(\Proj(A/I, A_{\graded, +}/I)\).
\end{proposition}

\begin{proof}
    Let \(\varphi\) denote the Frobenius lift of \(A\). Since \(\varphi\) is an isomorphism of pro-\(\setQ^n\)-graded rings, we can take an isomorphism of formal schemes
    \begin{equation*}
        \varphi^* \colon \Proj(A; A_{\graded, +}) \to \Proj(A; A_{\graded, +})
    \end{equation*}
    and this induces a \(\delta\)-structure on \(\Proj(A; A_{\graded, +})\).
    Since \(I\) is generated by a non-zero-divisor \(d \in A_0\) by \cite{ishizuka2025Graded}*{Proposition 5.3(2)}, the associated ideal sheaf \(I^{\Delta}\) on \(\Proj(A; A_{\graded, +})\) is given by
    \begin{equation*}
        I^{\Delta}(D_+(\mbff)) = dA_{(\mbff)}^{\wedge} \subseteq A_{(\mbff)}^{\wedge} = \mcalO_{\Proj(A; A_{\graded, +})}(D_+(\mbff)).
    \end{equation*}
    This defines an invertible ideal sheaf \(I^{\Delta}\) on \(\Proj(A; A_{\graded, +})\).
    Zariski locally, this pair \((\Proj(A; A_{\graded, +}), I^{\Delta})\) can be identified with the perfect prism \(((A_{(\mbff)})^{\wedge}, dA_{(\mbff)}^{\wedge})\) and then \((\Proj(A; A_{\graded, +}), I^{\Delta})\) is a pseudo-perfect prismatic scheme by \Cref{DefPrismaticScheme}.
    Since we have a map \((\Proj(A; A_{\graded, +}), I^{\Delta}) \to (\Spf(A_0), I_0)\) of pseudo-prismatic schemes, \((\Proj(A; A_{\graded, +}), I^{\Delta})\) is a prismatic scheme by \Cref{LocallyPrismaticScheme}.
    The last assertion follows from the isomorphism
    \begin{equation*}
        (A_{(\mbff)})^{\wedge}/d (A_{(\mbff)})^{\wedge} \cong ((A/I)_{(\mbff)})^{\wedge}
    \end{equation*}
    on the affine formal open subset \(D_+(\mbff)\) of \(\Proj(A; A_{\graded, +})^{\wedge}\).
\end{proof}

\section{Derived category of quasi-coherent sheaves on formal schemes} \label{AppendixQCoh}

In this section, we recall and summarize the notion of the derived category of quasi-coherent sheaves on formal schemes and its properties.
In modern algebraic geometry, as in \Cref{SectionPerfectization}, one usually considers the derived category of quasi-coherent sheaves on formal schemes rather than the derived category of all sheaves, because the former has better properties than the latter.

However, our construction of the absolute perfectoidization of the structure sheaf on formal schemes in \Cref{SubSectionAbsPerfdArc} is based on the derived category of all sheaves, so we need to compare these two categories.

We believe that the results in this section are well known to experts (for example, \cite{zavyalov2025Almost}), but we could not find a complete reference for some of them, so we include detailed proofs for completeness.

\begin{notation}
    Throughout this section, let \(\mscrX\) be an adic quasi-compact separated formal scheme with an ideal sheaf of definition \(\mcalI\) and let \(R\) be a complete adic ring with an ideal of definition \(I\).
    For each \(n \geq 1\), we denote by
    \[
        i_n \colon \mscrX_n \hookrightarrow \mscrX
        \quad \text{(resp., } i_n \colon \Spec(R/I^n) \hookrightarrow \Spf(R)\text{)}
    \]
    the closed immersion defined by \(\mcalI^n\) (resp., \(I^n\)).
\end{notation}

\begin{remark}
    Since the underlying morphism of topological spaces of \(i_n\) is the identity, we can identify \(\mcalO_{\mscrX_n}\) with \(\mcalO_{\mscrX}/\mcalI^n\) and the pushforward functor
    \begin{equation*}
        i_{n, *} \colon \Shv((\mscrX_n)_{\Zar}, \mcalO_{\mscrX_n}) \to \Shv(\mscrX_{\Zar}, \mcalO_{\mscrX})
    \end{equation*}
    is the restriction of scalars along the morphism of sheaves of rings \(\mcalO_{\mscrX} \to i_{n, *}\mcalO_{\mscrX_n} \cong \mcalO_{\mscrX}/\mcalI^n\).
    In particular, \(i_{n, *}\) is exact and fully faithful, and the pullback functor \(i_n^*\) is the base change functor \(- \otimes_{\mcalO_{\mscrX}} \mcalO_{\mscrX_n}\).

    The derived pushforward (resp., pullback) functor
    \begin{equation*}
        Ri_{n, *} \colon \mcalD((\mscrX_n)_{\Zar}, \mcalO_{\mscrX_n}) \to \mcalD(\mscrX_{\Zar}, \mcalO_{\mscrX}) \quad \text{(resp., } Li_n^* \colon \mcalD(\mscrX_{\Zar}, \mcalO_{\mscrX}) \to \mcalD((\mscrX_n)_{\Zar}, \mcalO_{\mscrX_n})\text{)}
    \end{equation*}
    is also the restriction of scalars (resp., derived base change) along \(\mcalO_{\mscrX} \to Ri_{n, *}\mcalO_{\mscrX_n} \cong \mcalO_{\mscrX}/\mcalI^n\) and thus \(Ri_{n, *}\) is fully faithful.
\end{remark}

\subsection{Weakly proregular sequences}

Since we are working on ``classical'' formal schemes, we have to consider the discrepancy between the derived quotients \(R/^L \underline{f^n}\) and the usual quotients \(R/I^n\), where \(\underline{f} = f_1, \dots, f_r\) is a sequence of elements of \(R\) generating \(I\).
This discrepancy is controlled by the notion of weak proregularity:

\begin{definition}[{\cite{yekutieli2021Weaka}*{Definition 3.1}}] \label{DefWeaklyProregular}
    Let \(\underline{f} = f_1, \dots, f_r\) be a sequence of elements of \(R\).
    We say that \(\underline{f}\) is \emph{weakly proregular} if the inverse system of \(R\)-modules \(\{H^i(\Kos(R; \underline{f^n}))\}_{n \geq 1}\) is pro-zero for any \(i > 0\) where \(\Kos(R; \underline{f^n})\) is the Koszul complex of \(R\) with respect to the sequence \(\underline{f^n} = f_1^n, \dots, f_r^n\).
    We say that an ideal of \(R\) is \emph{weakly proregular} if it is generated by a weakly proregular sequence.
\end{definition}

This can be rephrased as the following:

\begin{lemma} \label{WeaklyProregularEquiv}
    Let \(\underline{f} = f_1, \dots, f_r\) be a sequence of elements of \(R\). Then the following are equivalent:
    \begin{enumerate}
        \item \(\underline{f}\) is weakly proregular.
        \item The morphism \(\{R/^L \underline{f^n}\}_{n \geq 1} \to \{R/(f_1^n, \dots, f_r^n)R\}_{n \geq 1}\) in \(\Pro(\mcalD(R))\) is an isomorphism.
        \item The morphism \(\{R/^L \underline{f^n}\}_{n \geq 1} \to \{R/I^n\}_{n \geq 1}\) in \(\Pro(\mcalD(R))\) is an isomorphism.
    \end{enumerate}
\end{lemma}

\begin{proof}
    Since the canonical morphism \(\{R/(f_1^n, \dots, f_r^n)R\}_{n \geq 1} \to \{R/I^n\}_{n \geq 1}\) in \(\Pro(\mcalD(R))\) is an isomorphism, the equivalence \((2) \Leftrightarrow (3)\) is immediate.

    \((1) \Rightarrow (2)\): Take the fiber \(\{K_n\}_{n \geq 1}\) of the morphism \(\{R/^L \underline{f^n}\}_{n \geq 1} \to \{R/(f_1^n, \dots, f_r^n)R\}_{n \geq 1}\) in \(\Pro(\mcalD(R))\).
    Then we know \(K_n \cong \tau^{\leq -1}\Kos(R; \underline{f^n})\) for any \(n \geq 1\). This \(K_n\) is concentrated in cohomological degree \([-r, -1]\) and the inverse system of \(R\)-modules \(\{H^i(K_n)\}_{n \geq 1}\) is pro-zero for any \(i > 0\) by the assumption.
    Considering the fiber sequence
    \begin{equation*}
        \{\tau^{\leq -i-1}K_n\}_{n \geq 1} \to \{\tau^{\leq -i}K_n\}_{n \geq 1} \to \{H^{-i}(K_n)[-i]\}_{n \geq 1}
    \end{equation*}
    in \(\Pro(\mcalD(R))\) for any \(i > 0\), we can show by induction on \(i\) that \(\{\tau^{\leq -i}K_n\}_{n \geq 1}\) is pro-zero for any \(i > 0\), so \(K_n\) is pro-zero as well, which proves \((2)\).

    For the implication \((2) \Rightarrow (1)\), the pro-isomorphism implies that the fiber \(\{\tau^{\leq -1}\Kos(R; \underline{f^n})\}_{n \geq 1}\) is pro-zero.
    This implies that the inverse system of \(R\)-modules \(\{H^i(\Kos(R; \underline{f^n}))\}_{n \geq 1}\) is pro-zero for any \(i > 0\), so the sequence \(\underline{f}\) is weakly proregular.
\end{proof}

\begin{lemma} \label{CatEquivCompLimit}
    Assume that \(I\) is generated by a weakly proregular sequence \(\underline{f} = f_1, \dots, f_r\) (for example, \(R\) is Noetherian or \(I\) is principal and \(R\) has bounded \(I^{\infty}\)-torsion).
    Then we have a commutative diagram
    \begin{equation*}
        \begin{tikzcd}
            \mcalD^{\comp{I}}(R) \arrow[d, "- \otimes^L_R R/^L \underline{f^n}"'] \arrow[rrrr, hook]                                 &  &  &  & \mcalD(R) \arrow[d, "\{- \otimes^L_R R/I^n\}"] \\
            \lim_{n \geq 1} \mcalD(R/^L \underline{f^n}) \arrow[rrrr, "\{(- \otimes^L_{R/^L \underline{f^n}} R/I^n)\}_n"'] &  &  &  & \lim_{n \geq 1} \mcalD(R/I^n)                 
        \end{tikzcd}
    \end{equation*}
    whose left and lower functors are equivalences of \(\infty\)-categories. In particular, the quasi-inverse functor is given by
    \begin{equation*}
        \lim_{n \geq 1} \mcalD(R/I^n) \to \mcalD^{\comp{I}}(R); \quad \{M_n\}_{n \geq 1} \mapsto \lim_{n \geq 1} M_n.
    \end{equation*}
\end{lemma}

\begin{proof}
    The left vertical functor is an equivalence in general. See \cite{sahai2026Derived}*{Corollary 2.23 and Theorem 2.26}, for example.
    Since \(f_1, \dots, f_r\) is weakly proregular, the morphism \(\{R/^L \underline{f^n}\}_{n \geq 1} \to \{R/I^n\}_{n \geq 1}\) in \(\Pro(\mcalD(R))\) is an isomorphism (\Cref{WeaklyProregularEquiv}), so the lower functor is an equivalence as well.
    For any \(\{M_n\}_{n \geq 1}\) in \(\lim_{n \geq 1} \mcalD(R/I^n)\), we have an object \(\{M_n'\}_{n \geq 1}\) in \(\lim_{n \geq 1} \mcalD(R/^L \underline{f^n})\) which goes to \(\{M_n\}_{n \geq 1}\) under the lower functor.
    Since a quasi-inverse to the left vertical functor is given by taking the limit, this \(\{M_n\}_{n \geq 1}\) corresponds to the limit \(\lim_{n \geq 1} M_n'\) in \(\mcalD^{\comp{I}}(R)\).
    On the other hand, since \(\{R/^L \underline{f^n}\}_{n \geq 1} \to \{R/I^n\}_{n \geq 1}\) is an isomorphism, we have isomorphisms
    \begin{equation*}
        \lim_{n \geq 1} M_n' \xrightarrow{\cong} \lim_{n \geq 1} (M_n' \otimes^L_{R/^L \underline{f^n}} R/I^n) \cong \lim_{n \geq 1} M_n
    \end{equation*}
    in \(\mcalD(R)\) as desired.
\end{proof}

\subsection{Derived completion functors and derived categories}

We can associate a sheaf on \(\Spf(R)\) to any \(R\)-module by taking the completion of the quasi-coherent sheaf associated to it on \(\Spec(R)\) as follows.
Note that we define another notion of completion for sheaves on \(\Proj(R)\) for a graded ring \(R\) in \Cref{AssociatedSheafFormalGraded}.

\begin{definition}[{cf. \cite{grothendieck1971Elements}*{\S 10.10}}] \label{DefCompletionQcoh}
    Let \(M\) be an \(R\)-module.
    We define
    \begin{equation*}
        M^{\Delta} \defeq \lim_{n \geq 1} \widetilde{M/I^nM}
    \end{equation*}
    to be the completion of the quasi-coherent sheaf \(\widetilde{M}\) on \(\Spec(R)\) where \(\widetilde{I}\) is the quasi-coherent ideal sheaf on \(\Spec(R)\) associated to \(I\).
    We call it the \emph{quasi-coherent sheaf associated to \(M\)} and this defines a functor \((-)^{\Delta} \colon \Mod_R \to \Shv(\Spf(R), \mcalO_{\Spf(R)})\).
    This is the \(\widetilde{I}\)-adic completion of the quasi-coherent sheaf \(\widetilde{M}\) on \(\Spec(R)\).
\end{definition}

\begin{construction} \label{ConstDerivedCompletionDiagram}
    Since \(M^{\Delta} \otimes_{\mcalO_{\Spf(R)}} \mcalO_{\Spec(R/I^n)} \cong \widetilde{M/I^nM}\), we have a commutative diagram
    \begin{equation*}
        \begin{tikzcd}
            \Mod(R) \arrow[d, "- \otimes_R R/I^n"'] \arrow[r, "(-)^{\Delta}", two heads] & \mcalC \arrow[d, "\restr{(-)}{\Spec(R/I^n)}"] \arrow[r, hook] & {\Shv(\Spf(R)_{\Zar}, \mcalO_{\Spf(R)})} \arrow[d, "\restr{(-)}{\Spec(R/I^n)}"] \\
            \Mod(R/I^n) \arrow[r, "\widetilde{(-)}"] \arrow[r, "\simeq"']      & \QCoh(\Spec(R/I^n)) \arrow[r, hook]                                   & {\Shv(\Spec(R/I^n)_{\Zar}, \mcalO_{\Spec(R/I^n)})}                             
        \end{tikzcd}
    \end{equation*}
    of abelian categories which is compatible with the restriction maps along \(n\), where \(\mcalC\) is the essential image of the functor \((-)^{\Delta}\).\footnote{This is the category of ``quasi-coherent sheaves'' on \(\Spf(R)\) in the sense of \cite{gabber2018Foundations}*{Definition 15.1.31}. See also \Cref{AffineVanishing}.}

    Since the abelian category of \(R\)-modules has enough \(K\)-projective objects, the functor \((-)^{\Delta} \colon \Mod_R \to \Shv(\Spf(R)_{\Zar}, \mcalO_{\Spf(R)})\) has a left derived functor
    \begin{equation*}
        (-)^{L\Delta} \colon \mcalD(R) \to \mcalD(\Spf(R)_{\Zar}, \mcalO_{\Spf(R)}).
    \end{equation*}
    Using the following lemma (\Cref{LDeltaQCoh}), we can extend the above commutative diagram of abelian categories to one of \(\infty\)-categories as follows:
    \begin{equation} \label{DerivedCompletionDiagram}
        \begin{tikzcd}
            \mcalD(R) \arrow[d, "\{- \otimes^L_R R/I^n\}_{n \geq 1}"'] \arrow[rr, "(-)^{L\Delta}"] &                                              & {\mcalD(\Spf(R)_{\Zar}, \mcalO_{\Spf(R)})} \arrow[d, "\{Li_n^*\}_{n \geq 1}"] \\
            \lim_{n \geq 1}\mcalD(R/I^n) \arrow[r, "\widetilde{(-)}"] \arrow[r, "\simeq"']         & \lim_{n \geq 1} \mcalD_{\qcoh}(\Spec(R/I^n)) \arrow[r, hook] & {\lim_{n \geq 1} \mcalD(\Spec(R/I^n)_{\Zar}, \mcalO_{\Spec(R/I^n)})}                             
        \end{tikzcd}
    \end{equation}
    of \(\infty\)-categories, where the right vertical functor is the derived pullback functor along the closed immersion \(i_n \colon \Spec(R/I^n) \to \Spf(R)\) for each \(n \geq 1\).
\end{construction}

\begin{lemma} \label{LDeltaQCoh}
    Suppose that the ideal of definition $I$ of $R$ is weakly proregular.
    Then the diagram \eqref{DerivedCompletionDiagram} in \Cref{ConstDerivedCompletionDiagram} commutes.
    In particular, for any $M \in \mcalD(R)$ and any $n \ge 1$, there is a canonical isomorphism
    \begin{equation*}
        Li_n^*(M^{L\Delta}) \simeq \widetilde{M \otimes^L_R R/I^n}
    \end{equation*}
    in $\mcalD(\Spec(R/I^n)_{\Zar}, \mcalO_{\Spec(R/I^n)})$, where $i_n \colon \Spec(R/I^n) \to \Spf(R)$ denotes the closed immersion.
    Moreover, we have a canonical isomorphism
    \begin{equation*}
        M^{L\Delta} \cong \lim_{n \geq 1} \widetilde{(M \otimes^L_R R/I^n)} \cong \lim_{n \geq 1} Li_n^*(M^{L\Delta})
    \end{equation*}
    in \(\mcalD(\Spf(R)_{\Zar}, \mcalO_{\Spf(R)})\).
\end{lemma}

\begin{proof}
    Let \(P^{\bullet}\) be a \(K\)-projective resolution of \(M\) such that each term \(P^i\) is a free \(R\)-module (such a resolution exists; see, for example, \cite{keller1994Deriving}*{\S 3.1} and \cite{porta2014Homology}*{Proposition 2.1 (1)}).
    By the definition of the left derived functor, \(M^{L\Delta}\) is represented by \((P^{\bullet})^{\Delta}\) in \(\mcalD(\Spf(R)_{\Zar}, \mcalO_{\Spf(R)})\) and the derived pullback \(Lu^*\widetilde{M}\) of \(\widetilde{M} \in \mcalD(\Spec(R)_{\Zar}, \mcalO_{\Spec(R)})\) along the canonical morphism \(u \colon \Spf(R) \to \Spec(R)\) is represented by \(u^*\widetilde{P^{\bullet}}\).
    This produces a morphism
    \begin{equation} \label{DerivedCompletionMorphismFromPullback}
        Lu^*\widetilde{M} \to M^{L\Delta}
    \end{equation}
    in \(\mcalD(\Spf(R)_{\Zar}, \mcalO_{\Spf(R)})\).

    We first show that the above morphism \eqref{DerivedCompletionMorphismFromPullback} is the derived \(I^{\Delta}\)-completion morphism of \(Lu^*\widetilde{M} \cong u^*\widetilde{P^{\bullet}}\) in \(\mcalD(\Spf(R)_{\Zar}, \mcalO_{\Spf(R)})\).
    Since \(I\) is generated by a weakly proregular sequence \(\underline{f} = f_1, \dots, f_r\), the derived \(I^{\Delta}\)-completion of \(u^*\widetilde{P^{\bullet}}\) in \(\mcalD(\Spf(R)_{\Zar}, \mcalO_{\Spf(R)})\) is given by the limit
    \begin{align*}
        \dcomp{I^{\Delta}}{u^*\widetilde{P^{\bullet}}} & = \lim_{n \geq 1} ((u^*\widetilde{P^{\bullet}}) \otimes^L_{\mcalO_{\Spf(R)}} \mcalO_{\Spf(R)}/^L \underline{f^n}) \cong \lim_{n \geq 1} ((u^*\widetilde{P^{\bullet}}) \otimes^L_{\mcalO_{\Spf(R)}} \mcalO_{\Spf(R)}/(I^{\Delta})^n) \\
        & \cong \lim_{n \geq 1} ((u^*\widetilde{P^{\bullet}}) \otimes_{\mcalO_{\Spf(R)}} \mcalO_{\Spec(R/I^n)}) \cong \lim_{n \geq 1} i_n^*u^*\widetilde{P^{\bullet}} \cong \lim_{n \geq 1} \widetilde{P^{\bullet}/I^nP^{\bullet}}
    \end{align*}
    in \(\mcalD(\Spf(R)_{\Zar}, \mcalO_{\Spf(R)})\).
    We have to prove that the last limit in the derived category is represented by the term-wise limit \((\widetilde{P^{\bullet}})^{\Delta}\).

    For any \(f \in R\) with \(D(f) = \Spf(R[1/f]) \subseteq \Spf(R)\), we have
    \begin{align*}
        R\Gamma(D(f), \dcomp{I^{\Delta}}{u^*\widetilde{P^{\bullet}})} & \simeq \lim_{n \geq 1} R\Gamma(D(f), \widetilde{P^{\bullet}/I^nP^{\bullet}}) \simeq \lim_{n \geq 1} ((P^{\bullet}/I^nP^{\bullet})[1/f])
    \end{align*}
    in \(\mcalD(R)\).
    On the other hand, using the vanishing of higher cohomology on affine formal schemes (\Cref{AffineVanishing}) and the spectral sequence of hypercohomology of complexes, we have
    \begin{equation*}
        R\Gamma(D(f), (P^\bullet)^{\Delta}) \cong H^0(D(f), (P^\bullet)^{\Delta}) \cong (P^{\bullet})^{\Delta}(D(f)).
    \end{equation*}
    The former is computed by the latter complex, by the Mittag--Leffler condition on \(\mcalD(R)\) (see, for example, \citeSta{07KW}).
    This shows that \(M^{L\Delta} \cong (P^{\bullet})^{\Delta}\) is isomorphic to \(\dcomp{I^{\Delta}}{Lu^*\widetilde{M}} \cong \dcomp{I^{\Delta}}{u^*\widetilde{P^{\bullet}}}\) in \(\mcalD(\Spf(R)_{\Zar}, \mcalO_{\Spf(R)})\).

    Taking the derived quotient and the base change along \(\mcalO_{\Spf(R)}/^L \underline{f^n} \to \mcalO_{\Spf(R)}/(I^{\Delta})^n\), we obtain the isomorphism
    \begin{equation*}
        (Lu^*\widetilde{M}) \otimes^L_{\mcalO_{\Spf(R)}} \mcalO_{\Spf(R)}/(I^{\Delta})^n \xrightarrow{\cong} M^{L\Delta} \otimes^L_{\mcalO_{\Spf(R)}} \mcalO_{\Spf(R)}/(I^{\Delta})^n
    \end{equation*}
    in \(\mcalD(\Spf(R)_{\Zar}, \mcalO_{\Spf(R)})\), which is the same as
    \begin{equation*}
        \widetilde{M \otimes^L_R R/I^n} \cong Li_n^*(Lu^*\widetilde{M}) \xrightarrow{\cong} Li_n^*M^{L\Delta}
    \end{equation*}
    in \(\mcalD(\Spec(R/I^n)_{\Zar}, \mcalO_{\Spec(R/I^n)})\) by the exactness of \(i_n\).
    Together with the full faithfulness of \(i_{n, *}\), we get the desired isomorphism
    \begin{equation*}
        \widetilde{M \otimes^L_R R/I^n} \cong Li_n^*(M^{L\Delta})
    \end{equation*}
    in \(\mcalD(\Spec(R/I^n)_{\Zar}, \mcalO_{\Spec(R/I^n)})\) for any \(n \geq 1\).

    Since we have shown the isomorphism
    \begin{equation*}
        M^{L\Delta} \cong \dcomp{I^{\Delta}}{Lu^*\widetilde{M}} \cong \lim_{n \geq 1} \widetilde{(M \otimes^L_R R/I^n)}
    \end{equation*}
    in \(\mcalD(\Spf(R)_{\Zar}, \mcalO_{\Spf(R)})\), this proves the final assertion.
\end{proof}

\begin{definition} \label{DefDerivedQcohFormalScheme}
    We define the \emph{derived category of quasi-coherent sheaves on \(\mscrX\)} by the full subcategory \(\mcalD_{\qcoh}(\mscrX)\) of \(\mcalD(\mscrX_{\Zar}, \mcalO_{\mscrX})\) consisting of objects \(\mcalF\) such that the canonical morphism
    \begin{equation} \label{CompleteCondition}
        \mcalF \to \lim_{n \geq 1} Li_n^*(\mcalF)
    \end{equation}
    is an isomorphism in \(\mcalD(\mscrX_{\Zar}, \mcalO_{\mscrX})\) and the restriction \(Li_n^*(\mcalF)\) belongs to \(\mcalD_{\qcoh}(\mscrX_n)\) for any \(n \geq 1\).
    This condition can be checked on any affine open formal subscheme of \(\mscrX\).
\end{definition}

\begin{proposition} \label{DerivedQCohEquivAffine}
    Assume that \(I\) is generated by a weakly proregular sequence \(\underline{f} = f_1, \dots, f_r\).
    Then the canonical functor
    \begin{equation*}
        \mcalD_{\qcoh}(\Spf(R)) \to \lim_{n \geq 1} \mcalD_{\qcoh}(\Spec(R/I^n)); \quad \mcalF \mapsto \{Li_n^*(\mcalF)\}_{n \geq 1}
    \end{equation*}
    is an equivalence of \(\infty\)-categories whose quasi-inverse is given by \(\{\mcalF_n\}_{n \geq 1} \mapsto \lim_{n \geq 1} \mcalF_n\). In particular, this induces an equivalence of \(\infty\)-categories
    \begin{align*}
        \mcalD_{\qcoh}(\Spf(R)) & \xrightarrow{\cong} \mcalD^{\comp{I}}(R) \\
        \mcalF & \mapsto R\Gamma(\Spf(R), \mcalF) \\
        M^{L\Delta} \cong \lim_{n \geq 1} \widetilde{(M \otimes^L_R R/I^n)} & \mapsfrom M
    \end{align*}
\end{proposition}

\begin{proof}
    By the definition of \(\mcalD_{\qcoh}(\Spf(R))\), the functor is well defined.
    Under the equivalence \(R\Gamma(\Spec(R/I^n), -) \colon \mcalD_{\qcoh}(\Spec(R/I^n)) \simeq \mcalD(R/I^n) \colon \widetilde{(-)}\), the functor is identified with
    \begin{equation*}
        \mcalD_{\qcoh}(\Spf(R)) \to \lim_{n \geq 1} \mcalD(R/I^n); \quad \mcalF \mapsto \{R\Gamma(\Spec(R/I^n), Li_n^*(\mcalF))\}_{n \geq 1}.
    \end{equation*}

    For any \(\{M_n\}_{n \geq 1}\) in \(\lim_{n \geq 1} \mcalD(R/I^n)\), we have an object \(M \defeq \lim_{n \geq 1} M_n\) of \(\mcalD^{\comp{I}}(R)\) such that \(\{M \otimes^L_R R/I^n\}_{n \geq 1} \cong \{M_n\}_{n \geq 1}\) in \(\lim_{n \geq 1} \mcalD(R/I^n)\) by the equivalences in \Cref{CatEquivCompLimit}.
    Using the diagram \eqref{DerivedCompletionDiagram} in \Cref{ConstDerivedCompletionDiagram}, which is commutative by \Cref{LDeltaQCoh}, we have
    \begin{equation*}
        \{Li_n^*(M^{L\Delta})\}_{n \geq 1} \cong \{\widetilde{M \otimes^L_R R/I^n}\}_{n \geq 1} \cong \{\widetilde{M_n}\}_{n \geq 1}
    \end{equation*}
    in \(\lim_{n \geq 1}\mcalD_{\qcoh}(\Spec(R/I^n))\).
    Since this \(M^{L\Delta}\) is an object of \(\mcalD_{\qcoh}(\Spf(R))\) by \Cref{LDeltaQCoh} and \Cref{DefDerivedQcohFormalScheme}, this shows that the functor is essentially surjective.

    Next, we show that the functor is fully faithful. This follows from the fact that the composition
    \begin{equation*}
        \mcalD_{\qcoh}(\Spf(R)) \to \lim_{n \geq 1} \mcalD_{\qcoh}(\Spec(R/I^n)) \to \mcalD(\Spf(R)_{\Zar}, \mcalO_{\Spf(R)}); \quad \mcalF \mapsto \lim_{n \geq 1} (Li_n^*(\mcalF))
    \end{equation*}
    is assumed to be the inclusion functor by the condition \eqref{CompleteCondition} in \Cref{DefDerivedQcohFormalScheme}.

    Combining these arguments, the quasi-inverse is given by \(\{\mcalF_n\}_{n \geq 1} \mapsto \lim_{n \geq 1} \mcalF_n\) as claimed and the equivalence \(\mcalD_{\qcoh}(\Spf(R)) \simeq \mcalD^{\comp{I}}(R)\) is obtained.
\end{proof}

Finally, we can globalize the equivalence in \Cref{DerivedQCohEquivAffine} to general formal schemes:
See \Cref{CompatibilityQCDerived} for the compatibility for bounded \(p\)-adic formal schemes.

\begin{corollary} \label{DerivedQCohEquivGeneral}
    Assume that the ideal sheaf of definition \(\mcalI\) of \(\mscrX\) is generated by a weakly proregular sequence on any affine open formal subscheme of \(\mscrX\) (for example, if \(\mscrX\) is locally Noetherian or \(\mcalI\) is principal and \(\mscrX\) has bounded \(\mcalI^{\infty}\)-torsion).
    Then the canonical functor
    \begin{equation*}
        \mcalD_{\qcoh}(\mscrX) \to \lim_{n \geq 1} \mcalD_{\qcoh}(\mscrX_n); \quad \mcalF \mapsto \{Li_n^*(\mcalF)\}_{n \geq 1}
    \end{equation*}
    is an equivalence of \(\infty\)-categories, where \(\mscrX_n\) is the closed subscheme of \(\mscrX\) defined by \(\mcalI^n\) for each \(n \geq 1\).
\end{corollary}

\begin{proof}
    By the condition \eqref{CompleteCondition} in \Cref{DefDerivedQcohFormalScheme}, the functor is well defined and fully faithful.

    For any \(\{\mcalF_n\}_{n \geq 1}\) in \(\lim_{n \geq 1} \mcalD_{\qcoh}(\mscrX_n)\), we can take an object
    \begin{equation*}
        \mcalF \defeq \lim_{n \geq 1} \mcalF_n
    \end{equation*}
    of \(\mcalD(\mscrX_{\Zar}, \mcalO_{\mscrX})\) where \(i_n \colon \mscrX_n \to \mscrX\) is the closed immersion for each \(n \geq 1\).    

    We have \(\restr{\mcalF}{\mscrU} \cong \lim_{n \geq 1} \restr{\mcalF_n}{\mscrU}\) and each \(\restr{\mcalF_n}{\mscrU}\) belongs to \(\mcalD_{\qcoh}(\mscrU \cap \mscrX_n)\) by the assumption.
    So the equivalence in \Cref{DerivedQCohEquivAffine} implies that \(\restr{\mcalF}{\mscrU}\) belongs to \(\mcalD_{\qcoh}(\mscrU)\).
    Also, the equivalence shows that
    \begin{equation*}
        \restr{(Li_n^*\mcalF)}{\mscrU \cap \mscrX_n} \cong Li_{n, \mscrU}^*((\restr{\mcalF}{\mscrU})) \cong \restr{\mcalF_n}{\mscrU \cap \mscrX_n},
    \end{equation*}
    where \(i_{n, \mscrU} \colon \mscrU \cap \mscrX_n \to \mscrU\) is the closed immersion for each \(n \geq 1\).
    Consequently, \(\mcalF\) belongs to \(\mcalD_{\qcoh}(\mscrX)\) and it goes to \(\{\mcalF_n\}_{n \geq 1}\) in \(\lim_{n \geq 1} \mcalD_{\qcoh}(\mscrX_n)\) under the functor. 

\end{proof}

\subsection{Cohomology of derived completion}

We review the compatibility of cohomology on schemes and cohomology on formal schemes via derived completion.

\begin{lemma} \label{DerivedCompletionFormalSchemes}
    Assume that \(\mscrX\) is the \(I\)-adic formal completion of a scheme \(X\) over \(\Spec(R)\).
    Then for any \(\mcalF \in \mcalD_{\qcoh}(X)\), we have a canonical isomorphism
    \begin{equation*}
        \dcomp{I_0}{R\Gamma(X, \mcalF)} \xrightarrow{\cong} R\Gamma(\mscrX, \mcalF^{\wedge})
    \end{equation*}
    in \(\mcalD(R)\), where \(\mcalF^{\wedge} \defeq \lim_{n \geq 1} (Lj_n^*\mcalF)\) in \(\mcalD_{\qcoh}(\mscrX)\) is the corresponding object of \(\{Lj_n^*\mcalF\}_{n \geq 1}\) in \(\lim_{n \geq 1} \mcalD_{\qcoh}(\mscrX_n)\) under the equivalence in \Cref{DerivedQCohEquivGeneral} where \(j_n \colon \mscrX_n \cong X_n \to X\) is the closed immersion defined by \(I^n\).
\end{lemma}

\begin{proof}
    We compute
    \begin{align*}
        R\Gamma(\mscrX, \mcalF^{\wedge}) & \cong \lim_{n \geq 1} R\Gamma(\mscrX_n, Lj_n^*\mcalF) \cong \lim_{n \geq 1} R\Gamma(X, Rj_{n, *} Lj_n^*\mcalF) \cong \lim_{n \geq 1} R\Gamma(X, \mcalF \otimes^L_{\mcalO_X} \mcalO_{X_n}) \\
        & \cong R\Gamma(X, \lim_{n \geq 1} (\mcalF \otimes^L_{\mcalO_X} \mcalO_{X_n}))
    \end{align*}
    where the third isomorphism follows from \citeSta{08EU}.
    Then we have a canonical morphism
    \begin{equation*}
        R\Gamma(X, \mcalF) \to R\Gamma(X, \lim_{n \geq 1} (\mcalF \otimes^L_{\mcalO_X} \mcalO_{X_n})) \cong R\Gamma(\mscrX, \mcalF^{\wedge})
    \end{equation*}
    in \(\mcalD(R)\).
    We show that the above morphism is an isomorphism by the formal function theorem (e.g., \cite{gortz2023Algebraic}*{Corollary 24.29}).
\end{proof}

\begin{lemma} \label{SymMonDerivedCompletionFormalSchemes}
    Keep the setting of \Cref{DerivedCompletionFormalSchemes}.
    Then the construction \(\mcalF \mapsto \mcalF^{\wedge}\) in \Cref{DerivedCompletionFormalSchemes} defines a functor
    \begin{equation*}
        (-)^{\wedge} \colon \mcalD_{\qcoh}(X) \to \mcalD_{\qcoh}(\mscrX); \quad \mcalF \mapsto \lim_{n \geq 1} (Lj_n^*\mcalF)
    \end{equation*}
    which preserves colimits and, for any \(\mcalF, \mcalG \in \mcalD_{\qcoh}(X)\), the canonical morphism
    \begin{equation*}
        (\mcalF \otimes^L_{\mcalO_X} \mcalG)^{\wedge} \xrightarrow{\cong} \dcomp{I_0}{\mcalF^{\wedge} \otimes^L_{\mcalO_{\mscrX}} \mcalG^{\wedge}}
    \end{equation*}
    is an isomorphism in \(\mcalD(\mscrX_{\Zar}, \mcalO_{\mscrX})\), where \(\dcomp{I_0}{-}\) is the derived \(I_0\mcalO_{\mscrX}\)-completion.
\end{lemma}

\begin{proof}
    Since the construction is functorial on open subsets of \(X\), we may assume that \(X = \Spec(R)\) is affine.
    Using the categorical equivalences \(\mcalD_{\qcoh}(X) \simeq \mcalD(R)\) and \(\mcalD_{\qcoh}(\mscrX) \simeq \mcalD^{\comp{I}}(R)\) in \Cref{DerivedQCohEquivAffine}, and the isomorphism in the proof of \Cref{DerivedCompletionFormalSchemes}, we can reduce to show that the derived \(I\)-completion functor \(\dcomp{I}{-} \colon \mcalD(R) \to \mcalD^{\comp{I}}(R)\) preserves colimits and is symmetric monoidal. This is a standard result on derived completion.
\end{proof}



\end{document}